\documentclass[12pt, a4paper]{report}

\usepackage[utf8]{inputenc}
\usepackage{xcolor} 
\definecolor{darkgreen}{rgb}{0.0, 0.5, 0.0}

\usepackage{mathtools}          
\usepackage{amssymb, amsfonts, amsthm}
\usepackage{enumitem}
\usepackage{dsfont}
\usepackage[numbers]{natbib}

\usepackage{graphicx}
\usepackage{caption}
\usepackage{subcaption}         
\usepackage{booktabs}           

\usepackage{hyperref}           
\hypersetup{
  colorlinks=true,
  linkcolor=darkgreen,
  citecolor=darkgreen,
  urlcolor=darkgreen
}
\usepackage{cleveref}           

\newcommand{\R}{{\mathbb{R}}}
\newcommand{\Z}{{\mathbb{Z}}}
\newcommand{\Q}{{\mathbb{Q}}}
\newcommand{\N}{{\mathbb{N}}}
\newcommand{\C}{{\mathbb{C}}}

\newcommand{\GL}{{\mathrm{GL}}}
\newcommand{\SL}{{\mathrm{SL}}}

\newcommand{\map}{{\mathrm{Map}}}
\newcommand{\pmap}{{\mathrm{PMap}}}
\newcommand{\homeo}{{\mathrm{Homeo}}}
\newcommand{\p}{\partial}
\newcommand{\sym}{\mathrm{Sym}}
\DeclarePairedDelimiter{\set}{\{}{\}}

\newtheorem{theorem}{Theorem}[section]
\newtheorem{lemma}[theorem]{Lemma}
\newtheorem{proposition}[theorem]{Proposition}
\newtheorem{corollary}[theorem]{Corollary}

\theoremstyle{definition}
\newtheorem{definition}[theorem]{Definition}

\theoremstyle{remark}
\newtheorem{remark}[theorem]{Remark}
\newtheorem{example}[theorem]{Example}

\usepackage{geometry}

\usepackage{tikz-cd}
\usepackage{float}
\usepackage{adjustbox}

\title{Structure and Topological Generation of Big Mapping Class Groups}
\author{Celal Can Bellek}
\date{December 2025}

\begin{document}

\maketitle

\begin{abstract}
    Big mapping class groups are the mapping class groups of infinite-type surfaces, that is, surfaces whose fundamental groups are not finitely generated. While mapping class groups of finite-type surfaces have been extensively studied, the theory of big mapping class groups is a recent and rapidly developing area of research. This thesis provides a systematic introduction to the structure and topological generation of big mapping class groups, emphasizing the key differences from the classical finite-type case. It presents a clear exposition of foundational results concerning their topological and algebraic structure, including known results on finite topological generation for a certain family of infinite-type surfaces.
\end{abstract}

\tableofcontents

\chapter{Introduction}

At the intersection of low-dimensional topology and geometric group theory lies the study of the mapping class groups of surfaces, denoted $\map(S)$. They are defined as the group of orientation-preserving self-homeomorphisms that restrict to the identity on the boundary of $S$, under the equivalence relation of isotopy. Elements of the mapping class group are called mapping classes. 

For most of the last century, research on mapping class groups focused almost exclusively on the mapping class groups of finite-type surfaces, those whose fundamental group is finitely generated, beginning with the foundational work of Dehn, Nielsen, and Baer. Finite-type surfaces are the compact surfaces, possibly with finitely many punctures. The main result of the classical theory is that the mapping class groups of finite-type surfaces are generated by finitely many mapping classes, called Dehn twists, which are the isotopy classes of certain homeomorphisms defined geometrically on annular neighborhoods of simple closed curves.   

In recent years, the mapping class groups of infinite-type surfaces: those whose fundamental group is infinitely generated, have attracted increasing attention from researchers, due to their rich and complex structure. Notable examples of infinite-type surfaces include the Loch Ness Monster surface, which has infinitely many genera, and the Cantor Tree surface, which has a Cantor set of ends. Their mapping class groups are uncountable, earning them the name "big mapping class groups".

The works of Calegari, Patel, Vlamis, Aramayona, Mann, Rafi, and others have shown that big mapping class groups have fundamentally different properties from their finite-type counterparts. The main goal of this thesis is to provide a detailed introduction to the theory of big mapping class groups, with a particular focus on their topological and algebraic structure as well as topological generating sets, those generating sets that generate a countable dense subgroup. We develop the theory from the ground up, starting with the preliminaries of classical mapping class groups, proceeding to the classification of infinite-type surfaces, and then presenting the main results concerning the topological structure and generation of big mapping class groups.

\section{Outline of the Thesis}

\begin{itemize}
    \item \textbf{Chapter 2} provides a review of the classical theory of finite-type surfaces and their mapping class groups. We introduce surfaces and their classification, define the classical mapping class group, and focus on key properties exhibited by Dehn twists, including the fact that the mapping class group is finitely generated by Dehn twists.

    \item \textbf{Chapter 3} is dedicated to the classification of infinite-type surfaces. We introduce notable examples, define the crucial concept of the space of ends, and present the complete classification of infinite-type surfaces due to Richards, which relies on this space.

    \item \textbf{Chapter 4} formally defines big mapping class groups and explores their fundamental properties as topological groups. We show that they are Polish groups, exhibit the fundamental properties of the Baire space, and are therefore homeomorphic to it.

    \item \textbf{Chapter 5} addresses the main topic of this thesis: the topological generation of big mapping class groups. We present a new type of generator, and explain an important theorem concerning the algebraic structure of the pure mapping class group.
    
    \item \textbf{Chapter 6} demonstrates that a certain family of infinite-type surfaces, those with finitely many ends accumulated by genus, can be topologically generated by finitely many elements, including, and in particular, involutions. 
\end{itemize}
\chapter{Preliminaries}

Before delving into the theory of big mapping class groups, we must first understand the classical definitions and results about finite-type surfaces and mapping class groups. This chapter lays out the required preliminary definitions and results concerning finite-type surfaces and their mapping class groups.

\section{Surfaces and their Classification}\label{sec:surfaces}

\begin{definition}
    A \emph{manifold} is a topological space that is second countable, Hausdorff, and locally Euclidean. A $2$-manifold is called a \emph{surface}. A surface is of \emph{finite-type} if its fundamental group is finitely generated, and of \emph{infinite-type} otherwise. Throughout this thesis, all surfaces are assumed to be connected and orientable, and in addition throughout this chapter, they are assumed to be of finite-type.
\end{definition}

In addition, we will only consider surfaces with compact boundary (if any). The next theorem is a classical result regarding the classification of finite-type surfaces. For a proof, we refer the reader to~\cite{leetopological}.

\begin{theorem}[Classification of Finite-Type Surfaces]\label{thm:finiteclassification}
    Any surface $S$ is determined, up to homeomorphism, by a $3$-tuple $(g, b, n)$, where $g$ is the genus, $b$ is the number of boundary components, and $n$ is the number of punctures. We denote such a surface by $S_{g,n}^{b}$.
\end{theorem}

By Theorem~\ref{thm:finiteclassification}, it is easy to see that the cardinality of the set of all surfaces, up to homeomorphism, is the same as the cardinality of the $3$-tuple $(b,g,n)$. Since 
\[
|(b,g,n)| = |\mathbb{N} \times \mathbb{N} \times \mathbb{N}| = |\mathbb{N}| = \aleph_0,
\]
we see that there are countably infinitely many homeomorphism classes of surfaces.

\section{The Mapping Class Group}

We denote by $\homeo(S, \p S)$ the self-homeomorphisms of a surface $S$ that restrict to the identity function on the boundary $\p S$. Being a function space, it is natural to endow it with the \emph{compact-open topology}, generated by the subbasis consisting of sets 
\[
B(K,U) = \set{ f\in \homeo(S,\p S)\, | \, f(K)\subset U}
\] where $K\subseteq S$ is compact and $U \subseteq S$ is open. That is, the open sets are arbitrary unions of finite intersections of the subbasic sets $B(K,U)$.

The subspace of $\homeo(S,\p S)$ that consists of all the orientation-preserving homeomorphism is denoted by $\homeo^+(S,\p S)$, and the subspace consisting of homeomorphisms isotopic to the identity is denoted by $\homeo_0(S,\p S)$. Note that since the elements of $\homeo(S,\p S)$ fix the boundary pointwise, we need to consider isotopies relative to the boundary in the presence of boundaries. 

\begin{definition}\label{def:mapgr}
    Let $S$ be a surface. The \emph{mapping class group} of $S$, denoted by $\map(S)$, is the group of isotopy classes of elements of $\homeo^+(S,\p S)$. That is,
    \begin{align*}
         \map(S) = \pi_0(\homeo^+(S, \p S)) &= \homeo^+(S, \p S) / \sim \\
         &=\homeo^+(S, \p S) / \homeo_0(S,\p S) 
    \end{align*}
\end{definition}

Elements of $\mathrm{Map}(S)$ are called \emph{mapping classes}. When the context is clear, we will abuse notation by identifying a homeomorphism $f\in \homeo^+(S,\p S)$ with its corresponding mapping class $[f] \in \map(S)$.

The functor $\pi_0$ assigns to a topological space $X$ the set of all isotopy classes of maps $f : S^0 \cong \{0,1\} \longrightarrow X$, that is, the path-connected components of $X$. Unlike the fundamental group $\pi_1$, this set does not have a canonical group structure. However, since $\homeo^+(S, \p S)$ is a topological group, see Definition~\ref{def:topgroup}, $\map(S)$ naturally inherits both a group structure and a topology, which justifies its name as the mapping class group. 

\begin{definition}\label{def:topgroup}
    Let $G$ be a group. Let $\tau$ be a topology on $G$ such that the maps $\mu: G\times G \longrightarrow G$ and $\iota: G\longrightarrow G$ that are defined by, 
    \[
      \mu(g_1,g_2) = g_1g_2 \hspace{2em} \iota(g) = g^{-1}
    \]
    are continuous. Then $G$ endowed with $\tau$ is called a \emph{topological group}.
\end{definition}

\begin{remark}\label{rmk:cont}
    Note that if the maps $\mu$ and $\iota$ are continuous with respect to $\tau$, then the map $\bar\mu: G \times G \longrightarrow G$ defined by \[
    \bar\mu(g_1,g_2) = g_1g_2^{-1}
    \] is also continuous. Similarly, if $\bar\mu$ is a continuous map, then $\mu$ and $\iota$ are necessarily continuous. This means we can as well define topological groups by requiring only that $\bar\mu$ to be continuous.
\end{remark}

From Definition~\ref{def:topgroup} it is clear that topological groups are precisely those topological spaces that admit continuous group operations, or equivalently, those groups that can be endowed with a topology that makes the group operations continuous. The most basic example of a topological group is $\R$ with its usual addition. A less-obvious example of a topological group is the general linear group $\GL(2,\R)$ considered as a subspace of $\R^4$ with the subspace topology, where elements of $\GL(2,\R)$ are identified with $4$-tuples of real numbers. While not included in our definition, topological groups are usually required to be Hausdorff. The next theorem provides a necessary and sufficient condition for a topological group to be Hausdorff.

\begin{lemma}\label{lem:grhaus}
    Let $G$ be a topological group. Then $G$ is Hausdorff if and only if the set $\set{1_G}$ is closed.  
\end{lemma}
\begin{proof}
    The forwards direction is trivial, since in any Hausdorff space, singleton sets are closed. 
    
    For the backwards direction, assume that $\set{1_G}$ is closed. Recall the classical result from basic topology that a topological space $X$ is Hausdorff if and only if the diagonal 
    \[\Delta_X = \set{(x,x)\in X \times X \, \mid \,x\in X}\] 
    is closed in the product topology. By Remark~\ref{rmk:cont}, the map 
    \[
\mu \colon G \times G \longrightarrow G, \quad \mu(x,y) = xy^{-1}
\]
    is continuous. Note that
    \[
    \mu^{-1}(1_G) = \Delta_G,
    \]
    because $xy^{-1}=\mathds{1}_G$ implies $x = y$. By assumption, $\set{\mathds{1}_G}$ is closed, so  the continuity of $\mu$ implies that $\Delta_G$ is closed and we are done.
\end{proof}

The most important example of a topological group in the study of mapping class groups, other than the mapping class group itself, is the space $\homeo(S,\p S)$. It is known that the space $\homeo(S,\p S)$ is a topological group. Moreover, the space $\homeo(X)$ is a topological group when the topological space $X$ is locally compact, Hausdorff and locally connected~\cite{arens}.

\begin{lemma}
    The $\homeo_0(S, \p)$ is a normal subgroup of $\homeo(S,\p S)$.
\end{lemma}
\begin{proof}
    We shall show that $gfg^{-1}\in\homeo_0(S, \p)$ for any $f \in\homeo_0(S, \p) $ and $g\in \homeo(S, \p)$. Since $f$ is isotopic to the identity, we have an isotopy
    \[
    H \colon S\times [0,1] \longrightarrow S \quad \text{where} \quad H(x,0) = f(x) \text{ and } H(x,1) = x.
    \]
    Define 
    \[
    G \colon S \times [0,1] \longrightarrow S \quad \text{by} \quad G(x,t) = (g\circ H\circ g^{-1})(x,t).
    \]
    It is clear that $G(x,0)$ is the homeomorphism $g\circ f\circ g^{-1}$ and $G(x,1)$ is equal to the identity, which means that $\homeo_0(S, \p S) \trianglelefteq \homeo(S,\p S)$.
\end{proof}

Recall from Definition~\ref{def:mapgr} that the mapping class group is defined as the quotient 
\[
\homeo^+(S, \p S) / \homeo_0(S,\p S).
\]
Since $\homeo_0(S,\p S)$ is a normal subgroup, the mapping class group inherits its topology and group structure from $\homeo^+(S,\p S)$, becoming a topological group. Moreover, the quotient map 
\[
\pi\colon \homeo^+(S,\p S) \longrightarrow \map(S)
\]
is open and thus $\map(S)$ has a basis 
\[
\mathcal{B} = \set{\pi(B(K,U))\,\mid\, B(K,U) \text{ is a subbasic element of $\homeo(S,\p S)$}}.
\]

Although mapping class groups are topological groups, their topologies are rather uninteresting for finite-type surfaces, as they are discrete.

\begin{definition}
    A point $x$ in a topological space $X$ is called \emph{isolated} if there exists a neighborhood $U\subset X$ of $x$ such that $U \cap y = \emptyset$ for all $y\in X$.
\end{definition}

A nice property of topological groups is that translation by an element of the group is a homeomorphism, that is, any neighborhood $U\subset G$ of a point $g\in G$ can be written as $U= g\,U'$, where $U'\subset G$ is a neighborhood of the identity. We use this fact in the next theorem.

\begin{lemma}\label{lem:discrete}
    A topological group $G$ is discrete if and only if the point $\{\mathds{1}_G\}$ is isolated.
\end{lemma}
\begin{proof}
    We begin with the forwards direction. Suppose $G$ is discrete. Then, any element $g\in G$ has a neighborhood $U\subset G$ such that $U \cap g' = \emptyset$ for all $g'\in G$. Therefore, $1_G$ has a neighborhood that is disjoint from every element in $G$ except itself, that is, it is isolated.
    
    For the backwards direction, suppose $1_G$ is isolated, so that it has a neighborhood $U\subset G$ with $U \cap g' = \emptyset$ for all $g'\in G$. It follows that any $g\in G$ has a neighborhood $gU$ with $gU \cap g' = \emptyset$ for all $g' \in G$ and hence, $g$ is isolated. Since $g$ was arbitrary, $G$ is discrete.
\end{proof}

Checking for convergence from the definition of compact-open topology is quite cumbersome, but fortunately, there is a rather useful characterization of convergence in the compact-open topology.

\begin{proposition}[{\cite[Theorem 7.11]{kelley}}]\label{prop:compactopenconvergence}
    Let $S$ be any surface, $f\in \homeo(S,\p S)$ and $\set{f_n}_{n\in\N}$ be a sequence with $f_n \in \homeo(S,\p S)$. Then $f_n\rightarrow f$ in the compact-open topology if and only if for every compact subset $K \subseteq S$, $f_n|_K\rightarrow f|_K$ uniformly. We say in this case that $f_n$ \emph{converges compactly} to $f$. 
\end{proposition}

With this characterization, we can show that $\map(S)$ is discrete.

\begin{proposition}\label{prop:discretegroup}
    The mapping class group $\map(S)$ is a discrete topological group.
\end{proposition}
\begin{proof}
    By Lemma~\ref{lem:discrete}, it suffices to show that the identity in $\map(S)$ is isolated. Assume for a contradiction that it is not. Then for every open neighborhood $U\subset \map(S)$ of the identity, $U$ contains a nontrivial mapping class. In particular, there exists a sequence of mapping classes $\set{\hat{f_n}}_{n\in \N}$ converging to the identity, together with a sequence of representatives  $\set{f_n}$ converging to some $f\in \homeo_0(S,\p S)$.
    
   By Proposition~\ref{prop:compactopenconvergence}, $f_n\rightarrow f$ uniformly on every compact subset of $S$ and by the Alexander method, see Proposition~\ref{prop:alex}, there exists a finite collection of simple closed curves and essential arcs
    \[
    \Gamma = \set{\gamma_1,\dots,\gamma_k}
    \]
    such that any homeomorphism fixing elements of $\Gamma$ pointwise is isotopic to the identity. 
    
    Since $\Gamma$ is compact, $f_n$ converges uniformly to $f$ on $\Gamma$. This means for sufficiently large $n\in \N$, $f_n(\gamma_i)$ is arbitrarily close to $\gamma_i$ for all $i=1,\dots,k$. Therefore, for such $n$, we can assume, up to isotopy, that $f_n$ fixes the elements of $\Gamma$ pointwise. 
    
    By the Alexander Method, for sufficiently large $n$, $\hat{f_n}$ is isotopic to the identity, contradicting our assumption that they are all nontrivial. This shows that the identity is isolated and $\map(S)$ is discrete.
\end{proof}

We list below the mapping class groups of some important finite-type surfaces:
    \begin{itemize}
        \item The mapping class group of the disk is trivial.
        \item The mapping class group of the once punctured disk is trivial.
        \item The mapping class group of the $2$-sphere is trivial.
        \item The mapping class group of an annulus is $\Z$, generated by the Dehn twist shown in Figure~\ref{fig:twist}.
        \item The mapping class group of the $2$-torus is the special linear group $\SL(2,\Z)$, generated by the two Dehn twists about the curves shown in Figure~\ref{fig:twisttorus}.
    \end{itemize}

\section{The Alexander Method}

One of the most important tools in the study of mapping class groups is the \emph{Alexander method}, which gives a condition to determine whether a mapping class is the identity or not. 

We say a collection of curves \emph{fills} $S$ if the complement of the union of them is a disjoint union of disks and once punctured disks. Since mapping class groups of the disk and the once punctured disk are trivial, a mapping class is trivial if it fixes such a collection. The Alexander method makes this more precise. This is a good time to give the formal definitions for closed curves and arcs. We begin with closed curves.

\begin{definition}
    A \emph{closed curve} in $S$ is a map $\gamma \colon S^1 \rightarrow S$ where $S^1$ denotes the circle. A closed curve is \emph{simple} if it is a homeomorphism onto its image. A closed curve is \emph{essential} if it is not isotopic to a point, a puncture or a boundary component. A closed curve is \emph{peripheral} if it is isotopic to some boundary component. We identify a closed curve with its image in $S$.
\end{definition}

We define arcs similarly.

\begin{definition}
    An \emph{arc} in $S$ is a map $\alpha\colon [0,1]\rightarrow S$. An arc is \emph{simple} if it is a homeomorphism onto its image in its interior. An arc is \emph{proper} if its endpoints are on punctures or boundary components. An arc is \emph{essential} if it is not isotopic to a puncture or a boundary component.
\end{definition}

\begin{proposition}[{\cite[Proposition 2.8]{primer}}]\label{prop:alex}
Let $S$ be a finite type surface, and let $\phi \in \homeo^+(S, \p S)$. There exists a collection
\[
\gamma_1,\dots,\gamma_n
\]
of essential simple closed curves and simple proper arcs in $S$ with the following properties:

\begin{enumerate}
    \item The $\gamma_i$ are pairwise in minimal position.
    \item The $\gamma_i$ are pairwise nonisotopic.
    \item For distinct $i,j,k$, at least one of $\gamma_i\cap \gamma_j$, $\gamma_i\cap \gamma_k$, or $\gamma_j\cap \gamma_k$ is empty.
\end{enumerate}

\begin{enumerate}[label=(\arabic*)]
    \item If there is a permutation $\sigma$ of $\set{1,\dots,n}$ so that $\phi(\gamma_i)$ is isotopic to $\gamma_{\sigma(i)}$ relative to $\p S$ for each $i$, then $\phi(\cup\gamma_i)$ is isotopic to $\cup\gamma_i$ relative to $\p S$.
    If we regard $\cup\gamma_i$ as a (possibly disconnected) graph $\Gamma$ in $S$, with vertices at the intersection points and at the endpoints of arcs, then the composition of $\phi$ with this isotopy gives an automorphism $\phi_*$ of $\Gamma$.
    \item Suppose now that $\set{\gamma_i}$ fills $S$. If $\phi_*$ fixes each vertex and each edge of $\Gamma$, with orientations, then $\phi$ is isotopic to the identity. Otherwise, $\phi$ has a nontrivial power that is isotopic to the identity.
\end{enumerate}

\end{proposition}

\section{The Pure Mapping Class Group}\label{sec:puremap}

Let $S$ be a finite-type surface with $n$ punctures. Although they keep the set of punctures fixed, each mapping class act on the set of punctures by some permutation. Therefore, we can define a surjective map 
\[
\pi \colon \map(S) \longrightarrow \sym_n
\]
given by the action of a mapping class on the set of punctures, where $\sym_n$ is the symmetric group on $n$ letters. This map is well-defined, since isotopic homeomorphisms induce the same action on the set of punctures and it is clear that this is a group homomorphism. The kernel of this group homomorphism is an important normal subgroup of $\map(S)$.

\begin{definition}
The \emph{pure mapping class group}, denoted $\pmap(S)$, is the the kernel of the group homomorphism $\pi \colon \map(S) \longrightarrow \sym_n$. Equivalently, it is the subgroup of $\map(S)$ consisting of those mapping classes that fix the punctures of $S$ pointwise. An element of $\pmap(S)$ is called a \emph{pure mapping class}.
\end{definition}

The relation between $\map(S)$ and $\pmap(S)$ is encapsulated nicely by the short exact sequence
\[
0 \longrightarrow \pmap(S) \longrightarrow \map(S) \longrightarrow \sym_n \to 0 
\] 
where the first map is the inclusion of a mapping class in $\pmap(S)$ into $\map(S)$.

\begin{remark}
    When $S$ has no punctures, the pure mapping class group coincides with the entire mapping class group.
\end{remark}

The above short exact sequence tells that in order to generate the entire mapping class group in the presence of punctures, mapping classes whose images under $\pi$ generate $\sym_n$ need to be added to the generating set of $\pmap(S)$. An example for such mapping classes are \emph{half-twists}. However, punctured surfaces are outside the scope of this thesis and thus we skip half-twists. 

\section{Dehn Twists}

The most important mapping classes in the study of mapping class groups are undoubtedly the \emph{Dehn twists}. We follow~\cite{primer} for the construction of Dehn twists.

\begin{definition}
    Consider the annulus $A= S^1\times [0,1]$, where $S^1$ denotes the circle and embed it in $\R^2$ by the map $(\theta,t)\mapsto(\theta,t+1)$. The map $T\colon A\rightarrow A$ defined by
    \[
    T(\theta,t) = (\theta + 2\pi t, t)
    \]
    is called the \emph{twist map} of $A$.
\end{definition}

\begin{remark}
    We defined the twist map as a left twist. Similarly, a right twist is given by $(\theta,t) \mapsto (\theta - 2\pi t, t)$. Choosing between a right twist and left twist is a matter of convention.
\end{remark}

\begin{figure}[htbp]
  \centering
  \includegraphics[width=0.5\textwidth]{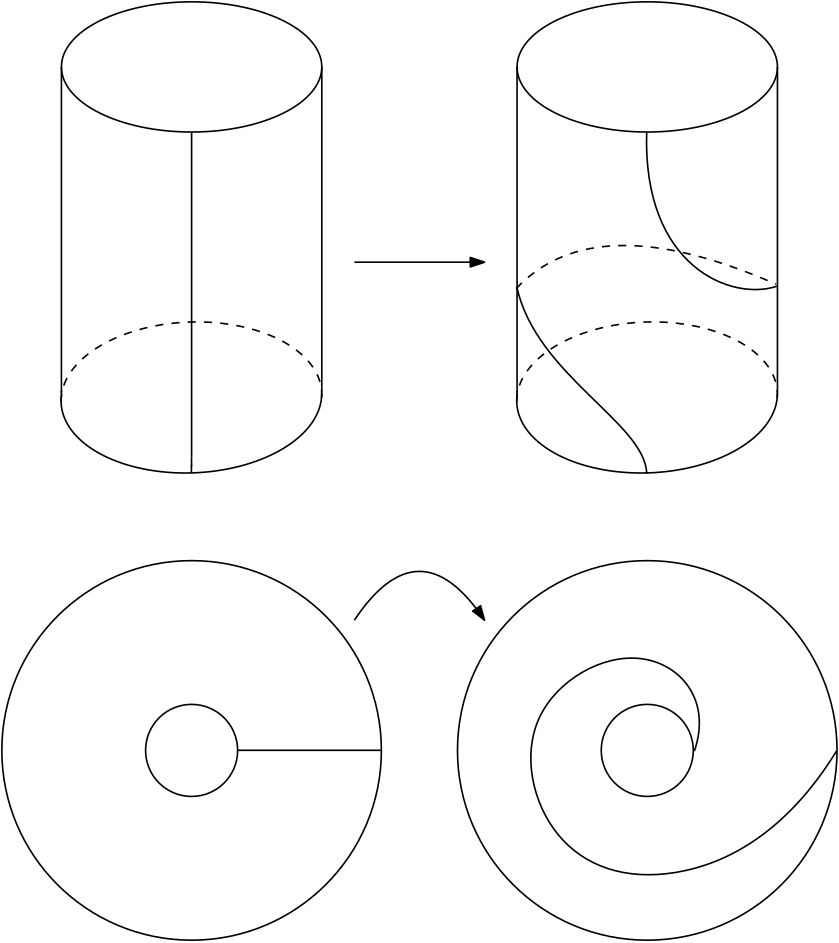}
  \caption{Two views of the twist map on $A$ }
  \label{fig:twist}
\end{figure}

We shall define a Dehn twist to be the twist map applied on an annular neighborhood of a simple closed curve. 

\begin{definition}
    Let $\alpha$ be a simple closed curve on a surface $S$. Let $N$ be an annular neighborhood of $\alpha$ and $\phi\colon A \rightarrow N$ to be an orientation-preserving embedding of $N$ into $A$. A \emph{Dehn twist} about $\alpha$, denoted by $T_a$ is defined by
    \[
    T_\alpha(x) = \begin{cases}
        \phi \circ T \circ \phi^{-1}(x) &  \text{if } x \in N, \\
        x & \text{if }x\in S\setminus N.
    \end{cases}
    \]
\end{definition}

\begin{remark}
    A Dehn twist is obtained by removing an annular neighborhood of $\alpha$, applying the twist map and gluing it back.
\end{remark}

Let $a$ be an isotopy class of a simple closed curve in $S$ and $\alpha$ be a representative of $a$. We write $T_a$ for the mapping class of $T_\alpha$. While the homeomorphism $T_\alpha$ depends on the representative $\alpha$ and the annular neighborhood $N$, the mapping class $[T_\alpha]$, denoted by $T_a$, does not, and depends only on the isotopy class of simple closed curves. This follows from the fact that any two annular neighborhoods are homeomorphic and that representatives of isotopy classes of simple closed curves are isotopic, which implies that the corresponding Dehn twists are isotopic. 

\begin{figure}[htbp]
  \centering
  \includegraphics[width=\textwidth]{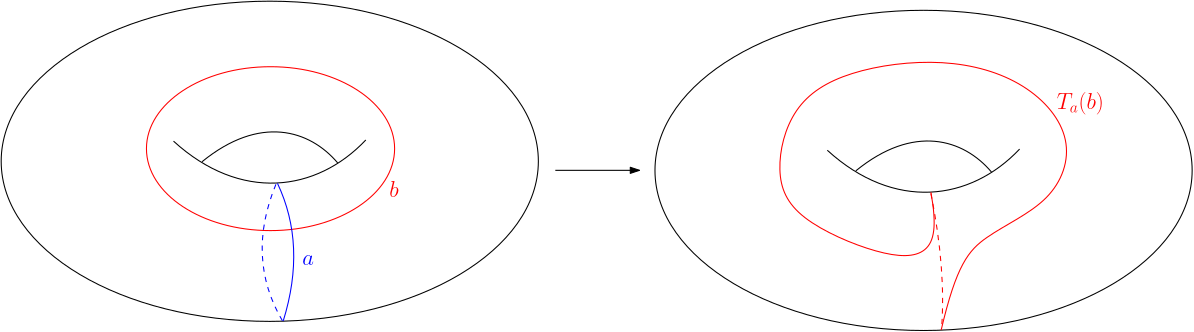}
  \caption{A Dehn twist on the torus}
  \label{fig:twisttorus}
\end{figure}

\begin{remark}
    Dehn twists are pure mapping classes, since they fix the set of punctures pointwise, i.e., they act trivially on the set of punctures.
\end{remark}

\begin{proposition}\label{prop: dehnprop}
    Let $a,b$ be two isotopy classes of simple closed curves and $f\in \mathrm{Map}(S)$. Then, Dehn twists about $a$ and $b$ satisfy the following properties:
    \begin{enumerate}[label=(\roman*)]
        \item $i(a,b) = 0$ $\Longleftrightarrow$ $T_a(b) = b$ and $T_b(a)=a$ $\Longleftrightarrow$ $T_aT_b = T_bT_a$.
        \item $T_a = T_b \Longleftrightarrow a=b$
        \item $T_{f(a)}=fT_af^{-1}$. (Conjugation property)
        \item $f$ commutes with $T_a$ $\Longleftrightarrow$ $f(a)=a$.
    \end{enumerate}
\end{proposition}

For the next property of Dehn twists, we need to define the \emph{geometric intersection number} of two simple closed curves.

\begin{definition}
    Let $a$ and $b$ be isotopy classes of simple closed curves in $S$. The \emph{geometric intersection number} of $a$ and $b$ is 
    \[
    i(a,b) = \set{\min|\alpha\cap \beta| \, \mid \, [\alpha]=a \text{ and }[\beta]=b}.
    \]
\end{definition}

Note that the geometric intersection number of two simple closed curves is $0$ if and only if they are disjoint up to isotopy.

\begin{proposition}[{\cite[Proposition 3.2]{primer}}]\label{prop:twistintersection}
    Let $a$ and $b$ be arbitrary isotopy classes of essential simple closed curves on a surface and let $k\in \mathbb{Z}$. Then
    \[
    i(T_a^k(b),b) = |k|i(a,b)^2.
    \]
\end{proposition}

One of the most important relations between Dehn twists is the \emph{lantern relation}. A compact subsurface is \emph{essential} if its complement has no component homeomorphic to a disk or an annulus.

\begin{proposition}[{\cite[Section 5.1]{primer}}]\label{prop:lantern}
    Let $a,b,c,d,x,y$ and $z$ be simple closed curves as shown in Figure~\ref{fig:lantern} on the surface $S_0^4$. Then the following relation holds in $\map(S_0^4)$:
    \[
    T_aT_bT_cT_d = T_xT_yT_z
    \]
    
\begin{figure}[htbp]
  \centering
  \includegraphics[width=0.5\textwidth]{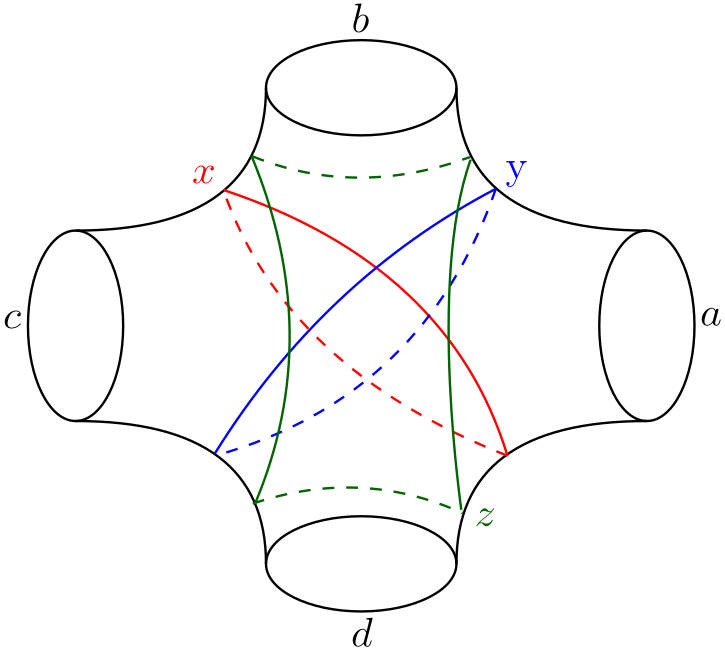}
  \caption{The simple closed curves $a,b,c,d,x,y,z$ on $S_0^4$}
  \label{fig:lantern}
\end{figure}
    
    Moreover, any surface $S$ with $\chi(S) \leq -2$, where $\chi(S)$ denotes the Euler characteristic of $S$, contains an essential subsurface $S'$ homeomorphic to $S_0^4$, so the lantern relation holds for all surfaces with at least two genus. The embedded copy of $S_0^4$ is called an \emph{embedded lantern}.
\end{proposition}

Another relation that is of importance to us is the \emph{braid relation}.

\begin{proposition}
    Let $a$ and $b$ be two isotopy classes of simple closed curves with $i(a,b) = 1$. Then
    \[
    T_aT_bT_a = T_bT_aT_b.
    \]
\end{proposition}

The braid relation implies that 
\[
(T_aT_b)T_a(T_b^{-1}T_a^{-1}) = T_b.
\]
By the conjugation property,
\[
T_{T_aT_b(a)} = T_b,
\]
and thus by Property $(iv)$ of Propositon~\ref{prop: dehnprop} implies that
\[
T_aT_b(a) =b.
\]
By a similar argument we also have that
\[
T_bT_a(b) = a.
\]
We shall call this property the braid relation as well.

\section{Generating the Pure Mapping Class Group of Finite-Type Surfaces}

First, Dehn proved that the pure mapping class of finite-type surfaces are generated by $2g(g-1)$ Dehn twists~\cite{dehn}. Lickorish then reduced the required number of Dehn twist generators to $3g-1$~\cite{lickorish}, which was reduced even further by Humphries to $2g+1$, and he showed that this is the minimal number of generators for $g\geq 2$~\cite{humphries}.

\begin{theorem}\label{thm:humphries}
    For $g\geq2$, the mapping class group $\map(S_g)$ is generated by $2g+1$ Dehn twists about nonseparating simple closed curves, called the Humphries generators, shown in Figure~\ref{fig:humphries}.

    \begin{figure}[htbp]
      \centering
      \includegraphics[width=0.8\textwidth]{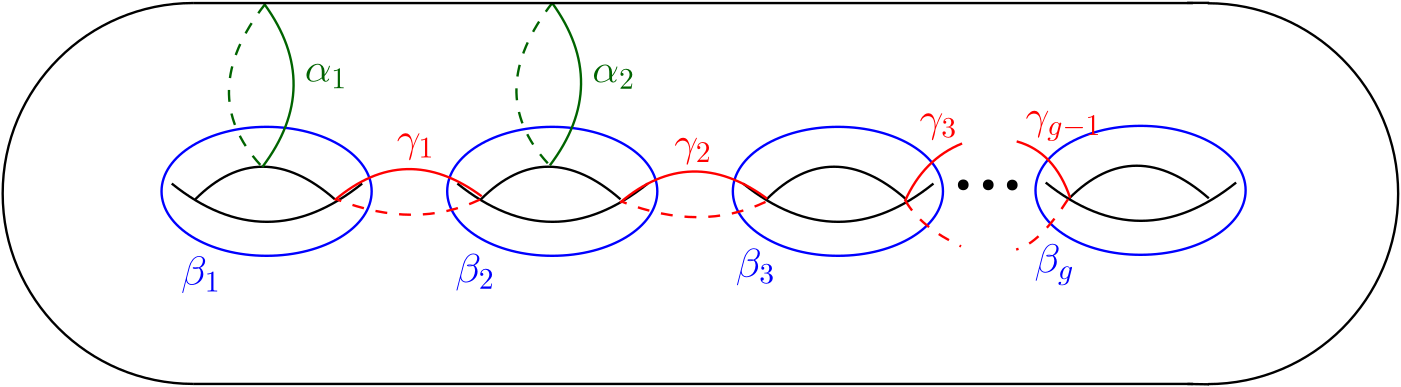}
      \caption{The Humphries generators}
      \label{fig:humphries}
     \end{figure}
\end{theorem}

Theorem~\ref{thm:humphries} can be proved by induction on $g$, by showing that finitely many Dehn twists generate $\pmap(S_g)$ and using relations between them to simplify the generator set, which is nontrivial and we omit it. 

In the presence of boundary components and punctures, the Humphries generators are not enough to generate the pure mapping class group. In that case, $b+n-1$ many Dehn twists about nonseparating curves need to be added to the Humphries generators, where $b$ is the number of boundary components and $n$ is the number of punctures.

\begin{theorem}[{\cite[Corollary 4.16]{primer}}]\label{thm:bdarygen}
    For $g\geq 2$, pure mapping class group $\pmap(S^b_{g,n})$ is generated by $2g+b+n$ Dehn twists about nonseparating simple closed curves, shown in Figure~\ref{fig:bdrygen}.
    \begin{figure}[htbp]
      \centering
      \includegraphics[width=0.8\textwidth]{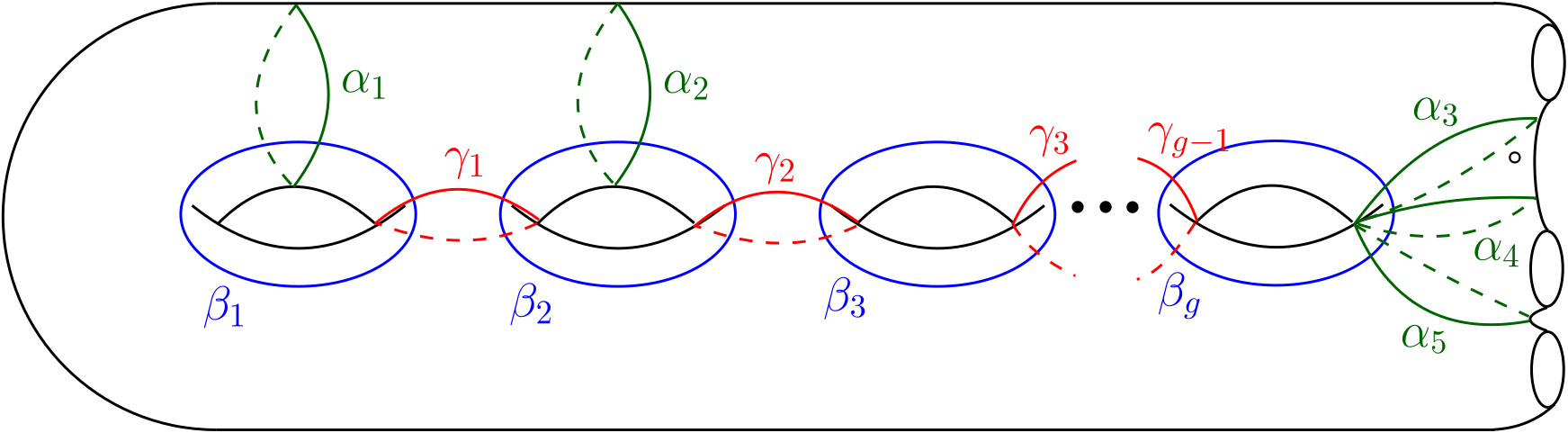}
      \caption{The generators for finite-type surfaces with boundary components and punctures}
      \label{fig:bdrygen}
     \end{figure}
\end{theorem}

Theorem~\ref{thm:humphries} and Theorem~\ref{thm:bdarygen} stand in stark contrast to the case of infinite-type surfaces but Theorem~\ref{thm:bdarygen} is especially important for the study of a certain family of infinite-type surfaces, as we will explore in the coming chapters. 

\section{The Extended Mapping Class Group}\label{sec:extended}

The mapping class group can be extended to include orientation-reversing mapping classes. This group is aptly called, the \emph{extended mapping class group}, and is denoted by $\map^\pm(S)$. It is clear that orientation-reversing mapping classes do not form a group, since the identity homeomorphism is orientation-preserving. 

Consider the map
\[
p\colon \map^\pm(S) \longrightarrow \Z_2
\]
that sends orientation-preserving mapping classes to $0$ and orientation-reversing mapping classes to $1$. It can be checked that $p$ is a homeomorphism and we get the short exact sequence
\[
0 \longrightarrow \map(S) \longrightarrow \map^\pm(S) \longrightarrow \Z_2 \longrightarrow 0.
\]
The kernel of $p$ is the mapping class group and therefore this short exact sequence splits:
\[
\map^\pm(S) = \map(S) \rtimes \Z_2.
\]
Since $\map(S) \trianglelefteq \map^\pm (S)$, and $\map^\pm(S)/\map(S) \cong \Z_2$, the mapping class group $\map(S) $is an index two subgroup of the extended mapping class group $\map^\pm(S)$ and
\[
\map^\pm(S) = \map(S) \sqcup f\map(S),
\]
where $f$ is any orientation-reversing mapping class. It immediately follows that the cardinality of orientation-preserving and orientation-reversing mapping classes are equal.
\chapter{The Classification of Infinite-Type Surfaces}

To study the structure of big mapping class groups, it is important to understand infinite-type surfaces, those surfaces whose fundamental groups are not finitely generated. Unlike their finite-type counterparts, which are compact (with the exception of surfaces with finitely many punctures), infinite-type surfaces are always non-compact. Intuitively, this is because they can have infinitely many directions that expand to infinity. We make this formal in Section~\ref{sec:spaceofends}, by introducing the concept of the \emph{space of ends}.

\section{Infinite-Type Surfaces}

Before we get into the classification of infinite-type surfaces, we give some examples of important infinite-type surfaces. These surfaces have infinitely many genera and/or different \emph{ends}, see Definition~\ref{def:ends}.

\begin{example}\label{ex:infsurf}
    Let us give examples of important infinite-type surfaces:
    \begin{enumerate}[label=(\alph*)]
    \item \textbf{The Loch Ness Monster surface}: The surface with infinitely many genera and one end. This surface is the limit of surfaces with finitely many genera, no punctures and one boundary component, i.e., every finite genus surface with no punctures and one boundary component can be embedded into this space.
    \item \textbf{The Jacob's Ladder surface}: The surface with infinitely many genera and two ends. This surface is the limit of surfaces with finitely many genera, no punctures and two boundary components, i.e., every finite genus surface with no punctures and two boundary components can be embedded into this space.
    \item \textbf{The Cantor Tree surface}: The surface obtained by removing a Cantor set from a sphere. It has uncountably many ends, all of which are accumulation points, but has no genus.
    \item \textbf{The Blooming Cantor Tree surface}: The surface whose space of ends is the Cantor set, all of which are \emph{accumulated by genus}, see Definition~\ref{def:accubygenus}.
    \item \textbf{The Flute surface}: The surface $\C \setminus \Z$. This surface has infinitely many ends, one of which is an accumulation point. We can also consider this surface as the sphere with north pole and countably infinitely many points that are accumulating towards the north pole, removed.
    \end{enumerate}
\end{example}

\begin{figure}[htbp]
    \centering

    \begin{subfigure}[b]{0.3\textwidth}
        \centering
        \adjustbox{valign=m}{\includegraphics[width=\linewidth, height=4cm, keepaspectratio]{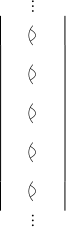}}
        \caption{The Jacob's Ladder}
    \end{subfigure}
    \hfill
    \begin{subfigure}[b]{0.3\textwidth}
        \centering
        \adjustbox{valign=m}{\includegraphics[width=\linewidth, height=4cm, keepaspectratio]{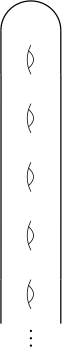}}
        \caption{The Loch Ness Monster}
    \end{subfigure}
    \hfill
    \begin{subfigure}[b]{0.3\textwidth}
        \centering
        \adjustbox{valign=m}{\includegraphics[width=\linewidth, height=4cm, keepaspectratio]{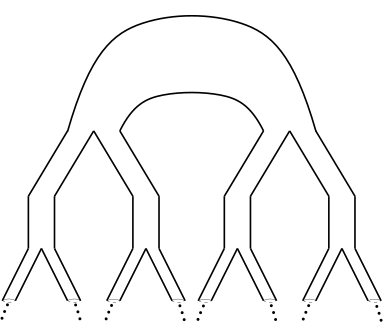}}
        \caption{The Cantor Tree}
    \end{subfigure}

    \vspace{1em} 

    \begin{subfigure}[b]{0.3\textwidth}
        \centering
        \adjustbox{valign=m}{\includegraphics[width=\linewidth, height=4cm, keepaspectratio]{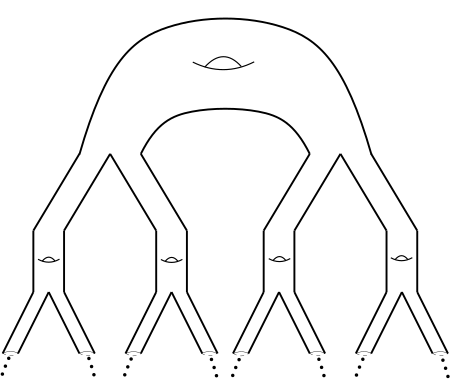}}
        \caption{The Blooming Cantor Tree}
    \end{subfigure}
    \hspace{0.1\textwidth} 
    \begin{subfigure}[b]{0.3\textwidth}
        \centering
        \adjustbox{valign=m}{\includegraphics[width=\linewidth, height=4cm, keepaspectratio]{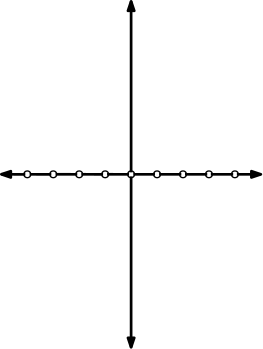}}
        \caption{The Flute}
    \end{subfigure}

    \caption{Examples of remarkable infinite-type surfaces}
    \label{fig:infsurfaces}
\end{figure}

\section{Cantor Spaces}

To understand the classification of infinite-type surfaces, it is crucial to understand \emph{Cantor spaces}. The Cantor set, denoted $C$, is traditionally defined inductively: start with the interval $[0,1]$, and remove the middle thirds of each remaining interval at each step. For this reason, $C$ is often called the middle-third Cantor Set. 
Every point $x\in [0,1]$ has a ternary expansion, that is for any $x\in [0,1]$, 
\[
x = \sum_{n=1}^{\infty} a_n3^{-n} \quad \text{where} \quad a_n \in \set{0,1,2}.
\]
Since $C$ is defined by inductively removing the middle-thirds of intervals, a point $x$ is in $C$ if and only if its ternary expansion contains only $0$s and $2$s, that is,
\[
x = \sum_{n=1}^{\infty} a_n3^{-n} \quad \text{where} \quad a_n \in \set{0,2}.
\]
This observation gives us another characterization of $C$, that is, 
\[
C = \left\{ x \in [0,1] \,\middle|\, x = \sum_{n=1}^{\infty} a_n 3^{-n},\ a_n \in \{0,2\} \right\}.
\]

We denote by $\{0,1\}^\N \cong 2^\N$ the countably infinite product of the discrete two point set $\set{0,1}$. Equivalently, it is the set of all sequences with binary terms. Endowed with the product topology, $2^\N$ is compact by the Tychonoff's Theorem [\citenum{kechrisdescriptive}, Proposition 4.1], which states that an arbitrary product of compact spaces is compact. This space is a \emph{Cantor space}, i.e., it is homeomorphic to the Cantor set. 

\begin{theorem}
    The space $2^\N$ is a Cantor space.
\end{theorem}
\begin{proof}
    Define
    \[
    f \colon 2^\N \longrightarrow C
    \]
    by 
    \[
        f(x_1,x_2,\dots) = \sum_{n=1}^{\infty} \frac{2x_n}{3^n}.
    \]
    From the ternary expansion characterization of $C$, it is clear that $f(x) \in C$ and that $f$ is bijective. Since $2^\N$ is compact and $C$ is Hausdorff, $f$ is a homeomorphism onto $C$ and we are done.
\end{proof}

The Cantor set exhibits important and interesting topological properties:
\begin{itemize}
    \item \textbf{Perfect}: A topological space is \emph{perfect} if it has no isolated points.
    \item \textbf{Compact}: A topological space is \emph{compact} if every open cover of the space has a finite subcover.
    \item \textbf{Metrizable}: A topological space is \emph{metrizable} if it can be embedded in a metric space, or equivalently, it admits a metric.
    \item \textbf{Totally disconnected}: A topological space is \emph{totally disconnected} if its only connected subsets are singletons.
    \item \textbf{Zero-dimensional}: A topological space is \emph{zero-dimensional} if it is Hausdorff and has a basis consisting of clopen (both open and closed) sets.
\end{itemize}

In fact, by a theorem of Brouwer, the Cantor set is unique up to homeomorphism among non-empty topological spaces that satisfy these properties~\cite{kechrisdescriptive}. That is, any non-empty, perfect, totally-disconnected, compact, and metrizable topological space is a Cantor space.

\section{Baire Spaces}\label{sec:baire}
Although not used in the classification of infinite-type surfaces, this is a good time to mention the \emph{Baire space} $\N^\N$, since it shares some similarities with Cantor spaces. The Baire space is the set of all sequences with natural number terms under the product topology. It is uncountable, being a countably infinite product of countably infinite sets and unlike the Cantor set, it is not compact, as the natural numbers are not compact under the discrete topology. 

The Baire space exhibits important and interesting topological properties:
\begin{itemize}
    \item \textbf{Completely metrizable}: A topological space is \emph{completely metrizable} if it admits a complete metric.
    \item \textbf{Polish}: A topological space is called \emph{Polish} if it is separable and completely metrizable.
    \item Every compact subset of $\N^\N$ has empty interior.
    \item \textbf{Zero-dimensional}
\end{itemize}

Remarkably, by a theorem of Alexandrov and Urysohn, $\N^\N$ is unique up to homeomorphism among non-empty topological spaces, that satisfy these properties~\cite{kechrisdescriptive}. That is, any non-empty, Polish, zero-dimensional topological space for which all compact subsets has empty interior is homeomorphic to the Baire space. Notably, 
\[
\R\setminus\Q \cong \N^\N,
\]
since irrational numbers under the subspace topology also satisfy these properties.

\section{The Space of Ends}\label{sec:spaceofends}

Informally, an end of a topological space represents a distinct way in which the space goes off to infinity. We formalize this by introducing the concept of \emph{exiting sequences}. A subset of a topological space $X$ is \emph{relatively compact} if its closure is compact in $X$.

\begin{definition}[\cite{aramayonavlamis2020}]
    Let $X$ be a topological space. A sequence $\set{U_n}_{n\in\N}$ of connected open subsets of $X$ is called an \emph{exiting sequence} if
    \begin{enumerate}[label=(\roman*)]
    \item $U_n \subset U_m$ whenever $m < n$, 
    \item $U_n$ is not relatively compact for any $n\in \N$, 
    \item $U_n$ has compact boundary for all $n\in \N$, and 
    \item any relatively compact subset of $X$ is disjoint from all but finitely many $U_n$.
    \end{enumerate}
\end{definition}

From this definition, see Figure~\ref{fig:exseq}, it is clear that two exiting sequences may go off to infinity in the same "direction", meaning that the terms of each sequence eventually lie within the terms of the other. Motivated by this, it is natural to put an equivalence relation on exiting sequences. We now state the required equivalence relation:

\begin{definition}\label{def:equivalence}
    Let $\set{U_n}_{n\in\N}$ and $\set{V_n}_{n\in\N}$ be two exiting sequences of a topological space X. If every element of $\set{U_n}$ is eventually contained in an element of $\set{V_n}$ and vice versa, then the two exiting sequences are equivalent.
\end{definition}

\begin{proposition}\label{prop:equivalence}
    The relation defined in Definition~\ref{def:equivalence} is an equivalence relation.
\end{proposition}
\begin{proof}
    Let $\set{U_n}_{n\in\N}$, $\set{V_n}_{n\in\N}$ and $\set{W_n}_{n\in\N}$ be three exiting sequences.
    \begin{enumerate}[label=(\roman*)]
        \item(Reflexivity) By the first property of exiting sequences, $\set{U_n} \sim \set{U_n}$. 
        \item (Symmetry) Assume $\set{U_n} \sim \set{V_n}$. Then for all $i\in\N$, there exists $j\in \N$ such that $U_i\subset V_j$, and vice versa. Thus, $\set{V_n} \sim \set{U_n}$.
        \item (Transitivity) Assume $\set{U_n} \sim \set{V_n}$ and $\set{V_n} \sim \set{W_n}$. Then for all $i_1\in\N$, there exists $j_1\in \N$ such that $U_{i_1}\subset V_{j_1}$ and for all $i_2\in\N$, there exists $j_2\in\N$ such that $V_{i_2}\subset W_{j_2}$. Therefore, for all $n\in \N$, there exists $m\in \N$ such that $U_{n}\subset W_{m}$. By the same argument, every $W_n$ lies within some $U_m$, so $\set{U_n} \sim \set{W_n}$.
    \end{enumerate}
\end{proof}

\begin{figure}[htbp]
  \centering
  \includegraphics[width=0.8\textwidth]{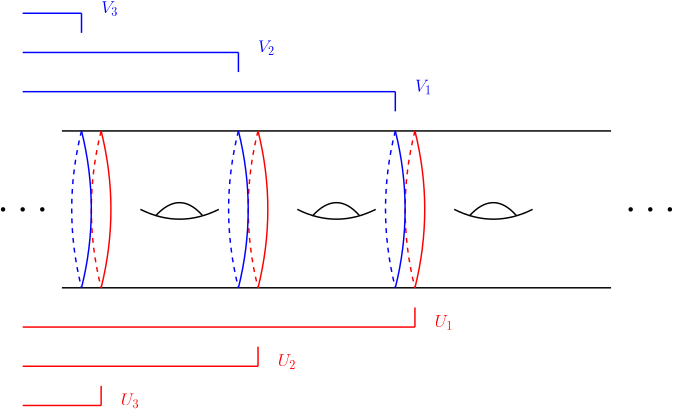}
  \caption{Two exiting sequences that define the same end on the Jacob's Ladder surface}
  \label{fig:exseq}
\end{figure}

We can properly define the \emph{Set of Ends} of a topological space $X$.

\begin{definition}
    Let X be a topological space. The equivalence classes of all exiting sequences under the relation given in Definition~\ref{def:equivalence} are called \emph{ends} of $X$. The set of all ends of $X$ is called the \emph{Set of Ends} of $X$ and is denoted by $\mathrm{Ends}(X)$.
\end{definition}

\begin{example} Let us give some examples of the ends of some topological spaces: 
\begin{itemize}
    \item Compact surfaces have no ends, since there are no subsets that are not relatively compact.
    \item The real line $\mathbb{R}$ has two ends, $\set{\infty,-\infty}$, represented by the sequences $\set{(-\infty, n)}_{n\in \N}$ and $\set{(n, \infty)}_{n\in \N}$ respectively. 
    \item The real plane $\R^2$ has one end at infinity, represented by the sequence $\set{\mathbb{R}^2 \setminus \overline{B(0,n)}}_{n \in \N}$.
    \item The real plane with a single puncture at the origin has two ends, one at infinity and one at the origin, represented by the sequences $\set{\mathbb{R}^2 \setminus \overline{B(0,n)}}_{n \in \N}$ and  $\set{\overline{B(0,\frac{1}{n})} \setminus \set{(0,0)}}_{n \in \N}$. Note here that each $\overline{B(0,\frac{1}{n})} \setminus \set{(0,0)}$ is relatively compact, since its closure inside $\mathbb{R}^2 \setminus \{(0,0)\}$ is itself, even though its closure in $\mathbb{R}^2$ is $\overline{B(0,\frac{1}{n})}$.
\end{itemize}
\end{example}

Note that for any exiting sequence $\{U_n\}_{n\in \mathbb{N}}$ and for any $m\in \mathbb{N}$, we can replace $U_m$ by uncountably many connected open sets contained in $U_{m-1}$ and containing $U_{m+1}$. This means that all ends of $X$ has uncountably many representatives, which makes working with this definition hard. Fortunately for surfaces (and manifolds in general), there is a convenient way to resolve this issue, using \textit{exhaustion by compact subspaces}:

\begin{definition}
    An \emph{exhaustion by compact subspaces} of a topological space X is a sequence of compact subspaces $\set{K_n}_{n\in\N}$ such that $K_n\subset \mathrm{int}({K}_{n+1})$ for all $n\in \N$ and $X=\bigcup_{n=0}^\infty K_n$.
\end{definition}

\begin{remark}
    Note that a topological space can also be exhausted by open sets. Consider the sequence $\set{B_n}_{n\in\N}$ in $\R^2$ where each $B_n$ is the open ball centered at the origin with radius $n$. Clearly, $\mathbb{R}^2=\bigcup_nB_n$. However, we are not interested in such exhaustions. For this reason, in this thesis, all exhaustions are assumed to be by compact subsets.
\end{remark}

\begin{example}
    Consider the sequence of surfaces $\set{K_n}_{n\in\N}$ with $K_n = S^3_{1+3n}$. The union $\bigcup_{n=0}^\infty K_n$ is the surface with three ends accumulated by genus. Note that $K_n$ is compact and $K_n \subset \mathrm{int}(K_{n+1})$ for all $n\in\N$. Therefore $\{K_n\}_{n\in\N}$ exhausts the surface with three ends accumulated by genus.
\end{example}

If $X$ admits an exhaustion, by fixing an exhaustion of $S$, we have a canonical choice of a representative of an end $e\in \mathrm{Ends}(X)$ given as follows. Fix an exhaustion $\{K_n\}_{n\in\mathbb{N}}$ of $X$, choose as the representative for an end $e\in \mathrm{Ends}(X)$, the sequence $\{U_n\}_{n\in \mathbb{N}}$ where each $U_n$ is a connected component of $X\setminus K_n$ and $U_n \subset U_{n+1}$ for all $n\in \mathbb{N}$.  

Surfaces, in virtue of being manifolds, are locally compact and Hausdorff. It is a classical result that locally compact Hausdorff spaces admit an exhaustion by compact subspaces~\cite{leetopological}. Now, we state the version of the definition for an end of $X$ that we will be using throughout the rest of this thesis~\cite{propertiesvlamispatel}.  

\begin{definition}\label{def:ends}
    Let X be a topological space that admits an exhaustion. Fix an exhaustion $\set{K_n}_{n\in\N}$ of $X$. Let $\set{U_n}_{n\in\N}$ be a sequence of connected components of $X\setminus K_n$ such that each $U_n \subset U_{n+1}$ for all $n\in \N$.
    Then, $\set{U_n}_{n\in \N}$ is called an \emph{end} of X.
\end{definition}

In light of Definition~\ref{def:ends}, one may wonder whether different exhaustions of $X$ determine different sets of ends. Since each end is an equivalence class of exiting sequences, it is clear that the choice of exhaustion only determines the choice of representatives for ends. That is, every exhaustion determines the same set of ends. Assume from now on that an exhaustion $\set{K_n}_{n\in\N}$ is fixed. 

 Let $U$ be a connected component of $X\setminus K_n$ for some $n\in \N$. We denote by $U^*$ the set of all ends contained in $U$, i.e., all ends such that $U_n\subset U$ for infinitely many $n\in \N$. We now give $\mathrm{Ends}(X)$ a topology to make it a topological space called the \emph{Space of Ends} of X.  If $U$ is an open subset of $X$ with compact boundary, and $e\in \mathrm{Ends}(X)$ is contained in $U^*$, then $U$ is said to be a \emph{neighborhood} of $e$. 

\begin{proposition}
    The collection of sets 
    \[
    \mathcal{B} = \set{U^* \, \mid \, U \text{ is a component of } S \setminus K_i \text{ for some } i\in \N}
    \]
    forms a basis for a topology on $\mathrm{Ends}(S)$.
\end{proposition}
\begin{proof}
    We need to prove two things:
    \begin{enumerate}[label=(\roman*)]
        \item $\mathcal{B}$ covers $\mathrm{Ends}(X)$:\\
        Let $ e\in \mathrm{Ends}(X)$ be arbitrary. By definition, $e = \set{U_n}_{n\in \N}$, and $U_1 \in S \setminus K_i$ for some $i\in \N$. Therefore, $e\in U_1^*\subset \mathcal{B}$ and $\mathcal{B}$ covers $\mathrm{Ends}(X)$.
        \item For every $U,V\in\mathcal{B}$ and for all $x\in U\cap V$ there exists a $W\in\mathcal{B}$ such that $x\in W\subseteq U\cap V$: \\
        Let $U$ be a connected component of $S\setminus K_i$, $V$ be a connected component of $S\setminus K_j$ for some $i,j \in \N$ and $e=\set{W_n}_{n\in \N}\in \mathrm{Ends}(X)$ be such that $e\in U^*\cap V^*$. Then, there exists $N_1,N_2\in \N$ such that for all $n\geq N_1$, $W_n \subset U$ and for all $n\geq N_2$, $W_n \subset V$. Since $e \in U^*$ and $e\in V^*$, without loss of generality we may assume $U\subset V$. Set $N = \max\{N_1,N_2\}$ so that for all $n \geq N$, $W_n \subset U = U \cap V$ and thus $e\in W_N^*\subset (U\cap V)^* = U^*\cap V^*$.
    \end{enumerate}
\end{proof}

\begin{remark}\label{rmk:endopen}
    Suppose $\mathrm{Ends}(X)$ has infinitely many ends. Then for some end $e= \set{U_n}_{n\in\N}$, $U^*_n$ must contain ends other than $e$ for infinitely many $n\in \N$, otherwise there would be only finitely many ends. Therefore, at least one end is a limit point, which shows that $\mathrm{Ends}(X)$ is not discrete and that the ends are not open when there are infinitely many ends.
\end{remark}
The space of ends exhibit some interesting properties, which will be familiar to the reader.

\begin{theorem}\label{thm:tdscc}
    For any surface $S$, the space of ends $\text{Ends}(S)$ is a second-countable, Hausdorff, totally disconnected, and compact topological space.
\end{theorem}
\begin{proof}
    (i) \textbf{Second-Countable}: A surface is second-countable, so it cannot have an uncountable number of connected components. Thus, for any exhaustion $\set{K_n}_{n\in\N}$, the number of components of each $S \setminus K_n$ is at most countable. The basis $\mathcal{B}$ is a countable union of these countable sets and is therefore countable.

    (ii) \textbf{Hausdorff}: Let $e_1 = \set{U_n}_{n\in\N}$ and $e_2 = \set{V_n}_{n\in\N}$ be two distinct ends. Therefore, there exists some $N \in \mathbb{N}$ such that for all $n \ge N$, $U_n \neq V_n$. It follows that $e_1\in U_n^*$ but $e_2 \notin U_n^*$ and $e_1 \notin V_n^*$ but $e_2 \in V_n^*$ for all $n\geq N$. The sets $U_N^*$ and $V_N^*$ are then disjoint open neighborhoods of $e_1$ and $e_2$, respectively, which shows that $\mathrm{Ends}(S)$ is Hausdorff.

    (iii) \textbf{Zero-dimensional}: We claim elements of $\mathcal{B}$ are clopen. They are open by definition, so it remains to show that they are closed. Let $U^*\in \mathcal{B}$. Then $U$ is a connected component of $S\setminus K_i$ for some $i\in \mathbb{N}$. Let $e\in \mathrm{Ends}(S) \setminus U^*$. By definition, there is a connected component $V \in S \setminus K_j$ for some $j \in \N$ with $U \cap V = \emptyset$. Hence, $V^*$ is a basic neighborhood of $e$ contained in $\mathrm{Ends}(S) \setminus U^*$. Since $e$ was arbitrary, $\mathrm{Ends}(X) \setminus U^*$ is open, which implies that $U^*$ is closed. Therefore, every element of $\mathcal{B}$ is clopen. Since we have already shown that $\mathrm{Ends}(S)$ is Hausdorff, it is zero-dimensional.

    (iv) \textbf{Totally Disconnected}: Assume for a contradiction that $\mathrm{Ends}(S)$ is not totally disconnected. Let $U$ be a connected subset of $\mathrm{Ends}(S)$ with $|U| > 1$. Let $e_1\not=e_2\in U$. Since $\mathrm{Ends}(S)$ is zero-dimensional, there is a clopen subset $V\subset \mathrm{Ends}(S)$ such that $e_1\in V$ and $e_2\notin V$. Observe that $U\cap V$ and $U \cap (\mathrm{Ends}(S)\setminus V)$ are both open and that $(U\cap V) \cap (U \cap(\mathrm{Ends}(S)\setminus V))= \emptyset$. Note that $U = (U \cap V) \cup (U\cap (\mathrm{Ends}(S)\setminus V))$, which contradicts the fact that $U$ is connected. Therefore, the only connected subsets of $\mathrm{Ends}(S)$ are singleton sets.

    (v) \textbf{Compact}: When $S$ has finitely many ends, this is trivial. So let $S$ have infinitely many ends and assume for a contradiction that $\mathrm{Ends}(S)$ is not compact, and let $\set{\mathcal{O}_\alpha}_{\alpha \in \Lambda}$ be an open cover with no finite subcover. Let $\set{K_n}_{n\in\N}$ be an exhaustion of $S$. Observe that since $K_n$ is compact, $S\setminus K_n$ has finitely many components for all $n\in\N$, say $U_{n,i}$. Because $\set{\mathcal{O}_\alpha}$ has no finite subcover, there must be for all $n$, some $i_n\in \N$ such that $U_{n,i_n}^*$ cannot be covered by finitely many $\mathcal{O}_\alpha$, otherwise we would have a finite subcover. Observe that the sequence $\set{U_{n,i_n}}_{n\in \N}$ defines an end $e$ of $S$. Because $\set{\mathcal{O}_\alpha}$ is a cover, $e\in \mathcal{O}_{\alpha_0}$ for some $\alpha_0$. Since $\mathcal{O}_{\alpha_0}$ is open and contains $e$, and since ends are not open, $\mathcal{O}_{\alpha_0}$ contains $U_{N,i_N}^*$ for some $N\in \N$, contradicting the fact that $U_{n,i_n}^*$ cannot be covered by finitely many $\mathcal{O}_\alpha$ for all $n$ and that $\set{\mathcal{O}_\alpha}$ has no finite subcover. Therefore, $\mathrm{Ends}(S)$ is compact.
\end{proof}

The space $S\ \cup \ \mathrm{Ends}(S)$ is a compactification of $S$ called the \emph{Freudenthal Compactification}. Theorem~\ref{thm:tdscc} can also be proved by working inside the Freudenthal Compactification. That approach is treated in~\cite{ahlforsriemann}; therefore, we provide a more direct proof in this thesis, which is due to the author. 

\begin{remark}
    A topological space $X$ with two ends can be embedded in a topological space $Y$ such that both ends of $X$ map to the single end of $Y$ in the Freudenthal compactification. Recall that $\R$ has two ends that correspond to $\pm \infty$ and its Freudenthal compactification is the extended real line $\R\cup \{+ \infty, - \infty\}$. Also recall that $\R^2$ has one end at infinity and its Freudenthal compactification is the Riemann sphere $S^2$. Note that if we embed $\R$ into $\R^2$ at $y=0$, both $+\infty$ and $-\infty$ are mapped to the point $\infty$ in the Freudenthal compactification. However, $\R$ still has two ends as an embedded submanifold of $\R^2$. This is because the terms of exiting sequences need to be connected, which means we cannot use open sets of $\R^2$ to define the ends of $\R$.
\end{remark}

Next corollary is actually a special case of a result in descriptive set theory that every zero-dimensional Polish space is homeomorphic to a closed subset of the Cantor set[\citenum{kechrisdescriptive}, Theorem 7.8]. We give a proof in the case of $\mathrm{Ends}(S)$ for the sake of completeness.

\begin{corollary}\label{crl:cantor}
    For any surface $S$, the space of ends $\text{Ends}(S)$ is homeomorphic to a closed subset of the Cantor set $\mathcal{C}$.
\end{corollary}
\begin{proof}
    Since $\mathrm{Ends}(S)$ is zero-dimensional, it has a basis consisting of clopen sets. Let \[\{B_1,B_2,B_3,...\}\] be such a basis. We can define a map $f: \mathrm{Ends}(S) \to 2^{\mathbb{N}}$ by sending an end $e$ to the sequence of its characteristic functions on the basis elements: \[f(e) = (\chi_{B_1}(e), \chi_{B_2}(e), \dots).\] This map is injective, since if the image of two ends is the same sequence, then they must share all of their neighborhoods, i.e., there is no open neighborhood separating the two, and since $\mathrm{Ends}(S)$ is Hausdorff, they must be equal. The map is also continuous, as characteristic functions defined on clopen sets are continuous. A continuous injection from a compact space to a Hausdorff space is a homeomorphism onto its image, and since $\mathrm{Ends}(S)$ is compact, its image is a closed subset of $C$.
\end{proof}

\section{Classification of Infinite-Type Surfaces}

As can be seen in Example~\ref{ex:infsurf}, some infinite-type surfaces have ends with infinitely many genera, while others do not. This motivates the next definition, which is required to classify infinite-type surfaces.

\begin{definition}\label{def:accubygenus}
    An end $e = \set{U_n}_{n\in \N}$ is \emph{accumulated by genus} or \emph{non-planar} if $U_n$ has infinitely many genera for all $n\in \N$. An end is $\emph{planar}$ if only finitely many $U_n$ has genera, i.e., there is a neighborhood $U\subseteq S$ of it that can be embedded into the Euclidean plane. The subspace of ends accumulated by genus is denoted by $\mathrm{Ends}_{np}(S)$.
\end{definition}

\begin{remark}
    By Definition~\ref{def:accubygenus}, a puncture is a planar end that is not an accumulation point in $\mathrm{Ends}(S)$.
\end{remark}

We are now ready to state the generalized classification theorem, which extends Theorem~\ref{thm:finiteclassification} to infinite-type surfaces:

\begin{theorem}[\cite{richardsclassification}]\label{thm:classification}
    Any surface $S$ is determined, up to homeomorphism, by the quadruple \[(b,g,\mathrm{Ends}_{np}(S),\mathrm{Ends}(S))\] where $b\in \N \cup \set{\infty}$ is the number of compact boundary components and $g\in \N\cup \{\infty\}$ is the number of genera.
\end{theorem}

\begin{remark}
    Note that by Theorem~\ref{thm:classification}, two surfaces $S$ and $S'$ with $b_1 = b_2$ and $g_1 =g_2$ are homeomorphic if and only if there exists a homeomorphism 
    \[
    f: \mathrm{Ends}(S) \longrightarrow\mathrm{Ends}(S')
    \]
    such that $f(\mathrm{Ends}_{np}(S)) = \mathrm{Ends}_{np}(S')$ 
\end{remark}

\begin{remark}\label{rem:class}
    Surfaces with non-compact boundary components, that is, surfaces with boundary components that are homeomorphic to $\R$, are not within the scope of Theorem~\ref{thm:classification}, as it cannot be exhausted by finite-type surfaces.
\end{remark}

By Corollary~\ref{crl:cantor}, we know that $\mathrm{Ends}(S)$ is homeomorphic to a closed subset of the Cantor Set. A natural question then is whether there is a surface such that the space of ends is homeomorphic to an arbitrary closed subset of the Cantor Set or not. The next theorem of Richards answers this in the affirmative, that is, for every closed subset $A$ of the Cantor Set, there is a surface $S$ such that $\mathrm{Ends}(S) \cong A$.

\begin{theorem}[\cite{richardsclassification}]
    Let $A\subseteq B$ be any two closed subsets of the Cantor Set. Then there exists a surface $S$ such that $\mathrm{Ends}(S) \cong B$ and $\mathrm{Ends}_{np}(S) \cong A$.
\end{theorem}

Recall from Section~\ref{sec:surfaces} that there are countably infinitely many homeomorphism classes of finite-type surfaces. A natural question, then, is what the cardinality of the homeomorphism classes of infinite-type surfaces is.

\begin{lemma}[\cite{reichbach}]\label{lem:closedcantor}
    There are uncountably many closed subsets of the Cantor set.
\end{lemma}

By Theorem~\ref{thm:classification} and Lemma~\ref{lem:closedcantor}, there are uncountable homeomorphism classes of infinite-type surfaces, in sharp contrast to the finite-type case.
\chapter{Structure of the Big Mapping Class Groups}\label{chp:bigmap}

Armed with the classification of infinite-type surfaces, we now turn our attention to their mapping class groups. For an infinite-type surface $S$, $\mathrm{Map}(S)$ is called a big mapping class group. Similar to the finite-type case, big mapping class groups inherit a topology via the quotient topology of the compact-open topology of $\homeo(S,\p S)$. However, in stark contrast to the finite case, mapping class groups of infinite-type surfaces are uncountable, and their topology is not discrete, hence the name big mapping class groups. Indeed, our main aim in this chapter is to show that big mapping class groups are homeomorphic to the Baire Space.

\section{The Alexander Method for Infinite-Type Surfaces}

The Alexander method is an invaluable tool for the study of mapping class groups of finite-type surfaces. Fortunately, thanks to the work of Hernandez-Morales-Valdez~\cite{infinitealexander}, the Alexander method can be extended to infinite-type surfaces with slight adjustments:

\begin{theorem}[The Infinite Alexander Method]\label{thm:infalex}
    Let $S$ be an infinite-type surface. There exists a collection of essential arcs and simple closed curves $\Gamma$ on $S$ such that any $f \in \mathrm{Homeo}^+(S,\partial S)$ that preserve the isotopy classes of the elements of $\Gamma$ is isotopic to the identity.
\end{theorem}

In particular, any $f\in \homeo^+(S,\p S)$ fixing every isotopy class of simple closed curves in $S$ is isotopic to the identity.

\section{The Curve Graph}

Another important tool for the study of big mapping class groups, is the \emph{curve graph} of a surface $S$, also called the complex of curves, which we denote by $\mathcal{C}(S)$.

The \emph{curve graph} $\mathcal{C}(S)$ is a graph, whose vertices are the isotopy classes of nontrivial non-peripheral simple closed curves on $S$, where there is an edge between two vertices if their geometric intersection number is zero, i.e., they have disjoint representatives.

\begin{definition}\label{def:aut}
    Let $\mathcal{G}=(V,E)$ be a graph with vertex set $V$ and edge set $E$. The \emph{automorphism group} of $\mathcal{G}$ is the group
    \[
    \mathrm{Aut}(\mathcal{G}) = \set{g \colon V \rightarrow V \, \mid \text{g is bijective and } (u,v) \in E \iff (g(u),g(v)) \in E}.
    \]
\end{definition}

The image of a simple closed curve under a homeomorphism is again a simple closed curve. Moreover, homeomorphisms preserve geometric intersection numbers. Therefore, it is clear that a mapping class $f\in \map^\pm(S)$ induces an automorphism $\bar{f}\in \mathrm{Aut}(\mathcal{C}(S))$. Moreover through the combined efforts of Ivanov~\cite{ivanov97}, Korkmaz~\cite{korkmaz99} and Luo~\cite{luo2000}, it is known in the case of finite-type surfaces, except the twice-punctured torus and surfaces with $3g+b-4<1$, that $\mathrm{Aut}(\mathcal{C}(S))$ is naturally isomorphic to $\map^\pm(S)$. In the case of infinite-type surfaces, Hernandez-Morales-Valdez~\cite{hernandezmorales} and Bavard-Dowdall-Rafi~\cite{bavardkafi} proved the following theorem independently, which extends the result to infinite-type surfaces:

\begin{theorem}\label{thm:simplicialaut}
    For any infinite-type surface $S$ with empty boundary, $\mathrm{Aut}(\mathcal{C}(S))$ is naturally isomorphic to $\mathrm{Map}^\pm(S)$.
\end{theorem}

Theorem~\ref{thm:simplicialaut} is especially important to study the topology of big mapping class groups.  

\section{The Topology of Big Mapping Class Groups}

Similarly to the mapping class groups of finite-type surfaces, big mapping class groups inherit a topology coming from the compact-open topology on $\homeo(S,\p S)$, but the similarities end there. In stark contrast to Proposition~\ref{prop:discretegroup}, big mapping class groups are not discrete.

\begin{proposition}\label{prop:indiscretegroup}
    Let $S$ be an infinite-type surface. Then, $\map(S)$ with its quotient topology is Hausdorff. Moreover, $\map(S)$ is not discrete.
\end{proposition}
\begin{proof}
By Lemma~\ref{lem:grhaus} it suffices to show that the singleton set $\set{1_{\map(S)}}$ is closed to establish that $\map(S)$ is Hausdorff.
    Since $\map(S)$ is endowed with the quotient topology, this is the same as showing that $\mathrm{Homeo}_0(S, \p S)$ is closed in $\mathrm{Homeo}^+(S, \p S)$.\\
    Let $\set{f_n}_{n\in\N}$ be a sequence converging to $f\in \homeo^+(S,\p S)$ with $f_n \in \homeo_0(S,\p S)$ for all $n\in \N$. Since each $f_n$ is isotopic to the identity, we can assume, up to isotopy, that $f_n$ fixes every isotopy class of simple closed curves in $S$. By Proposition~\ref{prop:compactopenconvergence}, $f_n|_{K}$ converges to $f|_K$ uniformly on every compact subset $K\subset S$. It follows that $f$ fixes the isotopy class of every simple closed curve, since they all fall in some compact set $K$, where $f_n$ converges uniformly to $f$. By Theorem~\ref{thm:infalex} (the Alexander Method for infinite-type surfaces), $f$ is isotopic to the identity, which implies that $\homeo_0(S, \p S)$ is closed.

    To show that $\map(S)$ is not discrete, fix an exhaustion $\set{K_n}_{n\in\N}$. We shall construct a sequence of nontrivial mapping classes that converge to the identity inside $\map(S)$. Since $S$ is of infinite type, the complement $S \setminus K_n$ is not a disjoint union of disks, once-punctured disks or spheres, and thus supports some non-trivial mapping class $[f_n]$. Since they are the identity on the boundary, each $[f_n]$ can be extended to $S$ to be the identity on $K_n$.
    We claim that the sequence $\set{[f_n]}$ converges to the identity. Indeed, for any compact subset $K$ of $S$, there exists $N\in \N$ such that $K\subseteq K_N$. By construction, $f_n|_{K_N}$ is the identity, up to isotopy, for all $n\geq N$. It follows that $f_n$ converges uniformly to the identity on $K$, which proves the claim. Hence, the identity is not isolated in $\map(S)$, which by Proposition~\ref{prop:discretegroup} implies that $\map(S)$ is not discrete.
\end{proof}

Since working with the quotient topology is often challenging, we instead work with the permutation topology. Originally introduced by Maurer, see~\cite{möller}, this topology was later applied by Aramayona, Patel, and Vlamis~\cite{aramayonapatelvlamis} to analyze the structure of $\map(S)$. Moreover, Aramayona and Vlamis~\cite{aramayonavlamis2020} state that the mapping class group of an infinite-type surface is homeomorphic to the Baire space proving an outline for a proof. We adopt their approach and provide full proofs for the arguments they sketched or left to the reader, for the sake of completeness. 

\begin{definition}
    Let $\mathcal{G}=(V,E)$ be a graph. For any finite subset $A \subset \mathcal{G}$, define
\[
U(A) = \set{g \in \mathrm{Aut}(\mathcal{G}) \ | \ g(a) = a \ \text{for all} \ a \in A }
\]
the point-wise stabilizers of $A$. The \emph{permutation topology} on $\mathrm{Aut}(\mathcal{G})$ is the topology generated by the $\mathrm{Aut}(\mathcal{G})$-translates of $U(A)$ for all finite $A\subset \mathcal{G}$, that is, sets of the form
\[
\set{g\,U(A) \mid g \in \mathrm{Aut}(\mathcal{G}), A \subset \mathcal{G} \text{ is finite}}.
\]
\end{definition}

\begin{lemma}\label{lem:sepa}
    For a graph $\mathcal{G} = (V,E)$ with countable $V$, $\mathrm{Aut}(\mathcal{G})$ with the permutation topology is separable.
\end{lemma}
\begin{proof}
    Let $g\in \mathrm{Aut}(\mathcal{G})$ be arbitrary, $A\subset V$ finite, and define
\[
D := \{ g \in \mathrm{Aut}(\mathcal{G}) \mid g(v) \neq v \text{ for only finitely many } v \in V \}.
\]

Let $F = A \cup g(A)$ and define
\[
h(v) = \begin{cases}
   g(v) & \text{if } v \in A, \\
   a & \text{if } v =g(a) \text{ for some }a \in A \\
   v & \text{ if } v \notin F
\end{cases}
\]
Then $h$ is an automorphism because $g$ is an automorphism, so $h \in g\,U(A) \cap D$.  

Since $g$ and $A$ were arbitrary, $D$ intersects every basic open set, so it is dense. Moreover, $D$ is countable because each element of $D$ has finite support and since $V$ is countable, there are countably many finite subgraphs of $\mathcal{G}$.

Thus $D$ is a countable dense subset of $\mathrm{Aut}(\mathcal{G})$, proving that it is separable.

\end{proof}

    Let $\set{x_i}_{n \in \N}$ be an enumeration of $V$. Consider the map 
    \[
    d\colon \mathrm{Aut}(\mathcal{G)}\times \mathrm{Aut}(\mathcal{G)} \longrightarrow \R
    \]
    given by
    \[
    d(g,g') = \inf_{\N \cup \set{0}}\set{2^{-n}\, \mid \, g(x_i) =g'(x_i) \text{ for all }i<n}.
    \]
    It is easy to see that $d$ is a metric. However, it is not complete. Consider the sequence $\set{g_n}_{n\in \N}$ on a complete graph $\mathcal{G}$ with $g_n \in \mathcal{G}$, where each $g_n$ is defined by
    \[
    g_n(x_i) = \begin{cases}
        x_{i+1} & \text{if }  i<n, \\
        x_1 & \text{if } i= n, \\
        x_i & \text{if }i>n. 
    \end{cases}
    \]
    Each $g_n$ is clearly an automorphism. Pick $\epsilon>0$. Then there exists $N\in \N$ such that $\frac{1}{2^N}<\epsilon$. Observe that for any $n,m\in \N$ with $m>n>N$,
    \[
    g_n(x_i) = g_m(x_i) \text{ for all } i< N.
    \]
    It follows that $d(g_n,g_m) < \frac{1}{2^N}< \epsilon$ and thus $\set{g_n}$ is Cauchy.
    
    Let $s$ be the map defined by $s(x_i) = x_{i+1}$ for all $n$. Pick $\epsilon >0$ and choose $N$ such that $\frac{1}{2^N}<\epsilon$. Then for any $n>N$,
    \[
    d(s,g_n) = \frac{1}{2^n} < \frac{1}{2^N} < \epsilon,
    \]
    which means that $\set{g_n}$ converges to $s$. However, $s$ is not an automorphism, since there exists no $x_i$ such that $s(x_i) = x_1$. This is because the sequence $\set{g^{-1}_n}_{n\in \N}$ is not Cauchy. Indeed, $g_n^{-1}(x_1) = x_n$ and $g_m^{-1}(x_1) = x_m$, which implies that $d(g^{-1}_n,g^{-1}_m) = 1$ for all $n$ and $m$. The next lemma gives a necessary and sufficient condition for a Cauchy sequence to converge to an automorphism.

\begin{lemma}\label{lem:iff}
    Let $\mathcal{G} = (V,E)$ be a graph with countable $V$ and $d$ be the metric defined above. A Cauchy sequence $\set{g_n}_{n\in \N}$ with $g_n \in \mathcal{G}$ converges to an element $g \in \mathrm{Aut}(\mathcal{G})$ if and only if the sequence $\set{g^{-1}_n}$ is Cauchy.
\end{lemma}
\begin{proof}
    For the forwards direction, suppose that $\set{g_n}$ is Cauchy and converges to $g\in \mathrm{Aut}(\mathcal{G})$. Since $g$ is an automorphism, $g^{-1}$ exists, and for all $i\in \N$, we can write $x_i = g(x_j)$ for some $j\in \N$. Let $\epsilon > 0$ and choose $N\in \N$ such that $\frac{1}{2^N}<\epsilon$. Since $\set{g_n}$ converges to $g$, we have $d(g_n,g)< \epsilon$ and $g_n(x_i) = g(x_i)$ for all $n\geq N$. Hence,
    \[
    g^{-1}(x_i) = g^{-1}(g(x_j)) = x_j = g^{-1}_n(g_n(x_j)) = g^{-1}_n(x_i) \text{ for all } i<n.
    \]
    which implies that $\set{g^{-1}_n}$ converges to $g^{-1}$ and is therefore Cauchy.

    For the backwards direction, suppose $\set{g_n}$ and $\set{g^{-1}_n}$ are both Cauchy. Let $\epsilon > 0$ and choose $N$ such that $\frac{1}{2^N}<\epsilon$. Then, for all $n,m \in \N$ with $m>n>N$, we have $d(g_n,g_m) < \epsilon$ and $d(g_n^{-1},g_m^{-1}) < \epsilon$. Hence,
    \[
    g_n(x_i) = g_m(x_i) \quad \text{and} \quad g^{-1}_n(x_i)=g_m^{-1}(x_i)
    \]
    for all $i<n$.
    It follows that the pointwise limits $g=\lim_{n\rightarrow\infty} g_n$ and $g'=\lim_{n\rightarrow\infty} g_n^{-1}$ exist. We claim that $g'(g(x_i)) = x_i$ and that $g(g'(x_i)) = x_i$ for all $i$. Indeed, for all $i$, there exists $N$ such that 
    \[
    g'(g(x_i)) = g^{-1}_N(g_N(x_i)) = x_i.
    \]
    Similarly, $g(g'(x_i)) = x_i$. It follows that $g'$ is the inverse of $g$ and that $g\in \mathrm{Aut}(\mathcal{G})$.
\end{proof}
    
A topological group is a \emph{Polish group} if its underlying topological space is Polish.

\begin{proposition}\label{prop:polish}
    For a graph $\mathcal{G} = (V,E)$ with countable $V$, $\mathrm{Aut}(\mathcal{G})$ with the permutation topology is Polish.
\end{proposition}
\begin{proof}
    Recall from Section~\ref{sec:baire} that a Polish space is a separable and completely metrizable space. By Lemma~\ref{lem:sepa}, $\mathrm{Aut}(\mathcal{G})$ is separable. Consider the metric given by
    \[
    d'(g,g') = d(g,g') + d(g^{-1},g'^{-1}).
    \]
    Let $\set{g_n}_{n\in \N}$ be a Cauchy sequence with respect to $d'$. Therefore, for all $\epsilon > 0$, there exists $N\in \N$ such that for all $n,m\in \N$ with $n,m>N$, we have $d'(g_n,g_m) < \epsilon$. Note that
    \[
    d(g_n,g_m) \leq  d(g_n,g_m) + d(g^{-1}_n,g_m^{-1}) = d'(g_n,g_m) < \epsilon,
    \]
    and
    \[
    d(g^{-1}_n,g_m^{-1}) \leq  d(g_n,g_m) + d(g^{-1}_n,g_m^{-1}) = d'(g_n,g_m) < \epsilon.
    \]
     It follows that $\set{g_n}$ and $\set{g_n^{-1}}$ are both Cauchy. By Lemma~\ref{lem:iff}, $\set{g_n}$ converges to an automorphism $g\in \mathrm{Aut}(\mathcal{G})$ and thus, $d'$ is a complete metric. Since $\mathrm{Aut}(\mathcal{G})$ is separable and completely metrizable, it is Polish.
\end{proof}

It turns out that the metric topology given by the metric $d'$ is the permutation topology, which follows from the definitions.

\begin{proposition}\label{prop:compactemptyinter}
   For an infinite-type surface $S$, every compact subset of $\map^\pm(S)$ has empty interior.
\end{proposition}
\begin{proof}
    Suppose for a contradiction that $K\subset \map^\pm(S)$ is compact and has nonempty interior. Therefore, $K$ contains a basic open set. After translation by some mapping class, we may assume without loss of generality that the identity lies in $K$, thus $U(A)\subset K$ for some finite $A\subset \mathcal{C}(S)$. Let $c$ be an essential simple closed curve on $S$ disjoint from the elements of $A$. Then, all powers of $T_c$ is contained in $U(A)$. Hence,
    \[
    T=\set{T_c^n \,\mid\, n\in \N} \subset U(A) \subset K.
    \]
    We claim that $T$ has no accumulation points in $\mathrm{Map}^{\pm}(S)$.  
Indeed, for any compact subset $K_0 \subset S$ that intersects $c$, the restrictions $(T_c^n)|_{K_0}$ do not converge uniformly to any orientation-preserving homeomorphism as $n \to \infty$, which implies that they do not converge to an orientation-preserving homeomorphism in the compact-open topology.  
  
Hence, no subsequence of $T_c^n$ converges, and $T$ has no limit points.

However, $K$ is compact and Hausdorff, and thus every infinite subset of $K$ must have an accumulation point.  
This contradicts the fact that $T$ has no accumulation points. Since $K$ was arbitrary, every compact subset of $\mathrm{Map}^{\pm}(S)$ has empty interior.
\end{proof}

The fact that any open set in a topological group can be translated to contain the identity follows from the fact that translation is a homeomorphism in topological groups.

\begin{lemma}\label{lem:contgr}
    A group homomorphism between topological groups is continuous if and only if it is continuous at the identity element.
\end{lemma}
\begin{proof}
    The forwards direction is trivial, since any continuous map is continuous at every point. For the backwards direction, assume that $f\colon X\rightarrow Y$ be a group homomorphism that is continuous at the identity $1_X$. Let $x\in X$ be arbitrary and $U$ be an open neighborhood of $f(x)$. Since translation on a topological group is a homeomorphism, there exists an open neighborhood $V$ of $1_Y$ such that
    \[
    f(x)V \subseteq U.
    \]
    By continuity at the identity, $f^{-1}(V)$ is an open neighborhood of $1_X$, and thus $xf^{-1}(V)$ is an open neighborhood of $x$. Finally for any $y\in xf^{-1}(V)$, there is some $z\in f^{-1}(V)$ such that $y = xz$. Using the homomorphism property
    \[
    f(y) = f(xz) = f(x)f(z) \in  f(x)V \subseteq U,
    \]
    which implies that $f(xf^{-1}(V))\subseteq U$ and thus $f$ is continuous at $x$.    
    
\end{proof}

\begin{proposition}\label{prop:permutmap}
    Let $S$ be a borderless infinite-type surface. Then the permutation topology on $\map^\pm(S)$ coming from $\mathrm{Aut}(\mathcal{C}(S))$ agrees with the quotient topology.
\end{proposition}
\begin{proof}
    
We need to show that any basic open neighborhood in the permutation topology is contained in some basic open neighborhood in the compact-open topology and vice versa. It is enough to show this for basic open neighborhoods of the identity, since translation is a homeomorphism in topological groups. 

Take a basic open neighborhood $U(A)$ of the identity in the permutation topology for some finite $A \subset V$. By definition, 
\[
U(A) = \set{g \in \map^\pm(S) \, \mid \, g([\gamma_i]) = [\gamma_i] \, \text{for }i \in \set{1,2,\dots,n}}
\]
for some $n\in \N$. Let $U_i$ be an open annular neighborhood of $\gamma_i$. Observe that the elements of $B(\gamma_i,U_i)$ fix the isotopy class of $\gamma_i$. The finite intersection 
\[
W = \bigcap_{i=1}^n B(\gamma_i,U_i)
\]
is a basic neighborhood of the identity in the compact-open topology since it is a finite intersection of subbasic neighborhoods, and every element $f \in W$ fixes the isotopy class of every $\gamma_i$, which implies that $W\subseteq V$.

For the other direction, let $W$ be a basic open neighborhood of the identity in the compact-open topology. Then,
\[
W = \bigcap_{i=1}^n B(K_i,U_i)
\]
for some compact $K_i$ and open $U_i$ with $K_i \subset U_i$. Let $K$ be a compact subset of $S$ containing every $K_i$. By the Alexander method (Proposition~\ref{prop:alex}), there is a finite collection of simple closed curves and essential arcs $\set{\gamma_i}^n_{i\in\N}$ on $K$ such that any mapping class fixing the isotopy class of every $\gamma_i$ is isotopic to the identity on $K$. Therefore the elements of the set
\[
U = \set{g \in \map^\pm(S) \, \mid \, g([\gamma_i]) = [\gamma_i] \, \text{for }i \in \set{1,2,\dots,m}}
\]
are identity, up to isotopy, on $K$. It follows that $U\subseteq W$, since for any $f\in V$, $f(K_i) \subset U_i$. The set $U$ is a basic open neighborhood of the identity in the permutation topology, which finishes the proof.
\end{proof}

Take an arbitrary $g \in \mathrm{Aut}(\mathcal{G})\setminus U(A)$. Then, $g(a) \not= a$ for some $a\in A$. Observe that $gU(\set{a}) \subseteq \mathrm{Aut}(\mathcal{G})\setminus U(A)$, that is, $g$ has an open neighborhood contained in $\mathrm{Aut}(\mathcal{G})\setminus U(A)$. Since $g$ was arbitrary, $\mathrm{Aut}(\mathcal{G})\setminus U(A)$ is open, and therefore $U(A)$ is closed. This shows that the basis of the permutation topology consists of clopen sets. 

\begin{proposition}\label{prop:zero}
    For an infinite-type surface $S$ with empty boundary, $\map^\pm(S)$ is zero-dimensional.
\end{proposition}
\begin{proof}
    We proved that the basis of the permutation topology consists of clopen sets. By Proposition~\ref{prop:permutmap}, the permutation topology agrees with the quotient topology and thus $\map^\pm(S)$ has a basis consisting of clopen sets. Moreover, by Proposition~\ref{prop:indiscretegroup}, it is Hausdorff and therefore zero-dimensional.
\end{proof}

We now arrive at the main result of this chapter.

\begin{theorem}\label{thm:mapbaire}
    Let $S$ be an infinite-type surface. Then the extended mapping class group $\map^\pm (S)$ is homeomorphic to the Baire space $\N^\N$.
\end{theorem}
\begin{proof}
    Recall from Section~\ref{sec:baire} that the Baire space is the unique nonempty, zero-dimensional, Polish topological space for which every compact subset has empty interior. When $S$ has no boundary, $\map^\pm(S)$ is Polish, zero-dimensional and any compact subset of $\map^\pm(S)$ has empty interior by Propositions~\ref{prop:polish}, \ref{prop:compactemptyinter}, \ref{prop:permutmap}, and \ref{prop:zero}. Therefore, 
    \[
    \mathrm{Map}^\pm(S) \cong \N^\N.
    \]
    When $S$ has boundary, we can embed $S$ in a larger surface $S'$ without boundary. In that case, $\map^\pm(S)$ is closed in $\map^\pm(S')$, which follows directly from the definitions. Since closed subgroups of Polish groups are Polish, $\map^\pm(S)$ is Polish. Again by Proposition~\ref{prop:compactemptyinter}, the interior of any compact subset $K \subset \map^\pm(S)$ is empty. Combining this with the fact that $\map^\pm(S)$ is zero-dimensional gives us
    \[
    \mathrm{Map}^\pm(S) \cong \N^\N.
    \]
    
\end{proof}

As an immediate corollary, we have that $\map(S)$ is homeomorphic to the Baire space.

\begin{corollary}
    Let $S$ be an infinite-type surface. Then the mapping class group $\map(S)$ is homeomorphic to $\N^\N$.
\end{corollary}
\begin{proof}
    
    Recall from Section~\ref{sec:extended} that $\map(S)$ is an index $2$ subgroup of $\map^\pm(S)$ and that,
    \[
    \map^\pm(S) = \map(S) \sqcup f\,\map(S),
    \]
    for an orientation reversing mapping class $f$. It follows from the definitions that $\map(S)$ is a closed subgroup of $\map^\pm(S)$ and therefore Polish. Similarly, $f\,\map(S)$ is also closed, and thus $\map(S)$ is open. The fact that $\map(S)$ is open in $\map^\pm(S)$ implies that their open sets coincide and since $\map(S)$ is closed, any compact $K\subset \map(S)$ is compact in $\map^\pm(S)$. It follows from Theorem~\ref{thm:mapbaire} that the interior of any compact $K$ is empty, and 
    \[
    \map(S) \cong \N^\N.
    \]
    \end{proof}

\chapter{Generating the Big Mapping Class Groups}\label{chp:gen}

 In Chapter~\ref{chp:bigmap}, we established that unlike the mapping class groups of finite-type surfaces, big mapping class groups are uncountable and indiscrete. This means that they cannot be generated by even a countably infinite generating set, let alone a finite one. This motivates the question concerning the nature of generating sets for big mapping class groups. Our aim in this chapter is to explain the answer to this question. In order to do this, we need the concept of topological generation. 

\section{Topological Generation}\label{sec:topogen}

A topological group $X$ is said to be \emph{topologically generated} by a subset $I$ if there exists a countable subset $I\subset X$ such that $I$ is dense in $X$, that is, $\bar I = X$. Since big mapping class groups are Polish, they are separable. A separable topological space $X$ necessarily has a countable dense subset, therefore, even though there are no countable generating sets for big mapping class groups, they are topologically generated by a countable set.

Recall the definition of the pure mapping class group for finite-type surfaces in Section~\ref{sec:puremap}. We can generalize the construction to the infinite-type case. Define the map

\[
\pi \colon \map(S) \longrightarrow \homeo(\mathrm{Ends}(S),\mathrm{Ends}_{\mathrm{np}}(S))
\]
given by the action of a mapping class on the space of ends. It can be checked that $\pi$ is well-defined and is a group homomorphism. The fact that $\pi$ is onto can be obtained by adapting Richards' proof of Theorem~\ref{thm:classification}. The kernel of $\pi$ is the pure mapping class group. We extend our definition to the general case.

\begin{definition}\label{def:puremap}
    The \emph{pure mapping class group}, denoted $\pmap(S)$, is the kernel of the group homomorphism $\pi \colon \map(S) \longrightarrow\homeo(\mathrm{Ends}(S),\mathrm{Ends}_{\mathrm{np}}(S))$. Equivalently, it is the subgroup of $\map(S)$ consisting of those mapping classes that act trivially on the space of ends. Elements of $\pmap(S)$ are called \emph{pure mapping classes}.
\end{definition}

We also have the following short exact sequence: 

\[
0\longrightarrow \mathrm{PMap}(S) \longrightarrow \mathrm{Map}(S) \longrightarrow \mathrm{Homeo}(\mathrm{Ends}(S),\mathrm{Ends}_{\mathrm{np}}(S)) \longrightarrow 0
\]
where the first map is the inclusion of $\pmap(S)$ into $\map(S)$.

\begin{remark}
    Note that when $\mathrm{Ends}(S)$ consists of finitely many ends of the same type, 
    \[
    \mathrm{Homeo}(\mathrm{Ends}(S),\mathrm{Ends}_{\mathrm{np}}(S)) = \mathrm{Sym}_n.
    \]
    When it consists of finitely many planar ends, then the short exact sequence is the same one in Section~\ref{sec:puremap}.
\end{remark}

\begin{remark}
    The pure mapping class group is equal to the mapping class group if and only if $|\mathrm{Ends}(S)| \leq 1$ or $|\mathrm{Ends}(S)|=2$ and $|\mathrm{Ends}_{\mathrm{np}}(S)| = 1$.
\end{remark}

It follows from the short exact sequence above that $\map(S)$ is topologically generated by $\pmap(S)$ together with any choice of lifts of generators for the relevant homeomorphism group of the ends. This motivates our focus on the topological generation of the pure mapping class group.

\begin{definition}\label{def:approx}
    Let $S$ be an infinite-type surface and let $f\in \homeo(S,\p S)$. We say that $f$ is \emph{approximated by} the sequence $\set{f_n}_{n\in\N}$ if $f_n\rightarrow f$ in the compact-open topology. In particular, we say that $f$ is approximated by Dehn twists if each $f_n$ is a product of Dehn twists.
\end{definition}

\begin{remark}
    By Proposition~\ref{prop:compactopenconvergence}, Definition~\ref{def:approx} states that $f$ is approximated by the sequence $\{f_n\}$ if $f_n\rightarrow f$ compactly.
\end{remark}

Recall from Theorem~\ref{thm:bdarygen}, that Dehn twists generate the pure mapping class group for finite-type surfaces. However, this result does not extend to infinite-type surfaces because of the existence of mapping classes in $\pmap(S)$ that cannot be approximated by Dehn twists. This was first shown by Patel and Vlamis~\cite{propertiesvlamispatel}. One such example of a mapping class on the Jacob's Ladder surface is the mapping class that shifts every embedded copy of $S^2_1$ to the right by one genus, whose action on homology classes is depicted in Figure~\ref{fig:hshift}. Consider the separating simple closed curve $\beta$ and an arbitrary finite-type subsurface $F\subset S$ containing $\beta$ whose boundary components are separating and the two components of $S\setminus F$ each intersect one end. Then $F\setminus \beta$ has two components. Call the component on the right $U$. For any compactly supported mapping class $g$, $U$ and $g(U)$ have the same number of genera and contain the same boundary components. Therefore, by Proposition~\ref{prop:twistintersection}, $\beta$ and $g(\beta)$ either intersect, or are isotopic. Note that $\beta$ and $f(\beta)$ in Figure~\ref{fig:hshift} bound a genus $1$ surface, which implies that they are disjoint but not isotopic, so no compactly supported mapping class agrees with $f$ on $F$. It follows that $f$ cannot be approximated by Dehn twists and that Dehn twists are not always enough to topologically generate $\pmap(S)$. The mapping class depicted in Figure~\ref{fig:hshift} is called a \emph{handle shift}, see Section~\ref{sec:hshift}.

\begin{figure}[htbp]
  \centering
  \includegraphics[width=0.75\textwidth]{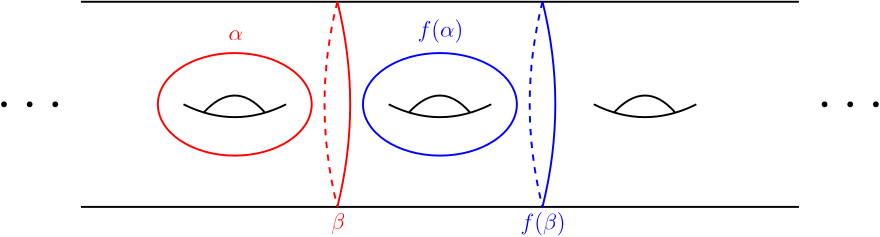}
  \caption{Action of a handle shift on two different homology classes on the Jacob's Ladder surface}
  \label{fig:hshift}
\end{figure}

 The addition of handle shifts to a topological generating set is not always required. In fact, there is a necessary and sufficient condition for when handle shifts are required, as stated in the next theorem.

\begin{proposition}[{\cite[Proposition 6.2]{propertiesvlamispatel}}]\label{prop:dehngen}
    Let $S$ be a surface. Then, $\pmap(S)$ is topologically generated by Dehn twists if and only if $|\mathrm{Ends}_{\mathrm{np}}(S)| \leq 1$.
\end{proposition}

We skip the proof of Proposition~\ref{prop:dehngen}, since it is a corollary of Theorem~\ref{thm:cohomology}, which is a much stronger result that we prove later in this chapter.

Although Dehn twists topologically generate $\pmap(S)$ only when there are less than two ends accumulated by genus, they always topologically generate a subgroup of $\pmap(S)$. Dehn twists are compactly supported, as they are defined in an annular neighborhood of some simple closed curve on $S$. This motivates the next definition.

\begin{definition}\label{def:compmap}
    Let $S$ be an infinite-type surface. The subgroup generated by Dehn twists, denoted by $\pmap_\mathrm{c}(S)$, is called the \emph{compactly supported mapping class group}. Elements of $\pmap_\mathrm{c}(S)$ are called compactly supported mapping classes.
\end{definition}

 We can consider the compactly supported mapping class group as the limit of the subgroups of $\pmap(S)$ that are topologically generated by Dehn twists over every essential subsurface, that is, 
\[
\pmap_\mathrm{c}(S) = \varinjlim(\pmap(X))
\]
over every essential subsurface $X$ of $S$. Suppose the complement $S\setminus X$ contains a disk component $D$. Since the mapping class group of a disk is trivial, extending $X$ to $X\cup D$ does not introduce any new mapping classes. By induction, we can fill all disk components in this manner, obtaining essential surfaces. Similarly, if $S\setminus X$ contains an annulus component $A$, the Dehn twist generating $\map(A)$ is isotopic to a Dehn twist about a boundary component of $X$ in $X\cup A$. Consequently, we can glue any annulus components of $S\setminus X$ to $X$ without the need to add any new mapping classes. This shows that it suffices to restrict our attention to essential subsurfaces.

\begin{remark}\label{thm:dehngen2}
    In light of the definition of $\mathrm{PMap}_\mathrm{c}(S), $Proposition~\ref{prop:dehngen} states that $\mathrm{PMap}(S) = \overline{\mathrm{PMap}_\mathrm{c}(S)}$ if and only if $|\mathrm{Ends}_{\mathrm{np}}(S)| \leq 1$.
\end{remark}

\section{Handle shifts}\label{sec:hshift}

Consider the surface obtained from $\R \times [-1,1]$ by removing disks of radius $\frac{1}{4}$ centered at each point $(n,0)$ for $n\in \Z$ and attaching a torus with $1$ boundary component along the boundary of each removed disk. This surface is called the \emph{model surface of a handle shift} and denoted by $\Sigma$. 

\begin{remark}
    By Remark~\ref{rem:class}, the model surface $\Sigma$ is not in the scope of Theorem~\ref{thm:classification}, as the boundary components $\mathbb{R}\times \{-1\}$ and $\mathbb{R}\times \{1\}$ are homeomorphic to $\R$. Fortunately, we do not need to classify $\Sigma$ for our purposes.
\end{remark}

The model surface $\Sigma$ is a surface with two ends accumulated by genus, 
\[
+\infty = \set{[0,1]\times [n,\infty)}_{n\in \N} \quad \text{and} \quad -\infty = \set{[0,1]\times (-\infty,n]}_{n\in\N},
\]
where of course the attached handles are contained in the relevant intervals. \\
The model surface $\Sigma$ can be embedded into any infinite-type surface $S$ with at least one end accumulated by genus, even if it is not immediately clear how for surfaces with only one end accumulated by genus, see Section~\ref{sec:lochness}. It is clear that any homeomorphism 
\[
f\colon\Sigma \longrightarrow\Sigma
\]
induces a homeomorphism 
\[
\bar f\colon S \longrightarrow S
\]
by extending $f$ to be identity on $S\setminus \Sigma$.

There is a homeomorphism of $\Sigma$ that is of particular importance to the topological generation of big mapping class groups, first introduced in~\cite{propertiesvlamispatel} by Patel and Vlamis. Define 
\[h\colon\Sigma \longrightarrow\Sigma\] by 
\[
h(x,y) = \begin{cases}
    (x+1,y) & \mathrm{if} \ y \in [-\frac{1}{2},\frac{1}{2}], \\
    (x+2-2y,y) & \mathrm{if} \ y \in [\frac{1}{2},1], \\
    (x+2+2y,y) & \mathrm{if} \ y \in [-1,-\frac{1}{2}] \\
\end{cases}
\]
 on $\R\times [-1,1]$. Each torus attached along the boundary of a removed disk moves together with that boundary.

This homeomorphism $h$, called a \emph{handle shift}, shifts every attached torus to the position of the next one on the right, hence the name. Any induced homeomorphism $\bar h$ on an infinite-type surface $S$, as well as its mapping class is also called an handle shift.

\begin{figure}[htbp]
  \centering
  \includegraphics[width=0.8\textwidth]{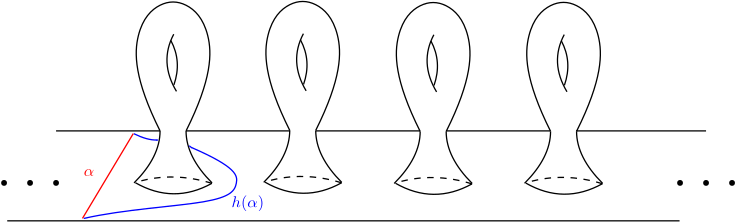}
  \caption{The model surface $\Sigma$ and the image of an arc under the handle shift $h$}
  \label{fig:modelhandleshift}
\end{figure}

Consider a handle shift 
\[
\bar h \colon S\longrightarrow S
\]
induced by 
\[
h \colon \Sigma \longrightarrow\Sigma
\]
where $S$ is any surface with at least two ends accumulated by genus. Let $e_+,e_-\in \mathrm{Ends}(S)$ be the two ends corresponding to the ends $\pm \infty$ of $\Sigma$. Assume without loss of generality that $h$ shift every genus by one to the right. Then $e_+$ is called the \emph{attracting end} of $\bar h$ and $e_-$ is called the \emph{repelling end}. The attracting and repelling ends are determined by how $\Sigma$ is embedded in $S$.

\begin{remark}
    Some authors use the phrases "handle shift acting on $e_-$ and $e_+$", or "handle shift taking $e_-$ to $e_+$" to specify which end is attracting and which is repelling. This convention should not be confused with the action of the handle shift on the space of ends. Since ends are equivalence classes of exiting sequences, the image of any exiting sequence under a handle shift lies in the same equivalence class. That is, handle shifts act trivially on the space of ends and therefore are pure mapping classes. To avoid any confusion, we shall use the phrase "handle shift from $e_-$ to $e_+$" to indicate which end is repelling and which is attracting.
\end{remark}

\section{Structure of Big Mapping Class Groups}

Recall from Remark~\ref{thm:dehngen2} that for all $S$ with $|\mathrm{Ends}_{\mathrm{np}}(S)| \leq 1$, we have $\mathrm{PMap}(S) = \overline{\mathrm{PMap}_\mathrm{c}(S)}$. The next theorem covers every other case,

\begin{proposition}[{\cite[Proposition 6.2]{propertiesvlamispatel}}]\label{prop:generate}
    Let $S$ be a surface with at least $2$ ends accumulated by genus. Then the set of Dehn twists together with the set of handle shifts topologically generate $ \mathrm{PMap}(S)$.
\end{proposition}

Analogous to Proposition~\ref{prop:dehngen}, we omit the proof of Proposition~\ref{prop:generate} as it follows directly from Theorem~\ref{thm:cohomology}.

A surface may admit uncountably many distinct handle shifts; this is particularly evident in cases where the surface possesses uncountably many ends accumulated by genus, such as the Blooming Cantor Tree surface. Surprisingly, this holds true even when a surface has finitely many ends accumulated by genus, see Remark~\ref{rmk:typeremark}. However, most of them are redundant and in fact, only a countable set of handle shifts is enough to topologically generate $\pmap(S)$, as shown by Aramayona-Patel-Vlamis~\cite{aramayonapatelvlamis}. Moreover, they showed that for every surface, the pure mapping class group is equal to the semidirect product of $\overline{\pmap_\mathrm{c}(S)}$ and a direct product of countably many handle shifts.

\begin{theorem}[{\cite[Corollary 6]{aramayonapatelvlamis}}]\label{thm:bigmap}
    Let $S$ be a surface. Then,
    \[
    \mathrm{PMap}(S) = \overline{\mathrm{PMap}_\mathrm{c}(S)} \rtimes \prod^\infty_{i=1}\langle h_i\rangle
    \]
    where each $h_i$ is a handle shift.
\end{theorem}

Theorem~\ref{thm:bigmap} is actually a corollary to a theorem about the first integral cohomology group of $\map(S)$. To understand the structure of big mapping class groups, it is important to understand Theorem~\ref{thm:cohomology} which we will work our way towards for the remainder of this chapter. The subgroup of $H_1(S,\Z)$ generated by separating homology classes are denoted by $H_1^{sep}(S,\Z)$ .

\begin{remark}\label{rmk:sep}
    When $S$ has one end accumulated by genus and no planar ends and no boundary components, $H_1^{sep}(S,\Z)$ is trivial, since any separating simple closed curve is the boundary of some subsurface $S_{g}$. If $S$ has more than one end accumulated by genus, then a separating simple closed curve may fail to be trivial in homology, since it may not always bound a compact subsurface. 
    \begin{figure}[htbp]
  \centering
  \includegraphics[width=0.8\textwidth]{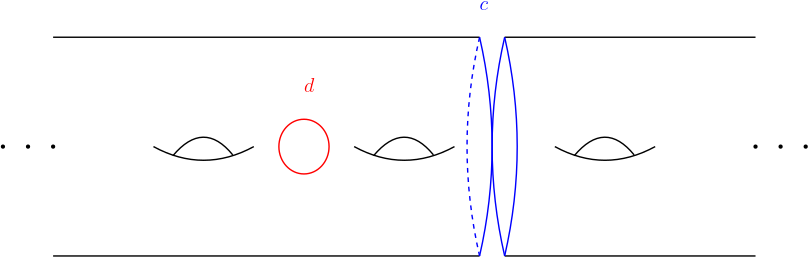}
  \caption{A non-trivial element $c$ and a trivial element $d$ in $H_1^{sep}(S,\Z)$ on the Jacob's Ladder surface}
  \label{fig:sep}
\end{figure}
\end{remark}

The subgroup of $H^1(S,\Z)$ that is isomorphic to $\mathrm{Hom}(H_1^{sep}(S,\Z),\Z)$ by the universal coefficient theorem for cohomology is denoted by $H^1_{sep}(S,\Z)$ . For details on the universal coefficient theorem and cohomology groups, see~\cite{hatcher}.

\begin{lemma}\label{thm:compurenormal}
    The group $\overline{\pmap_\mathrm{c}(S)}$ is a normal subgroup of $\pmap(S)$.
\end{lemma}
\begin{proof}
We start by showing $\pmap_\mathrm{c}(S)$ is a normal subgroup of $\pmap(S)$. Let $f \in \pmap_\mathrm{c}(S)$ and $g \in \pmap(S)$. 
Suppose $f$ is supported on a compact subset $K \subset S$. 
Since $f$ acts trivially on $S \setminus K$, the conjugate $g^{-1}fg$ is supported on $g^{-1}(K)$, which is again compact. 
Hence, $g^{-1}fg \in \pmap_\mathrm{c}(S)$, and therefore 
$\pmap_\mathrm{c}(S) \trianglelefteq \pmap(S)$.

To show that $\overline{\pmap_\mathrm{c}(S)}$ is normal in $\pmap(S)$, recall the classical fact of topology that for a continuous function $f\colon X\rightarrow Y$, and $A\subset X$
\[
f(\overline{A}) \subseteq \overline{f(A)}.
\]
Let $g\in \pmap(S)$. Since $\pmap(S)$ is a topological group, conjugation is continuous. It follows that
\[
g\overline{\pmap_\mathrm{c}(S)}g^{-1} \subseteq \overline{g\pmap_c(S)g^{-1}} = \overline{\pmap_\mathrm{c}(S)},
\]
where we used the fact that for normal subgroups are invariant under conjugation.

Observe that the homomorphism $g^{-1}fg$ is also continuous. Therefore,
\[
g^{-1}\overline{\pmap_\mathrm{c}(S)}g \subseteq \overline{g^{-1}\pmap_c(S)g} = \overline{\pmap_\mathrm{c}(S)},
\]
and
\[
\overline{\pmap_\mathrm{c}(S)} \subseteq g\overline{\pmap_\mathrm{c}(S)}g^{-1}.
\]
It follows that $\overline{\pmap_\mathrm{c}(S)} = g\overline{\pmap_\mathrm{c}(S)}g^{-1}$ and that $\overline{\pmap_\mathrm{c}(S)} \trianglelefteq \pmap(S)$.
\end{proof}

Since  $\overline{\pmap_\mathrm{c}(S)} \trianglelefteq \pmap(S)$, we can define the quotient group 
\[
A_S =  {\pmap(S)}/ \overline{\pmap_\mathrm{c}(S)}.
\] Note that a handle shift is mapped to a non-trivial element in $A_S$ with the quotient map while any mapping class approximated by Dehn twists is mapped to the identity in the quotient group. 

Recall that from Section~\ref{sec:spaceofends} that the Freudenthal compactification $\hat{S}$ of a surface $S$ is the union of $S$ and $\mathrm{Ends}(S)$. We can fill in any planar ends by considering the surface $\bar{S} = \hat{S}\setminus \mathrm{Ends}_{np}(S)$. Note that $\bar{S}$ only has ends accumulated by genus and no planar ends, since it is constructed by removing the non-planar ends from the Freudenthal compactification. The inclusion $\iota\colon S \hookrightarrow S$ induces an homomorphism $\iota_*\colon H^{sep}_1(S,\Z) \rightarrow H_1^{sep}(\hat{S},\Z)$ on the homology groups. 

We are now ready to state the main result of this chapter:

\begin{theorem}[{\cite[Theorem 3]{aramayonapatelvlamis}}]\label{thm:cohomology}
    Let $S$ be a surface with at least two ends accumulated by genus. There is an injection $\kappa : H^1_{sep}(\hat{S},\mathbb{Z}) \longrightarrow \mathrm{PMap}(S)$ such that $\pi \circ \kappa :H^1_{sep}(\hat{S},\mathbb{Z}) \longrightarrow A_S$ is an isomorphism, where $\pi: \mathrm{PMap}(S) \longrightarrow A_S$ is the natural projection. Moreover, there exists pairwise-commuting handle shifts $\{h_i\}^\infty_{i=1}$ such that
    \[
    \kappa(H^1_{sep}(\hat{S},\mathbb{Z})) = \prod^r_{i=1}\langle h_i\rangle
    \]
    where $r\in \mathbb{N} \cup \{\infty\}$ is the rank of $H_1^{sep}(\hat{S},\mathbb{Z})$.
\end{theorem}

\begin{remark}
    If $S$ has one end accumulated by genus, then $\kappa (H^1_{sep}(\hat{S},\Z))$ is trivial, since by Remark~\ref{rmk:sep} any separating simple closed curve bounds a subsurface and therefore is trivial in $H_1^{sep}(\hat{S},\Z)$. Then $\pmap(S) = \overline{\pmap_\mathrm{c}(S)}$ and this proves Proposition~\ref{prop:dehngen}.
\end{remark}

\begin{remark}\label{rem:zn-1}
     If $S$ has $n$ ends accumulated by genus, then there are exactly $n-1$ homology classes of separating simple closed curves and $\prod^n_{i=1}\langle h_i\rangle = \Z^{n-1}$.
\end{remark}

We shall assume that $S$ has no boundary for simplicity and follow the proof of Theorem~\ref{thm:cohomology} given by Aramayona-Patel-Vlamis. In the presence of boundary components, the same arguments work with slight modifications.

\begin{definition}
    Let $\Gamma$ be a graph with vertex set $V$ and edge set $E$. For a set $U$, a map $c: V\longrightarrow U$ is called a \emph{$U$-coloring}. If for any two adjacent vertices $u,w$, $c(u) \not= c(w)$, then $c$ is a \emph{proper $U$-coloring}.
\end{definition}

In particular, we will construct a proper \emph{$\Z$-coloring}.

Any oriented separating simple closed curve $\gamma$ partitions $\mathrm{Ends}(S)$ into two disjoint open sets, where some ends are on the left of $\gamma$, say $v^-$, and some are on the right, say $v^+$. This means that the homology class $v$ of $\gamma$ is determined, up to orientation, by the partition $v^+ \cup v^- = \mathrm{Ends}(S)$. Observe that if either one of $v^\pm$ is empty, then $v$ is the boundary of a compact subset and therefore is trivial. If $v^+$ or $v^-$ consists of a single end, we say $v$ separates that end. We call a homology class simple if there exists a simple closed curve representing that homology class.

For a nonzero element $v\in H_1(S,\Z)$, we denote by $\mathcal{C}_v(S)$ the subgraph of the curve graph consisting of the isotopy classes of simple closed curves that can be oriented to represent $v$.

Let $S$ be an infinite-type surface and fix a simple nonzero separating homology class $v$, and let $\gamma$ be an oriented simple closed curve representing $v$. Pick an end $e_v$ in $v^-$.

For any $c\in \mathcal{C}_v(S)$, $c\cup \gamma$ is compact, and thus there exists a connected compact subsurface $R$ such that each component of $\p R$ is separating and $c\cup \gamma$ lies inside the interior of $R$. There is a boundary component $\p_0$ of $R$ such that $e_v$ and $\mathrm{int}(R)$ are on different sides of $\p_0$. Let $\mathfrak{g}_R(c)$ and $\mathfrak{g}_R(\gamma)$ be the number of genera of the subsurface $R\setminus c$ and $R\setminus \gamma$ containing $\p_0$ respectively.

\begin{lemma}
    The map
    \[
    \phi_\gamma(c) = \mathfrak{g}_R(c) - \mathfrak{g}_R(\gamma).
    \]
    is a proper $\Z$-coloring of $\mathcal{C}_v(S)$.
\end{lemma}
\begin{proof}
    Observe that if $b,c \in \mathcal{C}_v(S)$ are adjacent, then they must cobound a compact surface $F$ with
    \[
    0 < \mathrm{genus}(F)  = |\phi_\gamma(b)- \phi_\gamma(c)|,
    \]
    which means that $\phi_\gamma$ is a proper $\Z$-coloring.
\end{proof}
 
\begin{remark}
    The map $\phi_\gamma(c)$ gives the sum of the signed number of genera of the subsurface(s) bounded by $\gamma$ and $c$, where the sign is determined by whether a genus is on the left of $\gamma$ or on its right. It follows that $\phi_\gamma$ does not depend on the choice $R$. Similarly, $\phi_\gamma$ does not depend on the choice of $e_v$ either.
\end{remark}

We will use the map $\phi_\gamma$ to construct a homomorphism from $\pmap(S)$ to $\Z$. The next lemma gives a few key properties of $\phi_\gamma$.

\begin{lemma}\label{lem:phi}
    If $v\in H_1^{sep}(S,\Z)$ is simple and nonzero, $\gamma$ and $\delta$ are oriented simple closed curves representing $v$ and $f\in \pmap(S)$, then
    \begin{enumerate}
        \item $(\phi_\gamma\circ f)(c) - \phi_\gamma(c) = \phi_\gamma(f(\gamma))$ for all $c\in \mathcal{C}_v(S)$.
        \item $\phi_\gamma(c) - \phi_\delta(c) = \phi_\gamma(\delta)=-\phi_\delta(\gamma)$ for all $c\in \mathcal{C}_v(S)$.
        \item $\phi_\gamma\circ f- \phi_\gamma = \phi_\gamma - \phi_{f(\gamma)}$.
    \end{enumerate}
\end{lemma}
\begin{proof}
    Choose an appropriate $R$. A direct computation gives the first property:
    \begin{align*}
        \phi_\gamma(f(c)) &= \mathfrak{g}_R(f(c)) - \mathfrak{g}_R(\gamma) \\
        &=\mathfrak{g}_R(f(c)) - \mathfrak{g}_R(f(\gamma)) + \mathfrak{g}_R(f(\gamma)) - \mathfrak{g}_R(\gamma) \\
        &= \phi_\gamma(c) + \mathfrak{g}_R(f(\gamma)) - \mathfrak{g}_R(\gamma) \\
        &= \phi_\gamma(c) + \phi_\gamma(f(\gamma)).
    \end{align*}
    Similarly, 
    \begin{align*}
        \phi_\gamma(c) - \phi_\delta(c) &= \mathfrak{g}_R(c) - \mathfrak{g}_R(\gamma) - \mathfrak{g}_R(c) + \mathfrak{g}_R(\delta) \\
        &= \mathfrak{g}_R(\delta)-\mathfrak{g}(\gamma) \\
        &= \phi_\gamma(\delta).
    \end{align*}
    Combining the first and second properties with $\delta = f(\gamma)$ gives the third property.
\end{proof}

Recall that a handle shift $h$ has an attracting end and a repelling end. For a simple $v\in H_1^{sep}(S,\Z)$, we say that $v$ \emph{cuts} $h$ if the attracting and repelling ends of $h$ are on the opposite sides of $v$. We are ready for the homomorphism from $\pmap(S)$ to $\Z$. 

\begin{proposition}\label{prop:properties}
    Let $S$ be an infinite genus surface. Let $v\in H_1^{sep}(S,\Z)$ be simple and nonzero. If $\gamma$ is any oriented simple closed curve representing $v$, then
    \[
    \varphi_v \colon \pmap(S) \longrightarrow \Z
    \]
    given by
    \[
    \varphi_v(f) = \phi_\gamma(f(\gamma)),
    \]
    is a wel-defined homomorphism. Moreover,
    \begin{enumerate}
        \item $\varphi_v(f) = 0$ for all $f\in \overline{\pmap_\mathrm{c}(S)}$.
        \item $\varphi_{-v}=-\varphi_v$.
        \item If $h\in \pmap(S)$ is a handle shift, then $\varphi_v(h) \not= 0$ if and only if $v$ cuts $h$.
        \item $\varphi_v \not= 0$.
    \end{enumerate}
\end{proposition}
\begin{proof}
    Fix a simple $v \in H_1^{sep}(S,\Z)$. Let $\gamma$ and $\delta$ be two nonisotopic simple closed curves representing $v$. By the second property of Lemma~\ref{lem:phi},
    \[
    \phi_\gamma(f(\gamma)) = \phi_\delta(f(\gamma)) - \phi_\delta(\gamma) = \phi_\delta(f(\delta)),
    \]
    and thus $\varphi_v$ is well-defined.

    Let $f,g\in \pmap(S)$. Then,
    \begin{align*}
        \varphi_v(fg) &= \phi_\gamma(fg(\gamma)) \\
        &= -\phi_{(fg)(\gamma)}(\gamma) \text{ (by the second property of Lemma~\ref{lem:phi})} \\
        &= -\phi_{g(\gamma)}(\gamma) + \phi_{g(\gamma)}(\gamma) - \phi_{(fg)(\gamma)}(\gamma) \\
        &= -\phi_{g(\gamma)}(\gamma) + (\phi_{g(\gamma)}\circ f)(\gamma)-\phi_{g(\gamma)}(\gamma) \text{(by the third property of Lemma~\ref{lem:phi})} \\
        &= -\phi_{g(\gamma)}(\gamma) + \phi_{g(\gamma)}(f(g(\gamma))) \text{(by first property of Lemma~\ref{lem:phi})} \\
        &= \varphi_v(g) + \varphi_v(f)
    \end{align*}

Now let $\bar{f}\in \overline{\pmap_\mathrm{c}(S)}$, $c \in \mathcal{C}_v(S)$ and $f\in \pmap_\mathrm{c}(S)$ such that $f(c) = \bar{f}(c)$. If we orient $c$ to represent $v$, $\mathfrak{g}_R(c)= \mathfrak{g}_R(f(c))$ for any appropriate $R$. Therefore,
\[
\varphi_v(\bar{f}) = \phi_c(\bar{f}(c)) = \phi_c(f(c)) = 0,
\]
which proves the first property.

Observe that
\begin{align*}
    \varphi_{-v}(f) &= \phi_{-\gamma}(f(-\gamma)) \\
    &= -\phi_{f(-\gamma)}(-\gamma) \\
    &= - [\mathfrak{g}_R(-\gamma)- \mathfrak{g}_R(f(-\gamma))] \\
    &= - [\mathfrak{g}_R(\gamma)- \mathfrak{g}_R(f(\gamma))] \\
    &= - \phi_{f(\gamma)}(\gamma) \\
    &= -\varphi_v(f),
\end{align*}
which proves the second property.

To prove the third property, let $h\in \pmap(S)$ be a handle shift cut by $v$. Then $h$ is supported on an embedded copy of the model surface $\Sigma$. Let $\gamma$ be a representative of $v$ such that the restriction of $\gamma$ to $\Sigma$ is a proper arc. Observe that $h(\gamma)$ is disjoint from $h$ and thus they cobound a surface with positive genus in $S$, which implies that $\varphi_v(h)\not=0$. Conversely, if $v$ does not cut $h$, then there exists a representative $\gamma$ of $v$ disjoint from the support of $h$, and thus $\varphi_v(h) =0$. 

    Recall that any nonzero $v\in H_1^{sep}(S,\Z)$ partitions $\mathrm{Ends}_\mathrm{np}(S)$ into two disjoint sets, and there exists a handle shift cut by $v$, which by the third property implies that $\varphi_v\not=0$, and this proves the fourth property.
\end{proof}

\begin{remark}\label{rmk:varphi}
    The map $\varphi_v$ assigns to each element $f\in \pmap(S)$, the signed genus count between a representative $\gamma$ of $v$ and its image under $f$.
\end{remark}

The homomorphism $\varphi_v$ is defined for every nonzero simple $v\in H_1^{sep}(S,\Z)$, and is an element of $H^1(\pmap(S),\Z)$. The main goal of the proof is to show that $\varphi_v$ can be defined for arbitrary elements of $H_1^{sep}(S,\Z)$, thereby showing the existence of a homomorphism from it to $H^1(\pmap(S),\Z)$. 

\begin{lemma}\label{lem:sum}
    Let $v_1,\dots,v_n$ be simple elements of $H_1^{sep}(S,\Z)$ such that there exist pairwise-disjoint oriented simple closed curves $\gamma_1,\dots,\gamma_n$ such that $\gamma_i$ represents $v_i.$ If $v= \sum v_i$ is simple, then,
    \[
    \varphi_v = \sum^n_{i=1}  \varphi_{v_i}.
    \]
\end{lemma}
\begin{proof}
    Since $v= \sum v_i$, the curves $-\gamma,\gamma_1,\dots,\gamma_n$ bound a compact subsurface $F$ in $S$. By the second property of Proposition~\ref{prop:properties}, we can assume without loss of generality that the $\gamma_i$ are on the right of $\gamma$.

    Let $f\in \pmap(S)$ and $e_v \in v^-$. We can choose a compact subsurface $R$ such that every component of $\p R$ is separating and $F\cup f(F) \subset \mathrm{int}(R)$. Let $\p_0$ be the component of $\p R$ that separates $e_v$ and $\gamma,\gamma_1,\dots,\gamma_n$lie on different sides of $\p_0$.

    Note that
    \[
    \mathfrak{g}_R(\gamma) = \left(\sum^n_{i=1}\mathfrak{g}(\gamma_i)\right) - (n-1)\mathrm{genus}(R) - \mathrm{genus}(F),
    \]
    and
    \[
    \mathfrak{g}_R(f(\gamma)) = \left(\sum^n_{i=1}\mathfrak{g}(f(\gamma_i))\right) - (n-1)\mathrm{genus}(R) - \mathrm{genus}(f(F)).
    \]
    Since $f$ is a homeomorphism, $\mathrm{genus}(F) = \mathrm{genus}(f(F))$. It follows that
    \begin{align*}
        \varphi_v(f) &= \phi_\gamma(f(\gamma)) \\
        &= \mathfrak{g}(f(\gamma)) - \mathfrak{g}(\gamma) \\
        &= \left(\sum^n_{i=1}\mathfrak{g}(f(\gamma_i))\right) - \left(\sum^n_{i=1}\mathfrak{g}(\gamma_i)\right) \\
        &= \sum^n_{i=1}(\mathfrak{g}(f(\gamma_i) - \mathfrak{g}(\gamma_i))\\
        &= \sum^n_{i=1} \varphi_{v_i}(f).
    \end{align*}
\end{proof}

An exhaustion $\set{K_n}_{n\in\N}$ of $S$ is called a \emph{principal exhaustion} if each component of $\p K_n$ is separating and each component of $S\setminus K_n$ is of infinite-type.

\begin{lemma}\label{lem:exhhom}
    If $\set{K_n}_{n\in\N}$ is a principal exhaustion, then
    \[
    H_1^{sep}(S,\Z) = \varinjlim H^{sep}_1(K_n,\Z),
    \]
    and every nonzero element $v\in H_1^{sep}(S,\Z)$ can be written as a linear combination of homology classes represented by peripheral curves on some $K_n$.
\end{lemma}
\begin{proof}
     Since $\set{K_n}_{n\in\N}$ exhausts $S$, any nonzero element $v\in H_1^{sep}(S,\Z)$ is represented by a curve $\gamma$. Being compact, $\gamma$ lies in some $K_n$, so its homology class lies in the subgroup $H_1^{sep}(K_n,\Z)$. This shows that 
     \[
     H_1^{sep}(S,\Z) \subseteq \varinjlim H_1^{sep}(K_n,\Z). 
     \] 
     For the reverse inclusion, observe that any homology class $v_n \in H_1^{sep}(K_n,\Z)$ represents a separating curve in $S$, and hence is an element of $H_1^{sep}(S,\Z)$, giving
     \[
     H_1^{sep}(K_n,\Z)\subseteq \varinjlim   H_1^{sep}(S,\Z) . 
     \]

     Note that for any $K_n$, the group $H_1^{sep}(K_n,\Z)$ is generated by peripheral curves. This follows from the fact that on a compact surface, every separating simple closed curve is homologous to a sum of boundary components. Since every class in $H_1^{sep}(S,\Z)$ lies in some $H_1^{sep}(K_n,\Z)$, every homology class can be written as a sum of peripheral curves in some $K_n$, which proves the second fact.
\end{proof}

\begin{remark}\label{rmk:freeabe}
    The group $H_1^{sep}(S,\Z)$ is free abelian, which follows from the fact that each $H_1^{sep}(K_n,\Z)$ is free abelian, as each $K_n$ is a compact surface with boundary, and that a basis of $H_1^{sep}(K_{n+1},\Z)$ is obtained by adding the homology class of some (if any) peripheral curves of $K_{n+1}$. Therefore, $H_1^{sep}(S,\Z)$ has a basis given by the union of these nested bases:
    \[
    \bigcup_{n=1}^\infty B_n,
    \]
    where $B_n$ is the basis of $H_1^{sep}(K_n,\Z)$.
\end{remark}

By combining Lemma~\ref{lem:sum} and Lemma~\ref{lem:exhhom}, we can define $\varphi_v$ for any nonzero element of $H_1^{sep}(S,\Z)$.

\begin{proposition}
    Let $\set{K_n}_{n\in\N}$ be a principal exhaustion and $v$ be an arbitrary nonzero element of $H_1^{sep}(S,\Z)$. The homomorphism
    \[
    \varphi_v \colon \pmap(S) \longrightarrow \Z
    \]
    given by
    \[
    \varphi_v = \sum_{i=1}^n a_k\varphi_{v_k},
    \]
    is well-defined and does not depend on the choice of exhaustion.
\end{proposition}
\begin{proof}
    Let $\set{K_n}_{n\in \N}$ be a principal exhaustion and fix $n,m\in \N$ with $m>n$. The compact surfaces $K_n$ has $k_n$ boundary components and $K_m$ has $k_m$ boundary components. We can choose a maximal linearly independent set of simple peripheral homology classes $v_1,\dots,v_{k_n-1}$ so that $v_{k_n} = \sum^{k_n-1}_{i=1}v_i$. Then there exists integers $n_i \in \set{0,\dots,k_m}$ for $i\in \set{0,\dots,k_n}$ with $n_0=0$ and simple peripheral homology classes $w_1,\dots,w_{k_m}$ such that
    \[
    v_i = \sum^{n_i}_{j=n_{i-1}+1}w_j.
    \]
    An application of Lemma~\ref{lem:sum} yields
    \[
    \varphi_{v_i} = \sum^{n_i}_{j=n_{i-1}+1}\varphi_{w_j}.
    \]
    For any $v\in H_1^{sep}(S,\Z)$,
    \[
    v = \sum^{k_n}_{i=1}a_iv_i \quad \text{and}\quad v= \sum^{k_m}_{j=1}b_jw_j.
    \]
    By the above sum,
    \[
    \sum^{k_n}_{i=1}a_i\varphi_{v_i} = \sum^{k_m}_{j=1}b_j\varphi_{w_j}
    \]
    with $b_j = a_i$ for $n_{i-1}<j \leq n_i$. This shows that $\varphi_v$ is well-defined with respect to a choice of principal exhaustion.

    Now let $\set{K'_n}_{n\in\N}$ be another principal exhaustion different from $\set{K'_n}_{n\in\N}$. For any $v\in H_1^{sep}(S,\Z)$ contained in both $H_1^{sep}(K_n,\Z)$ and $H_1^{sep}(K'_m,\Z)$. By definition of exhaustions, we can choose $N\in \N$ such that $K_n$ and $K'_m$ are both contained in $K_N$. By applying a similar argument as above, it is clear that the $\varphi_v$ does not depend on the choice of exhaustion.
\end{proof}

A homomorphism $f\colon G \rightarrow H$ is said to \emph{factor through} the homomorphism $g\colon K \rightarrow H$ if there exists a homomorphism $\hat{f}\colon G\rightarrow K$ and such that $f = \hat{f}\circ g$. Similarly, a homomorphism $f\colon G \rightarrow H$ is said to \emph{factor through} $K$ if there exists homomorphisms $\hat{f}\colon G\rightarrow K$ and $\bar{f}\colon K\rightarrow H$ such that $f = \bar{f}\circ \hat{f}$.

\begin{proposition}\label{prop:homo}
    Let $S$ be an infinite-genus surface. The map
    \[
    \Phi \colon H^{sep}_1(S,\Z) \longrightarrow H^1(\pmap(S),\Z)
    \]
    given by $\Phi(v)=\varphi_v$ is a homomorphism. Moreover, $\Phi$ factors through the homomorphism $\iota_*\colon H_1^{sep}(S,\Z) \rightarrow H^{sep}_1(\hat{S},\Z)$ yielding a monomorphism 
    \[
    \hat{\Phi}\colon H_1^{sep}(\hat{S},\Z) \longrightarrow H^1(\pmap(S), \Z).
    \]
\end{proposition}
\begin{proof}
    Let $v_1,v_2$ be two homology classes in $H_1^{sep}(S,\Z)$. Then 
    \[
    v_1 = \sum^n_{i=1}a_iw_i \quad \text{and} \quad v_2=\sum^m_{j=1}b_jw_j
    \]
    for a choice of basis $\set{w_i}$.
    By definition,
    \begin{align*}
    \Phi(v_1+v_2) = \varphi_{v_1+v_2} = a_1\varphi_{w_1}+\dots+ b_m\varphi_{w_m} = \sum^n_{i=1}a_i\varphi_{w_i} +\sum^m_{j=1}b_j\varphi_{w_j}
    = \Phi(v_1)+\Phi(v_2),
    \end{align*}
    which shows that $\Phi$ is a homomorphism. By the fourth property of Proposition~\ref{prop:properties}, 
    \[
    \mathrm{ker}\Phi=\mathrm{ker}\iota_*,
    \]
    which implies that $\hat{\Phi}$ is a well-defined group monomorphism such that $\hat{\Phi}\circ\iota_*=\Phi$.
\end{proof}

\begin{remark}
    Aramayona-Patel-Vlamis showed that the homomorphism $\hat{\Phi}$ is an isomorphism, but we omit the proof for conciseness.
\end{remark}

Analogously to Remark~\ref{rmk:varphi}, the map $\varphi_v$ assigns to each $f\in \pmap(S)$, the sum of the signed genus counts between $\gamma_i$ and $f(\gamma_i)$, weighted by the corresponding coefficients, where $\gamma_i$ is a representative of the linear factor $v_i$. By Proposition~\ref{prop:properties}, it is clear that $\varphi_v(f) = 0$ if $f\in \overline{\pmap_\mathrm{c}(S)}$. Moreover, the converse is also true.

\begin{theorem}\label{thm:zero}
    Let $S$ be any surface and $f\in \pmap(S)$. Then $\varphi_v(f) = 0$ for all $v\in H_1^{sep}(S,\Z)$ if and only if $f\in \overline{\pmap_\mathrm{c}(S)}$.
\end{theorem}
\begin{proof}
    The backwards direction is the first property of Proposition~\ref{prop:properties}. 
    
    We will prove the forwards direction by contraposition. Let $f\in \pmap(S)$ and suppose that $f\notin \overline{\pmap_\mathrm{c}(S)}$. Since $f$ cannot be approximated by Dehn twists, there must be a separating simple closed curve $\gamma$ such that $f(\gamma)$ is disjoint from $\gamma$ and they cobound a compact subsurface with positive genus. This implies that $f$ has a factor of $h^n$ for some handle shift $h$ and $n\in \Z$. The homology class $v$ represented by $\gamma$ cuts $h$, and by the third property of Proposition~\ref{prop:properties}, $\varphi_v(f) \not= 0$, which proves the contrapositive.
\end{proof}

We now have everything we need to prove Theorem~\ref{thm:cohomology}.

\begin{proof}[Proof of Theorem~\ref{thm:cohomology}]
    For convenience, let $H_S= H_{sep}^1(\hat{S},\Z)$. Suppose $S = \hat{S}$. By Remark~\ref{rmk:freeabe}, $H_1^{sep}(S,\Z)$ is free abelian, let $r\in \N\cup \set{0,\infty}$ be its rank. Hence, there is a collection of simple separating homology classes $\set{v_n}_{n =1}^r$ such that
    \[
    H_1^{sep}(S,\Z) = \bigoplus_{i=1}^r \langle v_i\rangle 
    \]
    and there exists pairwise-disjoint collection of oriented simple closed curves $\set{\gamma_i}^r_{i=1}$ on $S$ such that $\gamma_i$ represents $v_i$ and any compact set of $S$ intersects at most a finite number of the $\gamma_i$.

    An element $f\colon H_1^{sep}(S,\Z)\rightarrow \Z$ of $H_S$ can be written as
    \[
    f(v) = f(\sum_{i=1}^r a_iv_i) = \sum_{i=1}^r a_if(v_i) = \sum_{i=1}^r \mathrm{pr_i}(v)f(v_i) = \sum_{i=1}^r b_i\mathrm{pr_i}(v),
    \]
    where $b_i = f(v_i)$, and $\mathrm{pr}_i \colon H_1^{sep}(S,\Z) \rightarrow \Z$ 
    is the projection map defined by 
    \[
    \mathrm{pr}_i(v_j) = \begin{cases}
        1 & \text{if } j=i,\\
        0 & \text{if } j\not=i.
    \end{cases}
    \]
    Therefore,
    \[
    H_S = \prod^r_{i=1} \langle \mathrm{pr}_i\rangle.
    \]

    We shall construct a handle shift $h_i$ that correspond to $\mathrm{pr}_i$.

    Remove pairwise-disjoint annular neighborhoods of each $\gamma_i$ from $S$. Call the resulting surface $S'$. Then each separating curve on $S'$ bounds a compact subsurface, otherwise $\set{v_i}_{i=1}^r$ would not be a basis. It is clear that $S'$ is a disjoint union of a surfaces with one end accumulated by genus. By Theorem~\ref{thm:classification}, such a surface is classified up to homeomorphism by the number of its boundary components. Denote by $Z_n$ the surface with $n$ boundary components and one end accumulated by genus. The surface $Z_n$ can be obtained by removing $n$ open disks from $\R^2$ along the line $y=0$, and by attaching handles periodically and vertically above those disks at each line $y= k$ for every $k\in \N$, similar to the construction of the model surface $\Sigma$ of a handle shift. Therefore, the surface $S'$ is a disjoint union of (possibly infinitely many) copies of $Z_{n_i}$ where $n_i \in \N \cup \set{\infty}$.

    Let $Y$ be the surface obtained from $[0,1] \times [0, \infty) \subset \R^2$ by attaching handles periodically at $y= k$ for every $k \in \N$. Observe that $Z_n$ contains $n$ disjoint embedded copies of $Y$, and for each embedded copy, there is a unique boundary component of $Z_n$ containing the image of $[0,1]\times \set{0}$ in $Y$.  

    Fix $i\in \N$ and choose two components of $\p S'$, say $b_1$ and $b_2$, each isotopic to $\gamma_i$. Denote the corresponding components of $S'$ by $X_1$ and $X_2$. Let $Y_1$ and $Y_2$ be the images of $Y$ in $X_1$ and $X_2$ respectively. Since $b_1$ and $b_2$ bound an annulus in $S$, we can connect $Y_1 \cap b_1$ and $Y_2 \cap b_2$ with a strip $T \cong [0,1]\times [0,1]$ contained in that annulus. Observe that the surface $\Sigma_i =Y_1\cup T \cup Y_2$ is homeomorphic to the model surface $\Sigma$, and is embedded in $S$. Let $h_i$ be a handle shift supported in $\Sigma_i$ so that $\varphi_{v_i}(h_i) =1$. Since the curve $\gamma_j$ is disjoint from $\Sigma_i$ for all $j\not=i$,
    \[
    h_ih_j = h_jh_i,
    \]
    and by the third property of Proposition~\ref{prop:properties}, $\varphi_{v_j}(h_i)=0$ for all $j\not=i$.
    The fact that $h_ih_j=h_jh_i$ implies that the injection map from $\prod h_i$ to $\pmap(S)$ is a group homomorphism, and thus $\prod^r_{i=1}\langle h_i\rangle < \pmap(S)$. Therefore, we can define a monomorphism
    \[
    \kappa \colon H_S \longrightarrow \pmap(S)
    \]
    by letting $\kappa(\mathrm{pr}_i) = h_i$.

    In the case that $S \not= \hat{S}$, the same process can be applied by choosing embeddings of $Y_1$ and $Y_2$ such that they avoid the planar ends, meaning that the corresponding $\Sigma_i$ also does.
    
    Take any homomorphism $\hat{\Phi}(v)\colon \pmap(S) \rightarrow \Z$ obtained by the monomorphism $\hat{\Phi}$ of Proposition~\ref{prop:properties}. Define $\bar{\Phi}(v)\colon A_S\rightarrow \Z$ by $\bar{\Phi}(v)(\pi(f))) = \varphi_v(f)$ where $\pi\colon \pmap(S) \rightarrow A_S$ is the natural projection. It is easy to check that this is a group homomorphism and by the first property of Proposition~\ref{prop:properties}, $\bar{\Phi}(v)$ is injective.
    Let $h\in \kappa(H_S)$ be nonzero. Then, $h= \prod h_i^{m_i}$, where some of the $m_i$ are nonzero. Note that $\hat{\Phi}(v_i)(h) = m_i$ and that $\bar{\Phi}(v_i)(\pi(h))=m_i$. It follows that $\pi(h) \not=0$, which implies that $\pi \circ \kappa$ is injective.

    For any $a\in A_S$, define $a^*\in H_S$ by setting $a^*(v) = \bar{\Phi}(v)(a)$. It is clear that the map $\Psi \colon A_S \rightarrow H_S$ given by $\Psi(a) = a^*$ is a homomorphism. Suppose $a^* = 0$. That means $\bar{\Phi}(v)(a)=0$ for all $v\in H_1^{sep}(\hat{S},\Z)$. By Theorem~\ref{thm:zero}, $a$ must be zero and thus $\Psi$ is injective. By construction, $\pi(h_i)^* = \mathrm{pr}_i$. As $\pi\circ \kappa$ is injective, the composition $\pi\circ\kappa\circ \Psi$ is the identity, which implies that $(\pi\circ\kappa)^{-1} =\Psi$. Therefore, $\pi\circ\kappa$ is an isomorphism.
    
\end{proof}

Consider the short exact sequence
\[
0 \longrightarrow \overline{\pmap_\mathrm{c}(S)}\xrightarrow{\hspace{0.45em} \iota \hspace{0.45em}} \pmap(S) \xrightarrow{\hspace{0.45em} \pi \hspace{0.45em}}  A_S \longrightarrow 0.
\]
Since $\pi\circ\kappa$ is an isomorphism and $\kappa$ is a monomorphism, the composition
\[
\kappa\circ \Psi\colon A_S\longrightarrow \pmap(S)
\]
is injective and thus $\pi\circ (\kappa \circ \Psi)$ is identity on $A_S$. It follows that the sequence splits and that
\[
\pmap(S) = \overline{\pmap_\mathrm{c}(S)} \rtimes A_S.
\]
Moreover, $\pi$ is an isomorphism on $\kappa(H^1_{sep}(\hat{S},\Z))$, which implies that $A_S \cong \prod^r_{i=1}\langle h_i\rangle$. It immediately follows that
\[
\pmap(S) = \overline{\pmap_\mathrm{c}(S)} \rtimes \prod^r_{i=1}\langle h_i\rangle
\]
and this proves Theorem~\ref{thm:bigmap}.
\chapter{Topological Generating Sets for Surfaces with Finitely Many Ends Accumulated by Genus}\label{chp:final}

In  Chapter~\ref{chp:gen}, we established that Dehn twists and handle shifts topologically generate the pure mapping class group. By adding mapping classes whose action on the space of ends topologically generate $\mathrm{Homeo}(\mathrm{Ends}(S),\mathrm{Ends}_\mathrm{np}(S))$ to a topological generating set of $\pmap(S)$, we can topologically generate the entire mapping class group. In this chapter, we show in detail that the mapping class group of a surface with $n$ ends accumulated by genus and no planar ends, denoted by $S(n)$, can be topologically generated by finitely many elements. Note that for $S(n)$, $\mathrm{Homeo}(\mathrm{Ends}(S),\mathrm{Ends}_\mathrm{np}(S))$ is $\sym_n$, the symmetric group on $n$ letters.

\section{Finite Topological Generation}\label{sec:fingene}

The generation of mapping class groups of finite-type surfaces by a small sets of generators and by involutions have been studied extensively over the years. Notably, Korkmaz showed in 2004 that $\map(S_g)$ with genus $g\geq 2$ is generated by two elements, one of which is a torsion element~\cite{korkmaz2004}. Brendle and Farb proved that $\map(S_g)$ can be generated by six involutions for $g\geq3$~\cite{Brendle2004}, which was improved for $g\geq8$ by Korkmaz to three involutions~\cite{Korkmaz2020}. For the infinite-type case, Huynh, proved in his doctoral thesis~\cite{huynh2022} that the mapping class groups of many infinite-type surfaces can be topologically generated by finitely many involutions. His result for surfaces with finitely many ends accumulated by genus was improved by Altunöz-Pamuk-Yıldız~\cite{apy}.

\begin{theorem}[{\cite[Theorem 1.1]{huynh2022}}]\label{thm:involgen}
    The mapping class group $\map(S(n))$ is topologically generated by at most $7$ involutions. More specifically, if $n\leq2$, then it requires $6$ involutions. For $n\geq3$, it requires $7$ involutions.
\end{theorem}

Huynh begins his proof of Theorem~\ref{thm:involgen} by showing that Dehn twists, handle shifts, and the mapping classes whose action on the space of ends generate $\sym_n$, can all be expressed as a product of finitely many involutions. He then argues that a certain subgroup generated by a finite collection of Dehn twists, handle shifts, and a rotation contains the compactly supported mapping class group. Finally, by expressing the generators as products of involutions, then carefully removing redundant involutions, he finishes the proof. 

\begin{theorem}[\cite{apy}]
    For $n\geq 3$, the mapping class group $\map(S(n))$ is topologically generated by six involutions. For $n\geq 6$, it is topologically generated by five involutions.
\end{theorem}

We shall follow the approach given in~\cite{apy}, adding details to some of the arguments. We use the model depicted in Figure~\ref{fig:inf} for $S(n)$.

\begin{figure}[H]
      \centering
      \includegraphics[width=0.80\textwidth]{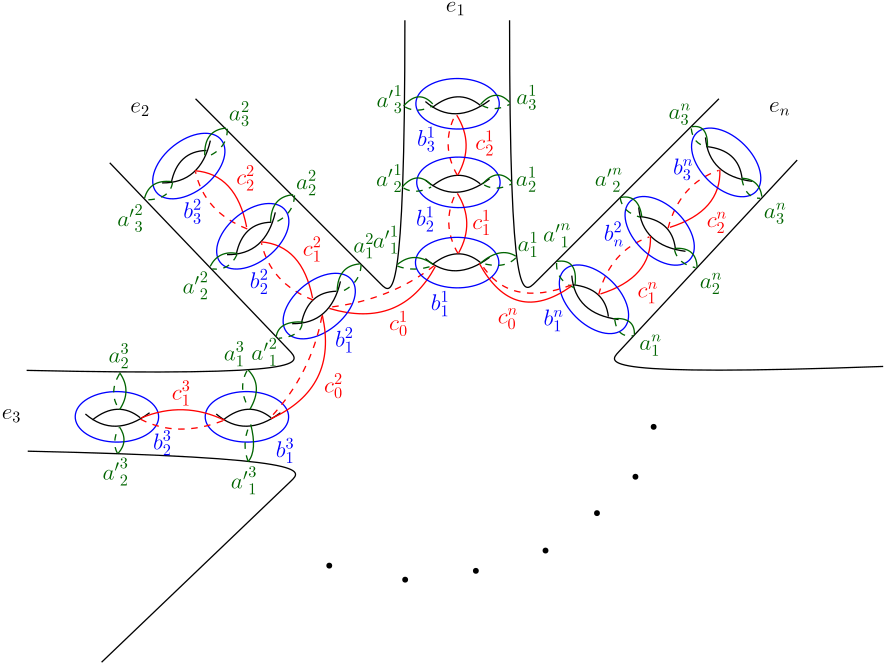}
      \caption{The model for $S(n)$}
      \label{fig:inf}
   \end{figure}

\begin{figure}[htbp]
    \centering

    \begin{subfigure}[t]{0.8\textwidth}
        \centering
        \includegraphics[width=\linewidth, height=4cm, keepaspectratio]{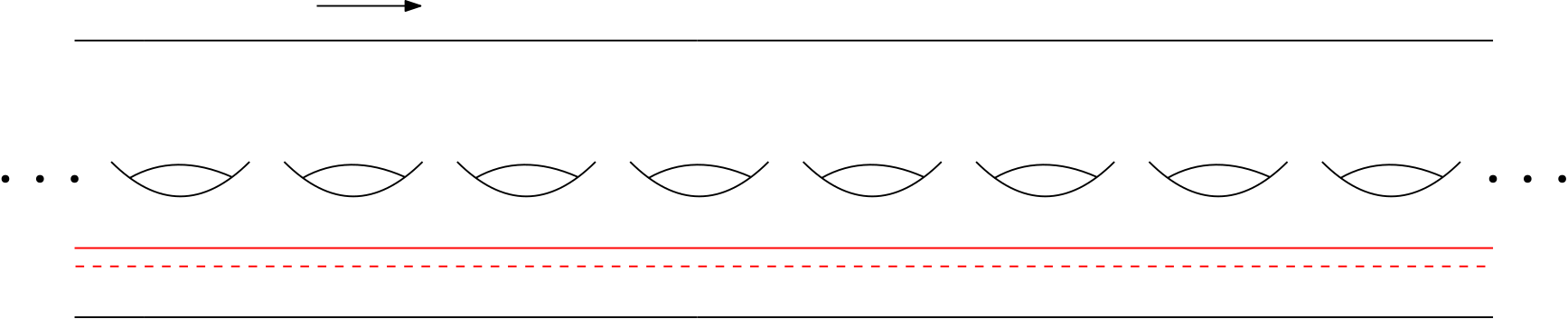}
        \caption{A Type I handle shift}
    \end{subfigure}

    \vspace{1em}

    \begin{subfigure}[t]{0.8\textwidth}
        \centering
        \includegraphics[width=\linewidth, height=4cm, keepaspectratio]{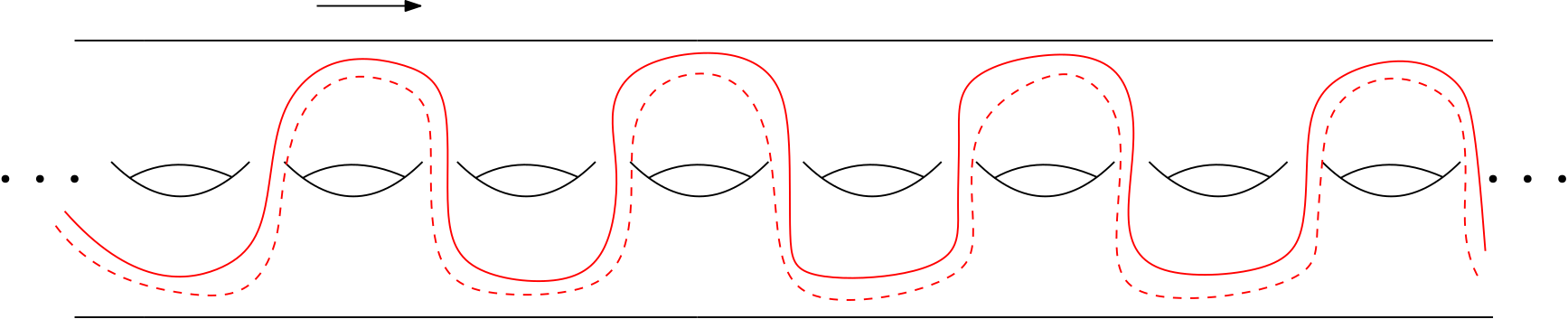}
        \caption{A Type II handle shift}
    \end{subfigure}

    \vspace{1em}

    \begin{subfigure}[t]{0.8\textwidth}
        \centering
        \includegraphics[width=\linewidth, height=4cm, keepaspectratio]{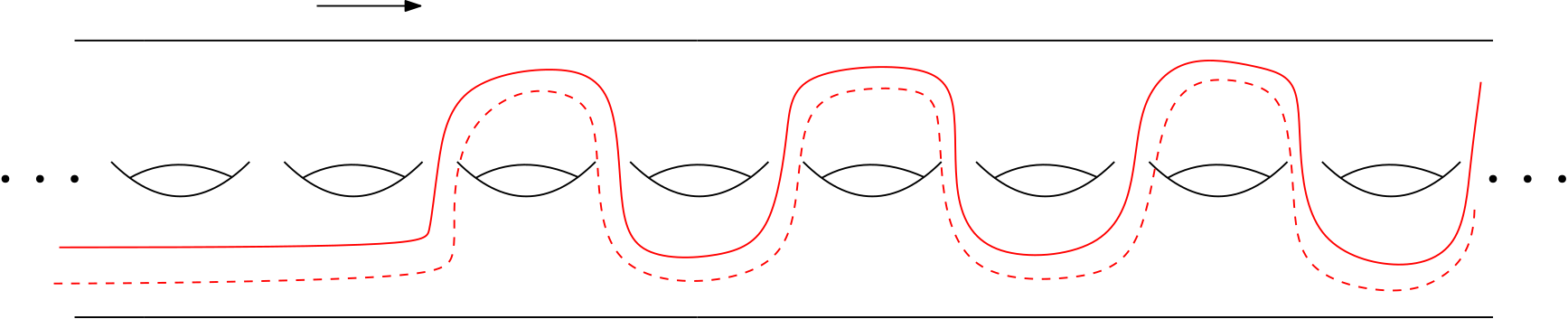}
        \caption{A Type III handle shift}
    \end{subfigure}

    \caption{Examples for every type of handle shift}
    \label{fig:hshifttypes}
\end{figure}

We use Huynh's convention of categorizing handle shifts into three types. Recall from Section~\ref{sec:hshift} that a handle shift $h$ has an attracting and a repelling end and that $h$ is induced by a handle shift $\bar{h}$ on the model surface $\Sigma$. Let $e_+= \set{E^+_i}_{i\in \N}$ be the attracting end and $e_- = \set{E^-_i}_{i\in \N}$ be the repelling end. Define the subsurfaces $\mathcal{E}^+_i = E^+_i \setminus (\Sigma \cap E^+_i)$ and $\mathcal{E}^-_i = E^-_i \setminus (\Sigma \cap E^-_i)$. We categorize handle shifts into three types:

\begin{itemize}
    \item {\textbf{Type I}}: A handle shift is of \emph{Type I} if there exists $N\in \N$ such that for all $n\geq N$, $\mathcal{E}^+_i$ and $\mathcal{E}^-_i$ has no genus.
    \item {\textbf{Type II}}: A handle shift is of \emph{Type II} if for all $n\in \N$, both $\mathcal{E}^+_i$ or $\mathcal{E}^-_i$ has infinitely many genera.
    \item {\textbf{Type III}}: A handle shift is of \emph{Type III} if for some $N\in \N$, either one of $\mathcal{E}^+_i$ or $\mathcal{E}^-_i$ has no genus and the other has infinitely many genera for all $n\geq N$.
\end{itemize}

\begin{remark}\label{rmk:typeremark}
    This classification is motivated by the observation that the model surface $\Sigma$ can be embedded in different ways inside $S(n)$, such that the handle shift $\bar{h}$ on $\Sigma$ induces different handle shifts on $S(n)$. That is, there are different handle shifts from $e_-$ to $e_+$ on $S(n)$, depending on how $\Sigma$ is embedded in $S(n)$. We can choose representatives $\set{U_n}_{n\in \N}$ and $\set{V_n}_{n\in \N}$ for $e_+$ and $e_-$ such that $S(n) \setminus (U_1 \cup V_1)$ is a surface with $2$ boundary components and $n-2$ ends accumulated by genus. Since $U_1\cup V_1$ contains countably many genera, we may index them by $\N$. Observe that $\Sigma$ can be embedded in $S(n)$ such that each genus identified with $n\in \N$ is either contained in the image of $\Sigma$ or not. It follows that each embedding of $\Sigma$ corresponds to an infinite sequence of the discrete binary set $\set{0,1}$, that is, each handle shift identifies to a point in the Cantor set. Therefore, $\Sigma$ can be embedded in $S(n)$ in uncountably many ways, which all induce nonisotopic handle shifts on $S(n)$. This implies that there are uncountably many handle shifts from $e_-$ to $e_+$.
\end{remark} 

Observe that for each $n\in \N$, the union $U_n\cup V_n$ contains countably many genera, where $U_n$ and $V_n$ are the $n$-th terms of $e_-$ and $e_+$, respectively.

Although there are many handle shifts from $e_-$ to $e_+$, we expect that only one would be enough by Remark~\ref{rem:zn-1}. This is indeed the case, as we shall soon see.

For any two simple closed curves $a$ and $b$, we denote the surface bounded by $a$ and $b$ by $S^{a,b}$. We also denote  the number of genera of the surface $S^{a,b}$ by $\mathrm{genus}(S^{a,b})$.

\begin{lemma}[{\cite[Lemma 3.11]{huynh2022}}]\label{lem:separate}
    Let $h$ be a handle shift from $e_-$ to $e_+$. Let $c$ be a separating simple closed curve that separates $e_-$. Then there exists a separating simple closed curve $\gamma$ such that $\gamma \cap h(c) = \emptyset$ and $\gamma \cap h^{-1}(c) = \emptyset$ and $\gamma$ satisfies:
    \begin{enumerate}[label=(\roman*)]
        \item $\mathrm{genus}(S^{\gamma,c}) + 1 = \mathrm{genus}(S^{\gamma,h(c)})$ and
        \item $\mathrm{genus}(S^{\gamma,c}) - 1 = \mathrm{genus}(S^{\gamma,h^{-1}(c)})$.
    \end{enumerate}
\end{lemma}
\begin{proof}
    Since $c$ is separating, $S \setminus c$ has two connected components, say $E_-$ and $E_+$. Suppose $E_-$ contains $e_-$ and $E_+$ contains $e_+$. Consider the embedded model surface $\Sigma_h$ in $S$ associated to $h$. We can identify this subsurface with an embedded copy of the model surface $\Sigma$ in $\R^3$.
    Consider $\Sigma_h \cap c$. Since $c$ is compact, $\Sigma_h \cap c$ is a finite disjoint union of separating arcs with endpoints on $\p \Sigma$. There exists an $L \in \R$ that is the minimum of all the $x$-coordinates of the elements of $\Sigma_h \cap c$ considered as a subspace of $\R^3$.

    For any $n\in \Z$, let $\gamma_n$ be the vertical line $\set{n+\frac{1}{2}} \times [-1,1]$. Pick an $N \in \Z$ such that $N + \frac{1}{2} < L -1$. Since $E_- \setminus (\Sigma_h \cap E_-)$ is path-connected, we can extend $\gamma_N$ to a separating simple closed curve on $S$ by connecting its endpoints on the complement of $\Sigma_h$. Call this curve $\gamma$.

    By construction, $\gamma \cap c = \emptyset$ and $\gamma \cap h(c) = \emptyset$. The fact that $N + \frac{1}{2} < L -1$ implies that $\gamma \cap h^{-1}(c)$. Since $h$ shifts every genus on $\Sigma_h$ by one, we have 
    \[
    \mathrm{genus}(S^{\gamma,c})+1 = \mathrm{genus}(S^{\gamma,h(c)}).
    \]
    Therefore,
    \begin{align*}
        \mathrm{genus}(S^{\gamma,h(c)}) &= \mathrm{genus}(S^{\gamma,h(\gamma)} \cup S^{h(\gamma),h(c)}) \\
        &=\mathrm{genus}(S^{\gamma,h(\gamma)}) + \mathrm{genus}(S^{h(\gamma),h(c)}) \\
        &=1 + \mathrm{genus}(S^{\gamma,c}),
    \end{align*}
    and
\begin{align*}
        \mathrm{genus}(S^{\gamma,c}) &= \mathrm{genus}(S^{h^{-1}(\gamma),h^{-1}(c)}) \\
        &=\mathrm{genus}(S^{h^{-1}(\gamma),\gamma}) + \mathrm{genus}(S^{\gamma,h^{-1}(c)}) \\
        &=1 + \mathrm{genus}(S^{\gamma,h^{-1}(c)}),
    \end{align*}
\end{proof}

\begin{remark}\label{rmk:separate}
    The curve $\gamma$ obtained by the construction in the proof above lies on $E_-$ because $L$ is chosen to be the minimum of the $x$-coordinates of the elements of $\Sigma_h \cap c$. By letting $L$ be the maximum of such $x$-coordinates and picking $N > L+1$, it is possible to construct a $\gamma$ lying in $E_+$ satisfying 
    \begin{enumerate}[label=(\roman*)]
        \item $\mathrm{genus}(S^{\gamma,c}) - 1 = \mathrm{genus}(S^{\gamma,h(c)})$ and
        \item $\mathrm{genus}(S^{\gamma,c}) + 1 = \mathrm{genus}(S^{\gamma,h^{-1}(c)})$.
    \end{enumerate}
\end{remark}

Lemma~\ref{lem:separate} shows us the existence of separating simple closed curves such that a handle shift either increases or decreases the number of genera bounded by it and a fixed separating simple closed curve. Our ultimate goal is to show that any two handle shifts with the same attracting and repelling ends can be approximated by elements of $\pmap_\mathrm{c}(S(n))$. To this end, we now consider a composition of two such handle shifts, $h_1$ and $h_2$, and analyze their effect on the genus count between separating curves.

\begin{lemma}[{\cite[Lemma 3.12]{huynh2022}}]\label{lem:separate2}
    Let $h_1$ and $h_2$ be two handle shifts from $e_-$ to $e_+$ and let $h$ be any one of $h_1\circ h_2^{-1}, h_2^{-1}\circ h_1, h_2 \circ h_1^{-1}$ or $h_1^{-1}\circ h_2 $. Let $c$ be a separating simple closed curve that separates $e_+$. Then there exists a separating simple closed curve $\gamma$ disjoint from $c$ and $h(c)$ such that $\mathrm{genus}(S^{\gamma,c}) = \mathrm{genus}(S^{\gamma,h(c)})$.
\end{lemma}
\begin{proof}
    Without loss of generality, we prove the theorem for $h= h_2^{-1}\circ h_1$ and $h' = h_2\circ h_1^{-1}$, since the same argument is symmetric between $h_1$ and $h_2$. \\
    We use Lemma~\ref{lem:separate} on $c$ and $h_1$ to find a curve $\delta$ such that 
    \[
    \mathrm{genus}(S^{\delta,c}) + 1 = \mathrm{genus}(S^{\delta,h_1(c)}),
    \]
    and 
    \[
    \mathrm{genus}(S^{\delta,c}) - 1 = \mathrm{genus}(S^{\delta,h_1^{-1}(c)}).
    \]
    We then apply Lemma~\ref{lem:separate} on $\delta$ and $h_2$ to find a curve $\gamma$ with 
    \[
    \mathrm{genus}(S^{\gamma,\delta}) + 1 = \mathrm{genus}(S^{\gamma,h_2(\delta)})
    \]
    and 
    \[
    \mathrm{genus}(S^{\gamma,\delta}) - 1 = \mathrm{genus}(S^{\gamma,h_2^{-1}(\delta)}).
    \] 
    Note that 
    \begin{align*}
        \mathrm{genus}(S^{\gamma,h(c)}) &= \mathrm{genus}(S^{\gamma,h_2^{-1}(\delta)}) + \mathrm{genus}(S^{h_2^{-1}(\delta),h(c)}) \\
        &=\mathrm{genus}(S^{\gamma,\delta}) - 1 + \mathrm{genus}(S^{\delta,h_1(c)}) \\
        &=\mathrm{genus}(S^{\gamma,\delta}) - 1 + \mathrm{genus}(S^{\delta,c}) + 1\\
        &=\mathrm{genus}(S^{\gamma,c}).
    \end{align*}
    Similarly,
    \begin{align*}
        \mathrm{genus}(S^{\gamma,h'(c)}) &= \mathrm{genus}(S^{\gamma,h_2(\delta)}) + \mathrm{genus}(S^{h_2(\delta),h'(c)}) \\
        &= 1 + \mathrm{genus}(S^{\gamma,\delta}) + \mathrm{genus}(S^{\delta,h_1^{-1}(c)}) \\
        &=1+\mathrm{genus}(S^{\gamma,\delta}) + \mathrm{genus}(S^{\delta,c}) - 1\\
         &=\mathrm{genus}(S^{\gamma,c}).
    \end{align*}
\end{proof}

\begin{remark}
    Because the construction depends on Lemma~\ref{lem:separate}, $\gamma$ lies in $E_-$. However, by using the alternative construction outlined in Remark~\ref{rmk:separate}, it is also possible to construct a curve $\gamma$ lying in $E_+$. 
\end{remark}

We can finally show that any two handle shifts attracted and repelled by the same ends differ by a mapping class that can be approximated by Dehn twists.

\begin{theorem}[{\cite[Lemma 3.13]{huynh2022}}]\label{thm:handledehn}
 Let $h_1$ and $h_2$  be two handle shift from $e_-$ to $e_+$. Then any one of $h_1\circ h_2^{-1}, h_2^{-1}\circ h_1, h_2 \circ h_1^{-1}$ or $h_1^{-1}\circ h_2 $ can be approximated by Dehn twists.
\end{theorem}
\begin{proof}
Without loss of generality, we prove the theorem for $h =h^{-1}_2\circ h_1$.

Let $\set{K_n}_{n\in \N}$ be an exhaustion of $S(n)$ and let $F_n$ be a surface containing $K_n$ and $h_2^{-1}\circ h_1(K_n)$ such that all of the boundary components of $F_n$ are separating. Let the boundary components of $F_n$ be $c_{1}^n,c^n_2,\dots, c^n_k$. For sufficiently large $n$, each $c^n_i$ separates one end and after relabeling, we may assume $c_1^n$ separates $e_-$ and $c^n_2$ separates $e_+$. Since the supports of $h_1$ and $h_2$ are on some subsurface containing only $e_-$ and $e_+$, we can also assume that $h$ is identity on $c^n_3,\dots,c^n_k$ for sufficiently large $n$.

By Theorem~\ref{lem:separate2} and Remark~\ref{rmk:separate}, there are separating simple closed curves $\gamma_1^n$ and $\gamma_2^n$ such that 
\[
\mathrm{genus}(S^{\gamma_1^n,c_1^n}) = \mathrm{genus}(S^{\gamma^n_1,h(c_1^n)}),
\]
and
\[
\mathrm{genus}(S^{\gamma^n_2,c_2^n}) = \mathrm{genus}(S^{\gamma^n_2,h(c_2^n)}),
\]
where $\gamma_1$ is on the component of $S\setminus c_1 \cup c_2$ containing $e_+$ and $\gamma_2$ is on the component containing $e_-$.

Let the surface $G_n = F_n \cup U_1 \cup U_2$, where $U_i$ is the surface bounded by $c^n_i$ and $\gamma^n_i$. Notice that $G_n$ is bounded by $\gamma^n_1,\gamma^n_2,c^n_3,\dots$ and $c^n_k$.

We shall construct $f_n\colon G_n \rightarrow G_n$. Define $f_n|_{F_n}$ to be $h|_{F_n}$. By construction, 
\[
\mathrm{genus}(U_i) = \mathrm{genus}(S^{\gamma^n_i,c^n_i}) = \mathrm{genus}(S^{\gamma^n_i,h(c^n_i)}).
\]
As they share both the number of genera and the number of boundary components, $U_i$ is homeomorphic to $S^{\gamma^n_i,h(c^n_i)}$. It follows that $f_n$ can be extended to $G_n$ by defining $f_n|_{\p G_n}$ to be the identity function.  Because $f_n|_{\p G_n}$ is identity, we can extend $f_n$ to be identity on $S(n) \setminus G_n$. Since $G_n$ is a finite-type surface, $f_n|_{G_n}$ is a product of Dehn twists, which implies that $\lim_{n\rightarrow\infty} f_n = h$ is approximated by Dehn twists.

\end{proof}

\begin{remark}
     Theorem~\ref{thm:handledehn} is a direct proof that any two handle shifts with the same attracting and repelling ends projects to the same element in the quotient group $A_S$ defined in Chapter~\ref{chp:gen}. 
\end{remark}

We now turn our attention to the compactly supported mapping class group. Recall from Section~\ref{sec:topogen} that the compactly supported mapping class group is the direct limit of the pure mapping class groups of essential subsurfaces. By Theorem~\ref{thm:humphries}, the pure mapping class group of a finite-type surface is generated by finitely many Dehn twists about non-separating simple closed curves. Therefore, it is reasonable to believe that the Dehn twists about all such non-separating simple closed curves generate the compactly supported pure mapping class group of infinite-type surfaces. However, we can do better by including elements outside the compactly supported mapping class group.

\subsection{More than Two Ends}

In this subsection we shall assume that $n\geq 3$, and treat the cases for $n=1$ (Loch Ness Monster) and $n=2$ (Jacob's Ladder) separately later. Consider the families of simple closed curves $a_i^j,{a'}_i^j,b_i^j$ and $c^j_{i-1}$ shown in Figure~\ref{fig:inf}. The upper index $j$ indicates the end on which the curve lies and the lower index $j$ specifies its location along that end. We denote by $h_{i,i+1}$ a Type I handle shift from the $i$-th end to the $(i+1)$-th end and by $R$ a counter-clockwise rotation of $S(n)$ by an angle of $\frac{2\pi}{n}$.

\begin{proposition}\label{prop:countableset}
    The subgroup generated by the set
    \[
    \set{T_{a_i^j},T_{{a'}_i^j},T_{b_i^j},T_{c^j_{i-1}}, h_{1,2},R}
    \]
    contains $\pmap_\mathrm{c}(S(n))$.
\end{proposition}
\begin{proof}
Let $G$ be the subgroup generated by the above set. Consider the sequence of compact subsets $\set{K_n}_{n\in\N}$ where $K_n= S^n_{kn}$ for every $n\in\N$, as shown in Figure~\ref{fig:exhaus}. It is clear that $S(n)=\bigcup_{i=1}^\infty K_i$. Therefore, $\set{K_n}$ exhausts $S(n)$ and thus
\[
\pmap_\mathrm{c}(S(n)) = \varinjlim \map(K_n).
\]

\begin{figure}[htbp]
      \centering
      \includegraphics[width=0.8\textwidth]{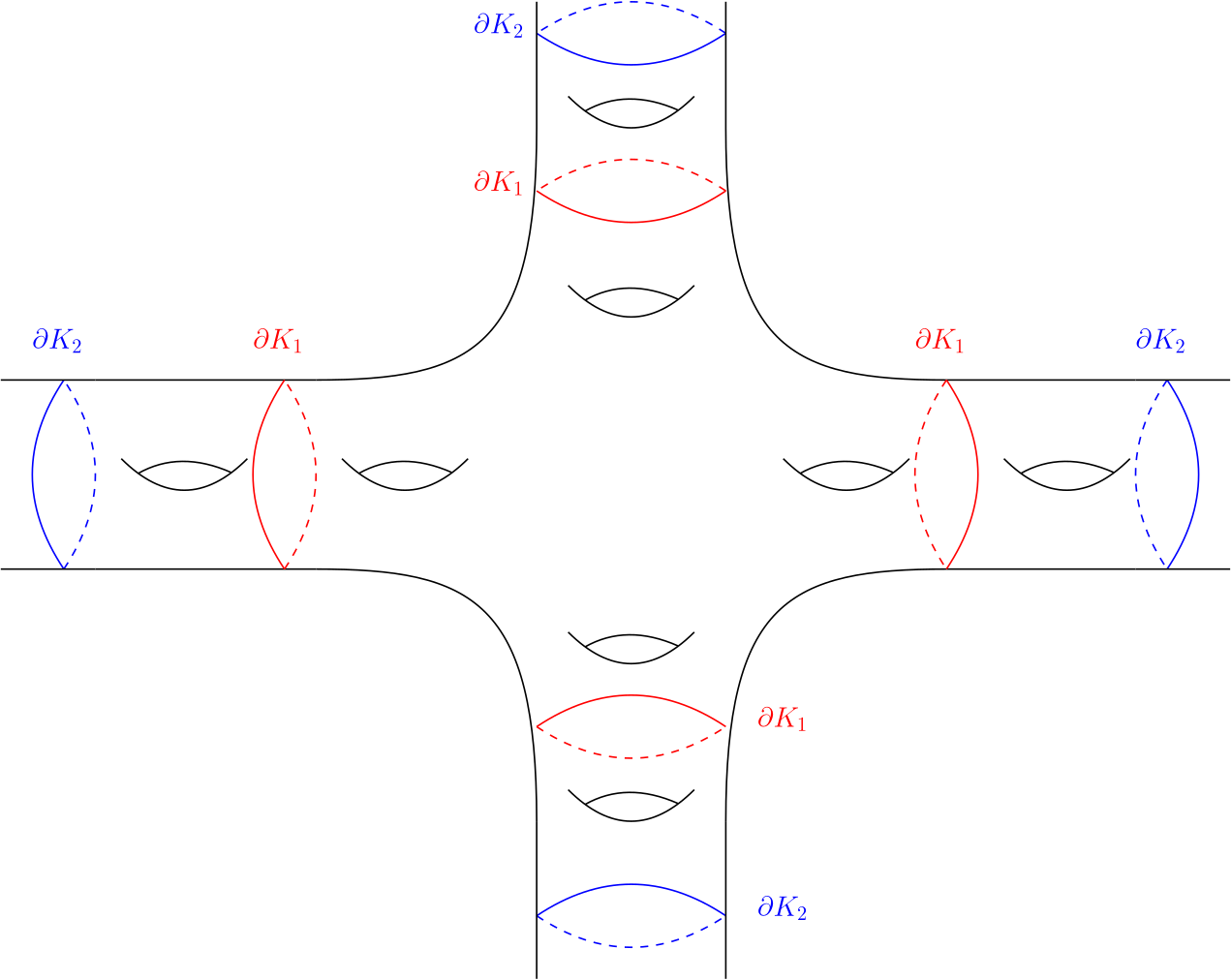}
      \caption{The sequence $K_n$ for $S(4)$}
      \label{fig:exhaus}
   \end{figure}    

We will show that for all $n\in \N$, $G$ includes the generators of $\map(K_n)$. Since a finite-type surface is determined up to homeomorphism by the number of boundary components and genera. Therefore, we can modify the model shown in Figure~\ref{fig:bdrygen} by a homeomorphism $\iota_n$ to the model of $K_n$. Note that $\iota_n$ is trivial as a mapping class, since it is just an embedding change. We can think of $\iota_n$ as sliding the position of the boundary components and the genera. 

\begin{figure}[htbp]
      \centering
      \includegraphics[width=\textwidth]{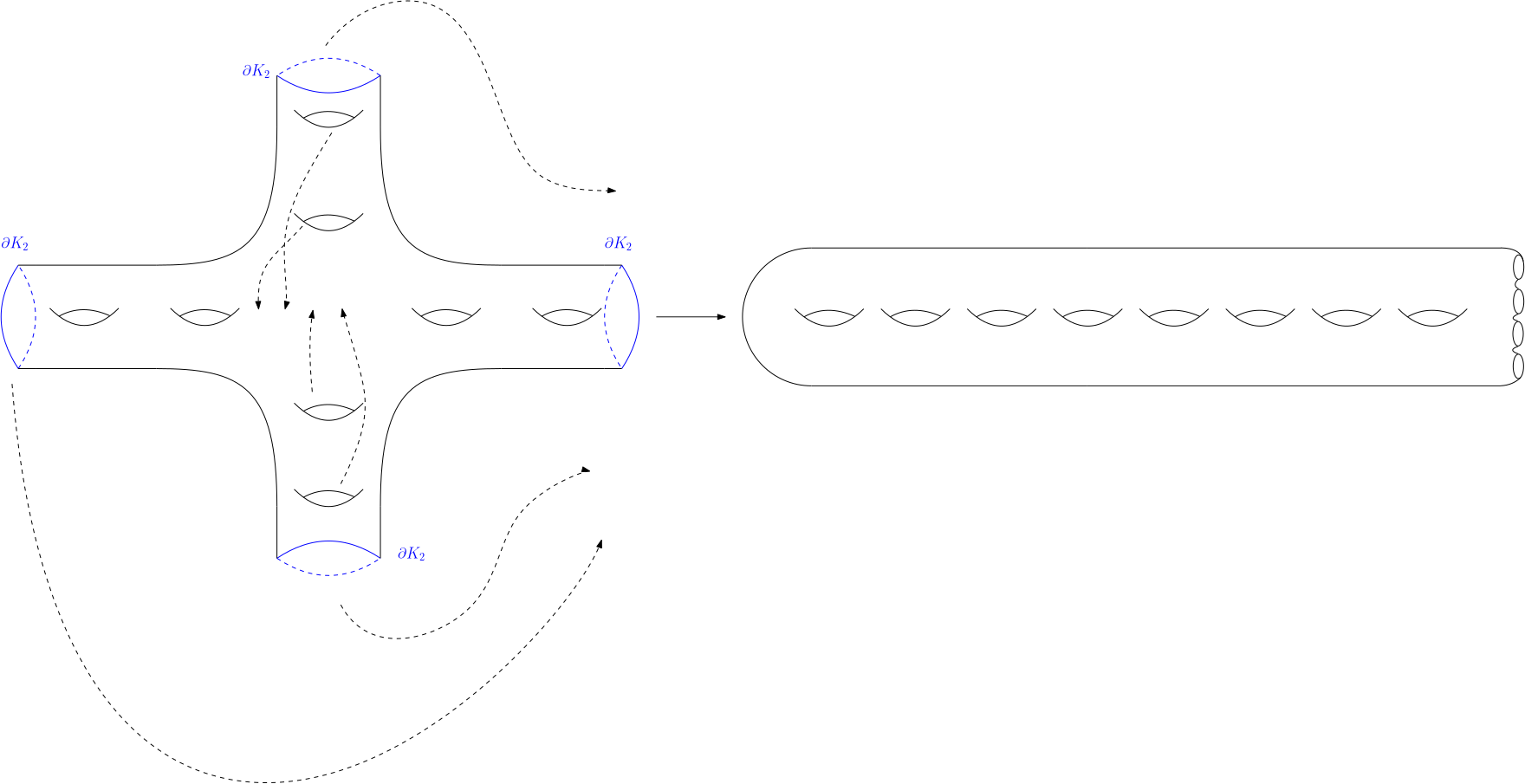}
      \caption{The homeomorphism $\iota_2^{-1}\colon K_2 \rightarrow K_2$ for $S(4)$}
      \label{fig:iota2}
   \end{figure}    

Recall that the handle shift $h_{1,2}$ is supported on an embedded copy of the model surface $\Sigma$ between the first and second ends. This copy of $\Sigma$ rotates alongside $S(n)$ under the rotation $R$ so that it lies between the second and third ends. Therefore,
\[
h_{2,3} = Rh_{1,2}R^{-1} \in G.
\]
Repeating the same argument, we can see that $G$ contains all handle shifts of the form $h_{i,i+1}$. It is clear that
\[
h_{i,k} = h_{i+(k-1),i+k}h_{i+1,i+2}h_{i,i+1},
\]
and thus every handle shift of the form $h_{i,j}$ is in $G$.

Observe that the images of $\beta_i$ curves in Figure~\ref{fig:bdrygen} under $\iota_n$ are the $b_j$ curves on $K_n$. Any $\gamma_i$ under $\iota_n$ is either a $c_i^j$ curve, or of the form as shown in Figure~\ref{fig:diff}. Observe in that case that there exists a handle shift $h_{k,k+1}$ and a curve $c_0^k$ such that $\iota_n(\gamma_i) = h_{k,k+1}^{-n}(c_0^k)$, so that
\[
T_{\iota_n(\gamma_i)} = T_{h_{k,k+1}^{-n}(c_0^k)} = h_{k,k+1}^{-n} T_{c^k_0}h_{k,k+1}^{n}
\]
by the conjugation property of Dehn twists, meaning that $G$ contains generators corresponding to each $\gamma_i$. Finally, observe that the images of $\alpha_i$ under $\iota_n$ either $a_j$ or $a'_j$ curves, or of the form shown as Figure~\ref{fig:diff}. This time, there exists a product of handle shifts $H$ and a curve $a^k_1$ such that $\iota_n(\alpha_i) = H(a^k_1)$ and
\[
T_{\iota_n(\alpha_i)} = T_{H(a_1^k)} = HT_{a^k_1}H^{-1} \in G.
\]
We have shown that for all $n\in \N$, the generators for $\map(K_n)$ are inside $G$, which mean that $G$ contains $\pmap_\mathrm{c}(S(n))$.

\begin{figure}[htbp]
      \centering
      \includegraphics[width=0.8\textwidth]{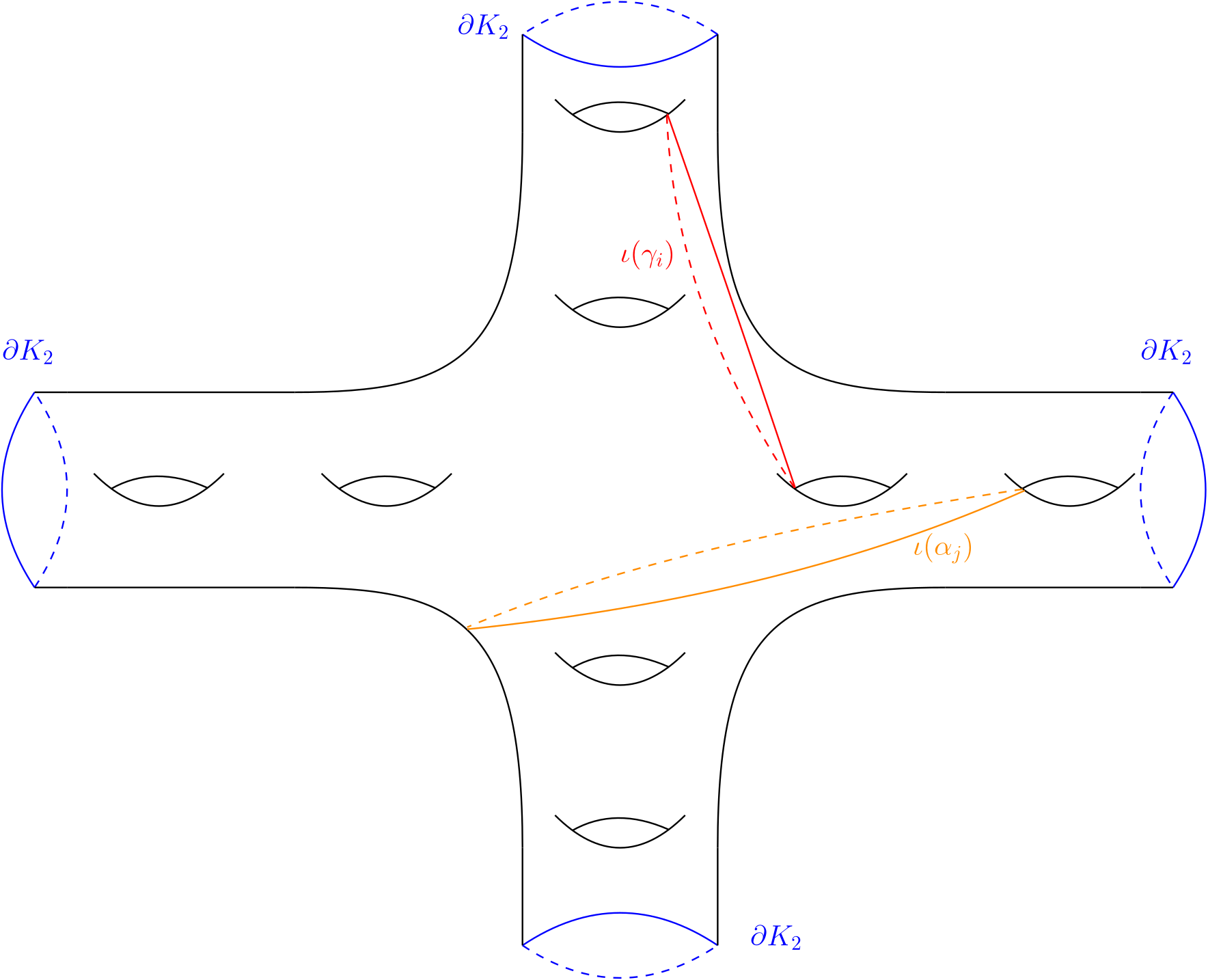}
      \caption{The curves $\iota_2(\gamma_i)$ and $\iota_n(\alpha_j)$ for some $i,j$}
      \label{fig:diff}
   \end{figure}    

\end{proof}

We now have a countable generating set for a subgroup containing $\pmap_\mathrm{c}(S(n))$. We can simplify this set further to obtain a finite generating set.

Observe that
\[
h_{i,i+1}(b^i_1) = b^{i+1}_1, \quad h_{i,i+1}(c_0^i)=c^{i+1}_1, \quad \overline{h_{i,i+1}}(a^{i+1}_1) = {a'}^{i}_1,
\]
for all $1\leq i \leq n$, and
\[
h_{i,i+1}(b^i_j) = b^i_{j-1}, \quad h_{i,i+1}(c_{j-1}^i)=c^i_{j-2}, \quad h_{i,i+1}(a^i_j) = a^i_{j-1}, \quad h_{i,i+1}({a'}^i_j) = {a'}^i_{j-1}
\]
for all $1\leq i \leq n$ and $j> 1$. 
Using this observation along with the fact that all four families of curves are preserved as families under the rotation $R$, we can find a finite subset of the set in Proposition~\ref{prop:countableset} that topologically generates the same subgroup. In particular, there is a finite set that topologically generates a subgroup containing the compactly supported pure mapping class group. 

\begin{lemma}\label{lem:finiteset}
    The subgroup generated by the set
    \[
    \set{T_{a_1^1},T_{b_1^1},T_{c^1_0}, h_{1,2},R}
    \]
    contains $\pmap_\mathrm{c}(S(n))$.
\end{lemma}
\begin{proof}
    Let $G$ be the subgroup generated by the set
    \[
    \set{T_{a_1^1},T_{b_1^1},T_{c^1_0}, h_{1,2},R}.
    \]
    By Proposition~\ref{prop:countableset}, it is enough to show that $T_{a_i^j},T_{{a'}_i^j},T_{b_i^j}$ and $T_{c_{i-1}^{j}}$ are in $G$ for all $i$ and $j$. Recall that since $R$ and $h_{1,2}$ are both in $G$, the handle shift $h_{i,i+1}$ is in $G$ for all $i$. \\
    Using the conjugation property of Dehn twists and the observations above, 
    \[
    T_{{a'}_1^{i}} = \overline{(h_{i,i+1})}T_{a^{i+1}_1}(h_{i,i+1}) \in G.
    \]
    Observe that,
    \[
    T_{b_1^{i+1}} = RT_{b^i_1}\overline{R},
    \]
    and 
    \[
    T_{b^i_j} = \overline{(h_{i,i+1})^{j-1}}T_{b^i_1}(h_{i,i+1})^{j-1}
    \]
    so that $T_{b^j_i}$ is in $G$ for all $i$ and $j$. By similar arguments, every $T_{a_i^j},T_{{a'}^j_i}$ and $T_{c_{i-1}^j}$ is also in $G$, and we are done.
    
\end{proof}

To topologically generate the entire mapping class group, we need mapping classes whose images under $\pi$ generates the symmetric group $\sym_n$. Luckily, it is a known fact that $\sym_n$ is generated by two elements. It is an exercise in abstract algebra to prove the following theorem so we skip the proof, see [\citenum{lang}, Exercise I.38] for an outline for the proof.

\begin{lemma}\label{lem:symgen}
    For $n\geq 3$, the symmetric group $\sym_n$ is generated by an $n$-cycle and a $2$-cycle.
\end{lemma}

\begin{remark}\label{rem:ncycle}
    Since the rotation $R$ permutes the ends by moving each end to the next one, it projects onto the $n$-cycle $(12\dots n)$ in $\sym_n$ under the homomorphism $\pi$ of Section~\ref{sec:topogen}.
\end{remark}

By Remark~\ref{rem:ncycle}, we already have an $n$-cycle that we use to obtain every compactly supported mapping class. By adding a mapping class that projects to a $2$-cycle to the set in Proposition~\ref{prop:countableset}, we can generate the entire mapping class group. For convenience later, we shall add the involution $\tau_1$ shown in Figure~\ref{fig:tau1}.

We have shown the following:

\begin{theorem}\label{thm:finset}
    The set
    \[
    \set{T_{a_1^1},T_{b_1^1},T_{c^1_0}, h_{1,2},R,\tau_1}
    \]
    topologically generates $\map(S(n))$.
\end{theorem}

\subsection{Loch Ness Monster}\label{sec:lochness}

For the Loch Ness Monster surface, we shall use the model shown in Figure~\ref{fig:lochness} for convenience, instead of the one shown in Figure~\ref{fig:infsurfaces}. Throughout this subsection, let $S$ denote the Loch Ness Monster surface. The lower indices denote the location of a curve on $S$ where a positive index means a curve is on the lower row and a negative index means it is on the top row. Note that the families $a_i,b_i$ have no curves indexed $0$. This is for symmetry purposes, as will be clear soon.

\begin{figure}[htbp]
      \centering
      \includegraphics[width=0.8\textwidth]{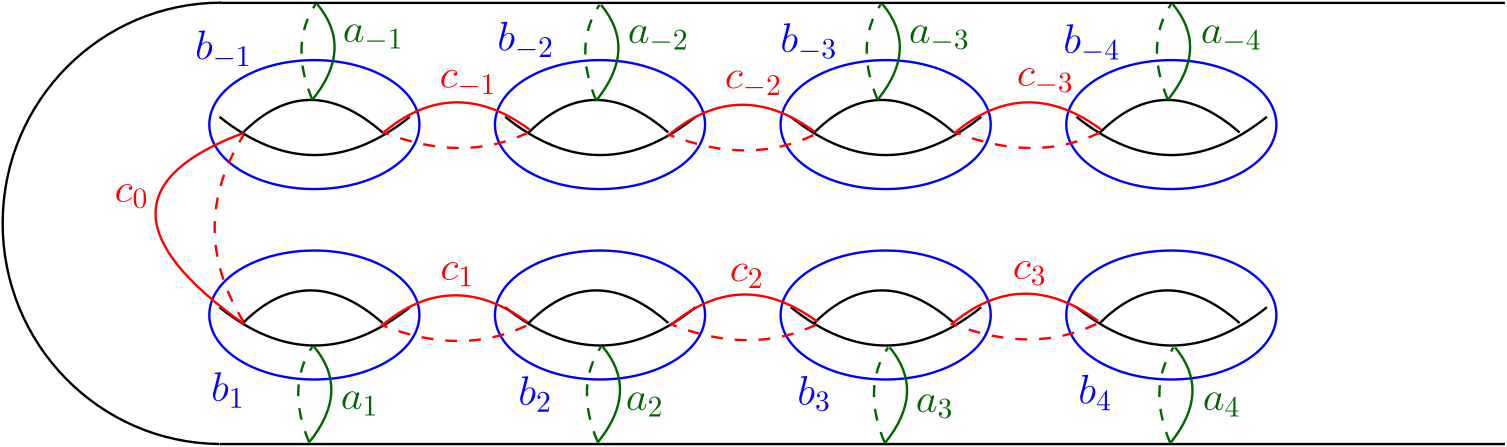}
      \caption{The model for the Loch Ness Monster surface $S$}
      \label{fig:lochness}
   \end{figure}

\begin{proposition}\label{prop:n1gen}
     The pure mapping class group $\pmap(S)$ is topologically generated by the set
     \[
     \set{T_{a_i}, T_{b_i}, T_{c_j} \mid i \in \mathbb{Z}\setminus\{0\}, j \in \mathbb{Z}}
     \] 
     as shown in Figure~\ref{fig:lochness}.
\end{proposition}
\begin{proof}
    Consider the sequence of surfaces $\set{K_n}_{n\in\N}$ of $S$ where $K_n = S^1_{2n}$. It is clear that $\set{K_n}$ exhausts $S$ and that any essential subsurface lies inside some $K_n$. By Definition~\ref{def:compmap},
    \[
    \pmap_\mathrm{c}(S) = \varinjlim \map(K_n).
    \]
    Therefore, we shall show that the generators for each $K_n$ is in our generating set.

    For each $n\in \N$, there is a homeomorphism $\iota_n\colon K_n \rightarrow K_n$ isotopic to the identity changing the model of $K_n$ from the one depicted in Figure~\ref{fig:iotaloch} to the one in Figure~\ref{fig:bdrygen}. The homeomorphism $\iota_n$ can be described as moving the position of the genera. The action of $\iota_n$ is shown in Figure~\ref{fig:iotaloch}

\begin{figure}[htbp]
      \centering
      \includegraphics[width=0.8\textwidth]{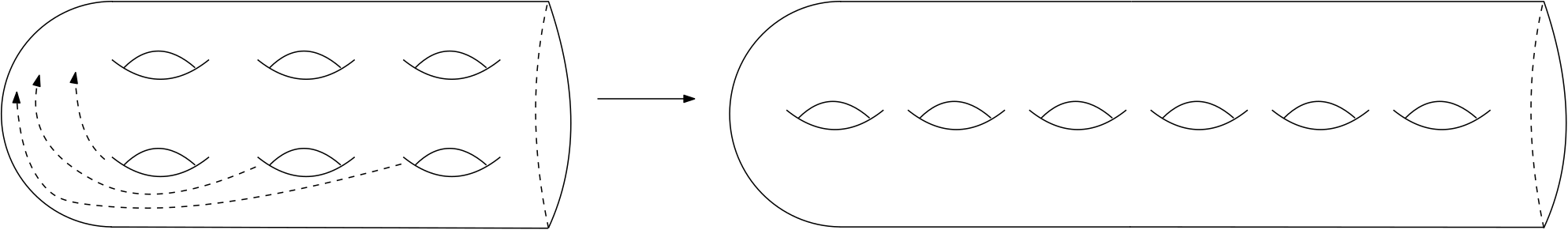}
      \caption{The homeomorphism $\iota_3\colon K_3 \longrightarrow K_3$}
      \label{fig:iotaloch}
   \end{figure}    
    It is clear that for every $n\in \N$, the generators $\alpha_i,\beta_j$ and $\gamma_k$ in Figure~\ref{fig:bdrygen}, are the images of some $a_m,b_l$ and $c_o$, respectively, under $\iota_n$. It follows that the Dehn twists
    \[
    \set{T_{a_i},T_{b_i},T_{c_j} \mid i \in \set{-n,\dots,-1,-1,\dots n}, j\in \set{-n+1,\dots,-1,0,1,\dots, n-1}}
    \]
    generate $\map(K_n)$, which implies that 
    \[
     \{T_{a_i}, T_{b_i}, T_{c_j} \mid i \in \mathbb{Z}\setminus\{0\}, j \in \mathbb{Z}\}
     \] 
     topologically generates $\map(S)$.
\end{proof}

Since the Loch Ness monster surface has a single end, $\pmap(S) = \map(S)$, and thus we have a countable generating set for $\map(S)$. We can do better by introducing a handle shift. We can embed the model surface $\Sigma$ in $S$ by "bending" $\Sigma$ in $\R^3$ so that the end lying on the left side of $\R^2$ is moved to the right side and on above the other end. The resulting embedding is shown in Figure~\ref{fig:n1shift}.

\begin{figure}[htbp]
\centering
      \includegraphics[width=0.8\textwidth]{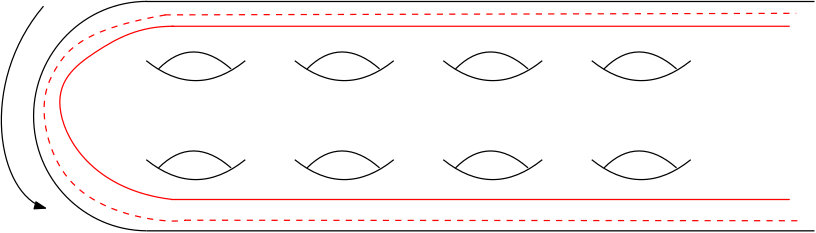}
      \caption{An embedding of $\Sigma$ in the Loch Ness monster surface}
      \label{fig:n1shift}
   \end{figure}    

\begin{remark}
    Since $S$ has a single end, the attracting and repelling ends of any handle shift on $S$ are the same. This is yet another example of ends not being preserved under embeddings, since $\Sigma$ has $2$ ends.
\end{remark}

Even though Dehn twists are enough to topologically generate $\map(S)$ by Proposition~\ref{prop:dehngen}, introducing a handle shift enables us to go from a countable generating set to a finite one. 

\begin{theorem}\label{thm:n1gen}
    The mapping class group $\map(S)$ is topologically generated by the set
     \[
     \set{T_{a_1}, T_{b_1}, T_{c_0},H}
     \] 
     where $H$ is the handle shift depicted in Figure~\ref{fig:n1shift}.
\end{theorem}
\begin{proof}
    Let $G$ be the subgroup topologically generated by the above set. Observe that $H^m$ sends the curves $a_1,b_1$ to $a_{1+m},b_{1+m}$ respectively for positive $m$ and to $a_{-m},b_{-m}$ for negative $m$ while it sends $c_{0}$ to $c_{m}$ for all $m$. By the conjugation property of Dehn twists, it is clear that every Dehn twist $T_{a_i},T_{b_j}$ and $T_{c_k}$ is contained in $G$, which implies by Proposition~\ref{prop:n1gen} that $G= \map(S)$.
\end{proof}

\subsection{Jacob's Ladder}

Let $S$ denote the Jacob's Ladder surface throughout this subsection. The lower indices denote the location of a curve on $S$ where a positive index means a curve is on the right to the curve $c_0$ and a negative index means it is on the left. Note that the families $a_i,{a'}_i$ and $b_i$ have no curves indexed $0$ like the Loch Ness Monster surface. This is again for symmetry purposes.

\begin{figure}[htbp]
\centering
      \includegraphics[width=0.8\textwidth]{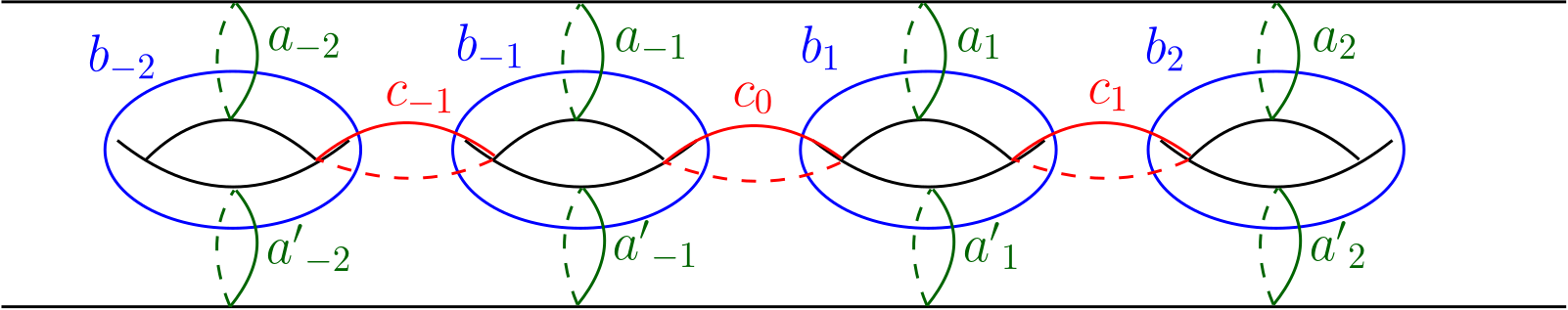}
      \caption{The model we use for the Jacob's Ladder surface}
      \label{fig:jacobs}
   \end{figure}    

\begin{proposition}
    The pure mapping class group $\pmap(S)$ is topologically generated by the set
    \[
     \set{T_{a_i},T_{{a'}_i}, T_{b_i}, T_{c_j} \mid i \in \mathbb{Z}\setminus\{0\}, j \in \mathbb{Z}}
     \] 
\end{proposition}
The proof is analogous to that of Proposition~\ref{prop:n1gen} and we omit it. 

Similar to the Loch Ness Monster surface, we can introduce a handle shift to obtain a finite topological generating set for $\pmap(S)$. To topologically generate the entire mapping class group, we need to add a mapping class that permutes its end, since the Jacob's Ladder surface has $2$ ends. A rotation by $\pi$ radians around the axis passing through the center of the curve $c_0$ is such a mapping class. 

\begin{theorem}
    The set
    \[
    \set{T_{a_1},T_{{a'}_1}, T_{b_1}, T_{c_0}, H, R}
    \]
    where $H$ is a handle shift from the left end to the right end and $R$ is the rotation described above topologically generates $\map(S)$.
\end{theorem}
The proof is analogous to the proof of Theorem~\ref{thm:n1gen}, and we skip it.

\section{Topological Generation by Involutions}

In this section, we focus on involution generators. Our strategy is to construct a finite set of involutions capable of generating the specific Dehn twists described in Section~\ref{sec:fingene}; generating these twists is sufficient to generate the entire group $\map(S(n))$.

For notational convenience, we adopt the conventions established in~\cite{apy} from now on. We denote the Dehn twists about the curves $a_i^j$, $b_i^j$, and $c_{i-1}^j$ by $A_i^j$, $B_i^j$, and $C_{i-1}^j$, respectively. We will use an overline to denote the inverse of an element in $\map(S(n))$. Furthermore, we employ exponential notation to represent conjugation, defined as $X^Y = YX\overline{Y}$ (for example, $A^{h_{1,2}} = h_{1,2}A\overline{h_{1,2}}$).

\subsection{More than Two Ends}

For the rest of this subsection, we assume $n \geq 3$. We start by expressing the rotation $R$ as a product of two involutions. Call the $\pi$ radians rotation around the first end $\rho_1$ and call the $\pi$ radian rotation between the first and $n$-th ends $\rho_2$. Observe that 
\[
\pi(\rho_1) = (1)(2n)(3(n-1))\dots((\frac{n}{2}+1))
\]
for even $n$, and 
\[
\pi(\rho_1) = (1)(2n)(3(n-1))\dots((\frac{n+1}{2})(\frac{n+1}{2}+1))
\]
for odd $n$. Similarly,
\[
\pi(\rho_2) = (1n)(2(n-1))\dots((\frac{n}{2})(\frac{n}{2}+1))
\]
for even $n$, and 
\[
\pi(\rho_2) = (1n)(2(n-1))\dots((\frac{n+1}{2}))
\]
for odd $n$. It is clear that 
\[
\pi(\rho_1\rho_2) = (12\dots n)
\]
in both cases, and we expressed $R$ as a product of two involutions.

\begin{lemma}\label{lem:lemma1}
   The group generated by the set
   \[
   \set{\rho_1,\rho_2,A^1_1\overline{A_1^2},B^1_1\overline{B_1^2},C^1_0\overline{C_0^2},h_{1,2}}
   \]
   contains the Dehn twists $A^2_1\overline{A_2^2},B^2_1\overline{B_2^2}$ and $C^2_1\overline{C_2^2}$.
\end{lemma}
\begin{proof}
    Let $G$ be the set generated by the above set. The general strategy in this proof as well as the following ones is to abuse the power of the conjugation property and the braid relation of Dehn twists to obtain new elements in $G$.

   The images of the curves $a_1^1$ and $a^2_1$ under the rotation $\rho_1$ are ${a'}^1_1$ and ${a'}^n_1$ respectively, so
   \[
   {A'}^1_1\overline{{A'}^n_1} = (A^1_1\overline{A^2_1})^{\rho_1} \in G.
   \]
   By the same argument,
   \[
   A^2_1\overline{{A'}^n_1} = ({A'}^1_1\overline{{A'}^n_1})^{h_{1,2}} \in G,
   \]
   and
   \[
   A^2_2\overline{{A'}^n_1} = (A^2_1\overline{{A'}^n_1})^{h_{1,2}} \in G.
   \]
   Therefore,
   \[
   (A^2_1\overline{{A'}^n_1})({A'}^n_1\overline{A^2_2}) = A^2_1\overline{A^2_2} \in G.
   \]

   Since the images of the curves $b^1_1$ and $b^2_1$ under the handle shift $h_{1,2}$ are $b^2_1$ and $b^2_2$, respectively, hence
   \[
   B^2_1\overline{B^2_2} = (B^1_1\overline{B_1^2})^{h_{1,2}} \in G.
   \]

   Recall that the product $\rho_1\rho_2$ is the rotation $R$. The images of the curves $c_0^1$ and $c^2_0$ under $R$ are $c^2_0$ and $c^3_0$, respectively, which gives us
   \[
   C^2_0\overline{C^3_0} = (C^1_0\overline{C_0^2})^R \in G.
   \]
   From this,
   \[
   (C^1_0\overline{C_0^2})(C^2_0\overline{C^3_0}) = C^1_0\overline{C_0^3} \in G.
   \]
   It can be checked that
   \[
   C^2_1\overline{C^3_0} = (C^1_0\overline{C_0^3})^{h_{1,2}} \text{ if } n\not= 3
   \]
   \[
   C^2_1\overline{C^3_0} = (C^1_0\overline{C_0^3})^{h_{1,2}}(C^2_0\overline{C^3_0}) \text{ if } n= 3
   \]
   is in $G$.
   Finally,
   \[
   (C^1_0\overline{C_0^3})(C^3_0\overline{C^2_1}) = C^1_0\overline{C^2_1} \in G,
   \]
   and
   \[
   C^2_1\overline{C^2_2} = (C^1_0\overline{C^2_1})^{h_{1,2}} \in G,
   \]
   and we are done.
   
\end{proof}

We now demonstrate that the elements derived in the previous lemma are sufficient to obtain every Dehn twist. Since the compactly supported pure mapping class group is generated by these Dehn twists, the following results establish that our set constitutes a generating set for the entire group.

\begin{lemma}\label{lem:lemma2}
   The group generated by the set
   \[
    \set{\rho_1,\rho_2,A^2_1\overline{A_2^2},B^2_1\overline{B_2^2},C^2_1\overline{C_2^2},h_{1,2}}
   \]
   contains the Dehn twists $A^j_j,B^j_i$ and $C_{i-1}^j$ for all $i\geq 1$ and $j \in \set{1,2,\dots,n}$.
\end{lemma}
\begin{proof}
    We will work with curves on the second end throughout this proof, so we drop the upper index throughout this proof, for the sake of simplicity. 
    Let $G$ be the group generated by the above set. Since $h_{1,2}(a_1) = a_2$ and $h_{1,2}(a_2) = a_3$
    \[
    A_1\overline{A_3} = (A_1\overline{A_2})(A_1\overline{A_2})^{h_{1,2}} \in G.
    \]
    The curves $a_1$ and $b_1$ intersect once. By the braid relation, the product $(A_1\overline{A_2})(B_1\overline{B_2})$ takes the curves $a_1$ to the curve $b_1$ and fixes $a_3$ since it is disjoint from $a_1$,$a_2$,$b_1$ and $b_2$. By the conjugation property,
    \begin{align*}
    (A_1\overline{A_3})^{(A_1\overline{A_2})(B_1\overline{B_2})} &= [(A_1\overline{A_2})(B_1\overline{B_2})][A_1\overline{A_3}][(B_2\overline{B_1})(A_2\overline{A_1})] \\
    &= (\overline{A_2}\mkern3mu\overline{B_2})[(A_1B_1)(A_1\overline{A_3})(\overline{(A_1B_1)}](B_2A_2) \\
    &= (\overline{A_2}\mkern3mu\overline{B_2})(B_1\overline{A_3})(B_2A_2) \\
    &= B_1\overline{A_3} \in G.
    \end{align*}
    Applying the same argument to $B_1\overline{A_3}$ and $B_1\overline{B_2}C_1\overline{C_2}$, $C_1\overline{A_3}$ is in $G$. Using the elements we have obtained thus far, it can be checked that $B_2\overline{A_1},C_1\overline{A_1},C_2\overline{A_1}$ and $A_2\overline{C_2}$ are all in $G$.  
    The curves $a_1,a_3,c_1$ and $c_2$ bound an embedded lantern, shown in Figure~\ref{fig:lemmalantern}. The product $(B_2\overline{A_1})(C_1\overline{A_1})(A_1\overline{A_2})(C_1\overline{A_2})$ takes the curve $b_2$ to $d_1$ and fixes $a_1$. By the conjugation property, $D_1\overline{A_1}$ is in $G$. It is clear that $B_2\overline{B_3}$ and $B_2\overline{A_1}$ are in $G$, so $B_3\overline{A_1}$ is also in $G$. The product $(B_3\overline{A_1})(C_2\overline{A_1})(A_3\overline{A_1})(B_3\overline{A_1})$ takes $d_1$ to $d_2$ and fixes $a_1$, so $D_2\overline{A_1}$ is in $G$. It follows that
    \[
    D_2\overline{C_1} = (D_2\overline{A_1})(A_1\overline{C_1}) \in G.
    \]
    By the lantern relation, 
    \[
    A_1C_1C_2A_3 = A_2D_1D_2,
    \]
    and hence,
    \[
    A_3 = (A_2\overline{C_2})(D_1\overline{A_1})(D_2\overline{C_1}) \in G.
    \]
    Conjugating $A_3$ with powers of $h_{1,2}$ and $R$, gives us that $A^j_i$ is in $G$ for all $i\geq1$ and $j \in \set{1,2,\dots,n}$. It immediately follows that
    \[
    C_1 = (C_1\overline{A_1})A_1
    \]
    and
    \[
    B_1 = (B_1\overline{A_3})A_3
    \]
    are in $G$. Hence, $B_i^j$ and $C_i^j$ are in $G$ for all $i\geq1$ and $j \in \set{1,2,\dots,n}$. Since $\overline{h_{1,2}}(c_1^2) = c^1_0$,
    \[
    C^1_0 = \overline{h_{1,2}}C_1^2h_{1,2} \in G,
    \]
    therefore $C^j_0$ is in $G$ for all $j \in \set{1,2,\dots,n}$, and we are done.

\begin{figure}[htbp]
      \centering
      \includegraphics[width=0.8\textwidth]{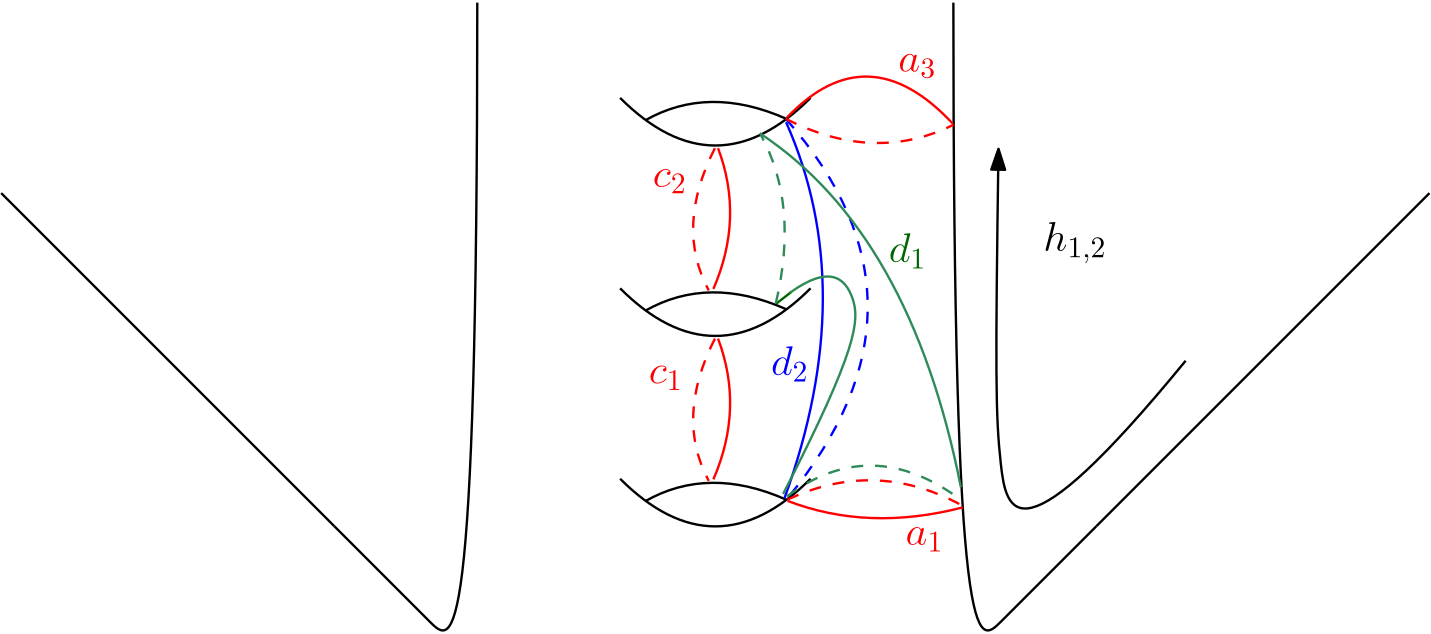}
      \caption{The embedded lantern bounded by the curves $a_1,a_3,c_1$ and $c_2$}
      \label{fig:lemmalantern}
   \end{figure} 
    
\end{proof}

By combining Lemma~\ref{lem:lemma1}, Lemma~\ref{lem:lemma2} and Proposition~\ref{prop:countableset}, we get the following theorem, which gives us a new topological generating set to work with.

\begin{theorem}\label{thm:mainlemma}
    The group generated by the set
    \[
   \set{\rho_1,\rho_2,A^1_1\overline{A_1^2},B^1_1\overline{B_1^2},C^1_0\overline{C_0^2},h_{1,2}}
    \]
    contains $\pmap_\mathrm{c}(S(n))$.
\end{theorem}

We can use this set to begin constructing our involutions.

\begin{lemma}\label{lem:lemma3}
    The group generated by the set
    \[
   \set{\rho_1,\rho_2,B^1_1C^1_0\overline{C_0^2}\mkern3mu\overline{B^3_1},A^1_1\overline{{A'}_1^2},h_{1,2}}
    \]
    contains $\pmap_\mathrm{c}(S(n))$.
\end{lemma}
\begin{proof}
    Let $G$ be the group generated by the above set and let $F_1$ denote $B^1_1C^1_0\overline{C_0^2}\mkern3mu\overline{B^3_1}$ and $L_1$ denote $A^1_1\overline{{A'}_1^2}$. As before, $R$ is in $G$ since it is the product $\rho_1\rho_2$. We shall change the notation once again by letting the lower index denote the end on which a twist or curve is, and always assume the position of the curves to be the first one on that end. That is, $A^i_1,{A'}^i_1,B_1^i$ and $C_0^i$ will be denoted by $A_i$,${A'}_i,B_i$ and $C_i$ respectively.

    Since the curve $a_1$ intersects $b_1$ once, the braid relation gives that the product $L_1F_1 = (A_1\overline{{A'}_2})(B_1C_1\overline{C_2}\mkern3mu\overline{B_3})$ sends the curve $a_1$ to $b_1$ while keeping ${a'}_2$ fixed. This follows from the fact that $a_1$ is disjoint from the supports of all the twists in the product except $B_1$, while ${a'}_2$ is disjoint from all of them. By the conjugation property,
    \begin{align*}
    L_2 = L_1^{L_1F_1} &= (L_1F_1)L_1(\overline{F_1}\mkern3mu \overline{L_1})\\
    &= [(A_1\overline{{A'}_2})(B_1C_1\overline{C_2}\mkern3mu\overline{B_3})](A_1\overline{{A'}_2})[(B_3C_2\overline{C_1}\mkern3mu\overline{B_1})({A'}_2\overline{A_1})]\\
    &= B_1\overline{{A'}_2} \in G.
    \end{align*}
    The element
    \[
    L_3 = \overline{L_1^{R}}\overline{F_1} = ({A'}_3\overline{A_2})(B_3C_2\overline{C_1}\mkern3mu\overline{B_1})
    \]
    is in $G$ and takes ${a'}_3$ to $b_3$ while fixing $a_2$. This implies that
    \[
    L_4 = \overline{L_1^{R}}^{\overline{L_1^{R}}L_3} = B_3\overline{A_2} \in G.
    \]
    It is easy to see that $F_1$ takes $b_1$ to $c_2$ and fixes ${a'}_2$. It follows that 
    \[
    L_5 = L_2^{F_1} =(B_1C_1\overline{C_2}\mkern3mu\overline{B_3})B_1\overline{{A'}_2}(B_3C_2\overline{C_1}\mkern3mu\overline{B_1}) = C_1\overline{{A'}_2} \in G.
    \]
    Moreover, 
    \[
    L_6 = L_4^{\overline{R}}L_1 = (B_2\overline{A_1})((A_1\overline{{A'}_2}) = B_2\overline{{A'}_2} \in G,
    \]
    therefore
    \[
    B_1\overline{B_2} = L_2\overline{L_6} = (B_1\overline{{A'}_2})({A'}_2\overline{B_2}) \in G.
    \]
    We also have that
    \[
    L_7 = \overline{L_2}(B_1\overline{B_2})L_2^R = ({A'}_2\overline{B_1})(B_1\overline{B_2})(B_2\overline{{A'}_3}) ={A'}_2\overline{{A'}_3} \in G,
    \]
    and
    \[
    C_1\overline{C_2} = L_5L_7\overline{L_5^R} = (C_1\overline{{A'}_2})({A'}_2\overline{{A'}_3})({A'}_3\overline{C_2}) \in G.
    \]
    Finally,
    \[
    A_1\overline{A_2} = \overline{L_4^{\overline{R}}}(B_1\overline{B_2})^RL_4 = (A_1\overline{B_2})(B_2\overline{B_3})(B_3\overline{A_2}) \in G.
    \]
    Since $G$ contains $A_1^1\overline{A^2_1}, B^1_1\overline{B_1^2}$ and $C_0^1\overline{C_0^2}$, by Theorem~\ref{thm:mainlemma}, it contains $\pmap_\mathrm{c}(S(n))$.
    
\end{proof}

To finish constructing the generating set consisting of involutions, we need to show that the handle shift $h_{1,2}$ can be written as a product of involutions.

\begin{theorem}\label{thm:hshiftinvo}
    The handle shift $h_{1,2}$ can be written as a product of two involutions.
\end{theorem}
\begin{proof}
    We construct a new model for $S(n)$ as follows. Start with a model of $S(2)$, in which $n-2$ of the genera lie along on a common line and the remaining genera are arranged perpendicularly to it. Remove $n-2$ closed disks $D_i$ from the portion of $S(2)$ containing the $n-2$ aligned genera such that their centers are on the aligning line, and each $\p D_i$ is looking towards the next one along the line. Attach $n-2$ surfaces, each having one end accumulated by genus and one boundary component, by gluing each surface along $\p D_i$. This new surface has $n$ ends accumulated by genus, and no boundary components, so it is homeomorphic to $S(n)$. This model is shown in Figure~\ref{fig:tau1}.

\begin{figure}[htbp]
      \centering
      \includegraphics[width=0.8\textwidth]{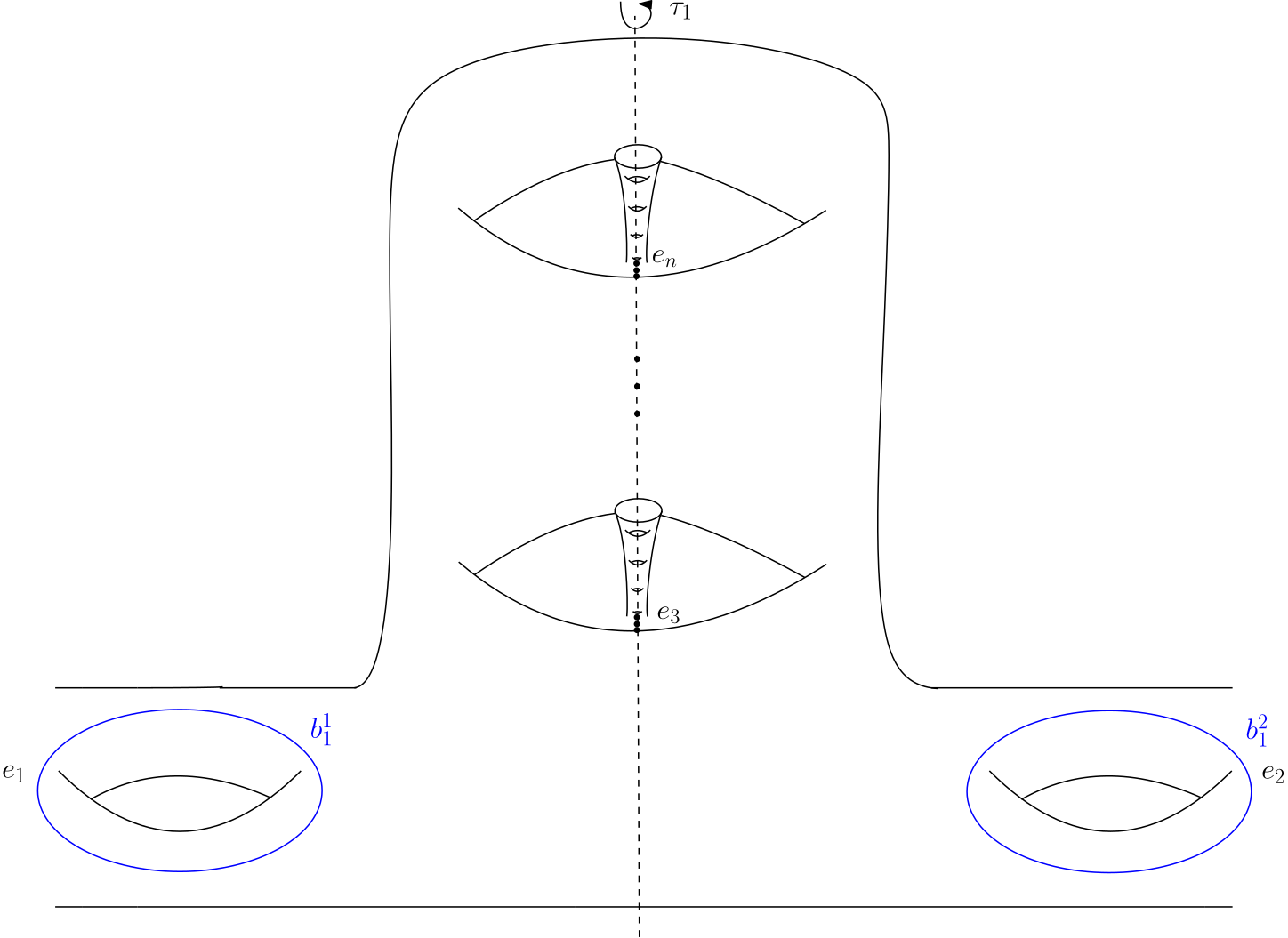}
      \caption{The new model for $S(n)$, and the rotation $\tau_1$}
      \label{fig:tau1}
   \end{figure}   

   Let $\tau_1$ be the rotation of $\pi$ radians about the line along which the $n-2$ genera are aligned. Because the rotation axis passes through the center of each the ends corresponding to the surfaces attached along the boundaries $\p D_i$, these ends are fixed by $\tau_1$. The remaining two ends, lying perpendicular to this axis, are interchanged by rotation, so we have $\pi(\tau_1) = (12)$. Since it is a rotation by $\pi$ radians, $\tau_1$ is an involution.

   The construction for the new model can be slightly modified to obtain yet another model, where $n-3$ genera lie along a common line and the remaining genera are perpendicular to the $(n-3)$-th genus. This yields the model shown in Figure~\ref{fig:tau2}.

   Let $\tau_2$ be the rotation of $\pi$ radians about the line along which the $n-3$ genera are aligned. Because the rotation axis passes through the center of each the ends corresponding to the surfaces attached along the boundaries $\p D_i$, these ends are fixed by $\tau_2$. The remaining two ends, lying perpendicular to this axis, are interchanged by rotation, so we have $\pi(\tau_2) = (12)$. Since it is a rotation by $\pi$ radians, $\tau_2$ is an involution.

\begin{figure}[htbp]
      \centering
      \includegraphics[width=0.8\textwidth]{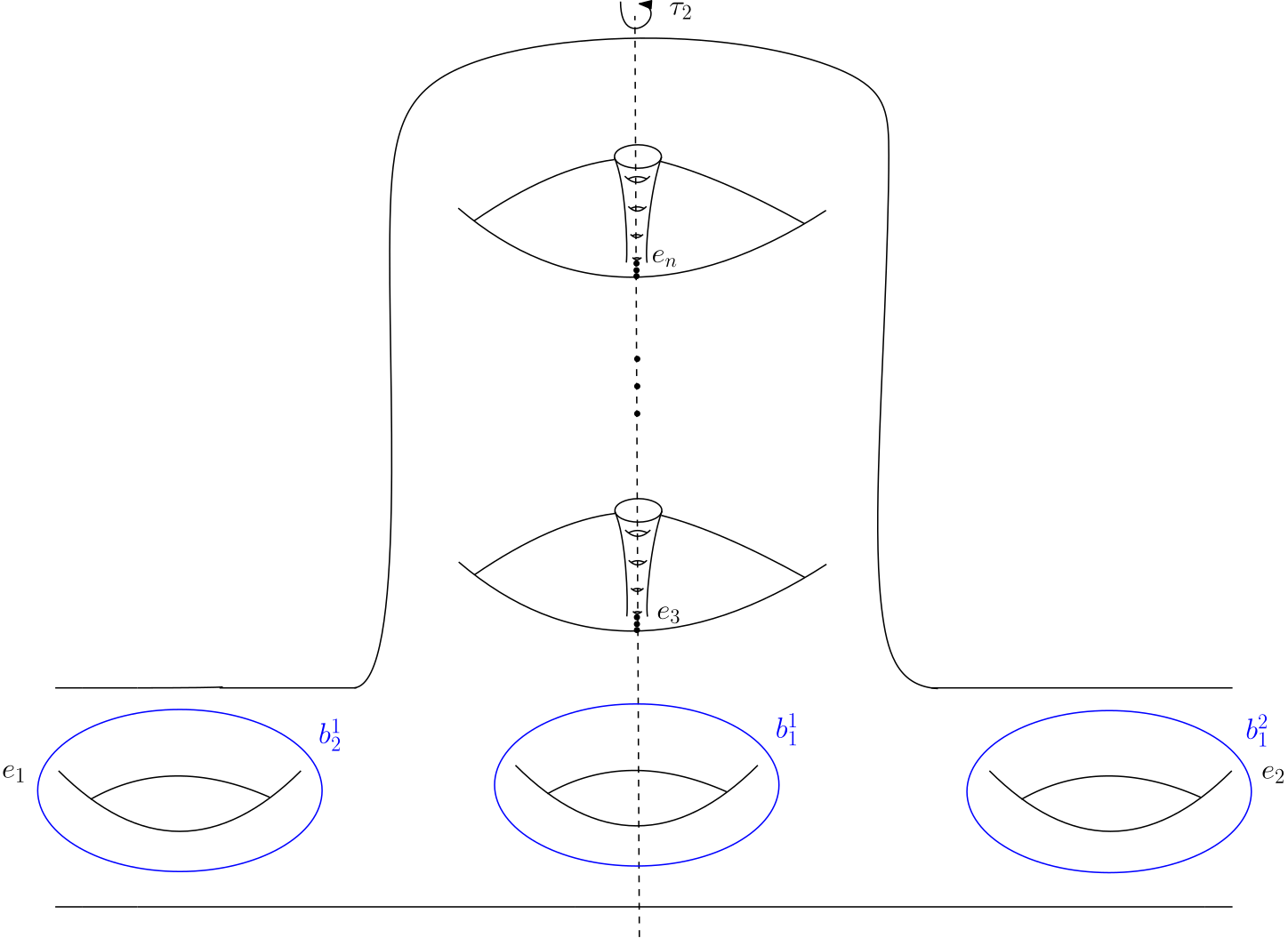}
      \caption{The new model for $S(n)$, and the rotation $\tau_2$}
      \label{fig:tau2}
   \end{figure}   
   Consider the composition $h_{1,2}\overline{(\tau_1\tau_2)}$. It can easily be checked that it fixes every isotopy class of simple closed curves. By Theorem~\ref{thm:infalex}, $h_{1,2}\overline{(\tau_1\tau_2)}$ is isotopic to the identity and $\tau_1\tau_2 = h_{1,2}$, and we are done.
    
\end{proof}

\begin{remark}
    While $\tau_1$ and $\tau_2$ fix the $n-2$ ends pointwise, they do not act as the identity on any neighborhood of those ends, since every neighborhood undergoes rotation as well.
\end{remark}

Combining Lemma~\ref{lem:lemma3} and Theorem~\ref{thm:hshiftinvo}, we immediately obtain the following:

\begin{theorem}\label{thm:6invo}
    The mapping class group $\map(S(n))$ is topologically generated by six involutions.
\end{theorem}
\begin{proof}
    Let $G$ be the subgroup of $\map(S(n))$ topologically generated by the set
    \[\set{\rho_1,\rho_2,\rho_3B^1_1C^1_0\overline{C_0^2}\mkern3mu\overline{B^3_1},\rho_4A^1_1\overline{{A'}_1^2},\tau_1,\tau_2},
    \]
    where $\rho_3 = R\rho_1\overline{R}$ and $\rho_4 = R\rho_2\overline{R}$. Since $\rho_1$ and $\rho_2$ are in $G$, $R$, $\rho_3$ and $\rho_4$ are in $G$, which implies that $B^1_1C^1_0\overline{C_0^2}\mkern3mu\overline{B^3_1}$ and $A^1_1\overline{{A'}_1^2}$ are in $G$. By Theorem~\ref{thm:hshiftinvo}, $h_{1,2}=\tau_1\tau_2$ is also in $G$. By Lemma~\ref{lem:lemma3}, $G$ contains $\pmap_{c}(S(n))$. \\
    Recall that the action of $R$ on $\sym_n$ is an $n$-cycle and the action of $\tau_1$ is a $2$-cycle, which by Theorem~\ref{lem:symgen} implies that $G$ contains mapping classes whose action generate on the space of ends generate $\sym_n$. It follows that $G=\map(S(n))$. \\
    We only need to show that $\rho_3B^1_1C^1_0\overline{C_0^2}\mkern3mu\overline{B^3_1}$ and $\rho_4A^1_1\overline{{A'}_1^2}$ are involutions.
    Observe that the image of the curves $b^1_1$ and $c^1_0$ under $\rho_3$ are the curves $b^3_1$ and $c^2_0$, by the conjugation property,
    \[
    \rho_3(B^1_1C^1_0\overline{C_0^2}\mkern3mu\overline{B^3_1})\rho_3 = B^3_1C^2_0\overline{C_0^1}\mkern3mu\overline{B^1_1},
    \]
    so
    \[
    (\rho_3B^1_1C^1_0\overline{C_0^2}\mkern3mu\overline{B^3_1})^2 = (\rho_3B^1_1C^1_0\overline{C_0^2}\mkern3mu\overline{B^3_1}\rho_3)B^1_1C^1_0\overline{C_0^2}\mkern3mu\overline{B^3_1} = B^1_1C^1_0\overline{C_0^2}\mkern3mu\overline{B^3_1}B^3_1C^2_0\overline{C_0^1}\mkern3mu\overline{B^1_1} = \mathds{1}_{\map(S(n))}.
    \]
    Applying the same argument for $A^1_1\overline{{A'}_1^2}$ and $\rho_4$, we see that $$\rho_4A^1_1\overline{{A'}_1^2}$$ is also an involution, and we are done.
\end{proof}

Using similar methods, Altunöz, Pamuk, and Yıldız~\cite{apy} demonstrated that when the number of ends is at least six, the number of involution generators can be reduced to five.  The existence of additional ends allows for a single generator comprising twists $A_i,B_i$ and $C_i$ rather than requiring two separate generators for Dehn twists. This approach hinges on the fact that a twist can be interchanged with another when their corresponding curves intersect exactly once and are disjoint from all others.  For $n<6$ however, this reduction is impossible because the elements $A_1\overline{A_2},B_1\overline{B_2}$ and $C_1\overline{C_2}$ cannot be isolated by interchanging single twists.

\bibliographystyle{abbrv}
\bibliography{references}

\end{document}